\documentclass[12pt]{amsart}
\usepackage[latin1]{inputenc}
\usepackage{amsfonts}
\usepackage[toc,page,title,titletoc,header]{appendix}
\usepackage{graphicx,psfrag,epsfig,multirow,caption,subcaption}
\usepackage{amssymb,amsmath,amscd,amsthm,amssymb,verbatim,setspace}
\numberwithin{equation}{section}
\usepackage{mathrsfs,wrapfig}
\usepackage{indentfirst}
\usepackage{extarrows}

\usepackage[colorlinks=true]{hyperref}

\newtheorem{theorem}{Theorem}[section]
\newtheorem{definition}[theorem]{Definition}

\newtheorem{proposition}[theorem]{Proposition}
\newtheorem{lemma}[theorem]{Lemma}
\newtheorem{remark}[theorem]{Remark}
\newtheorem{corollary}[theorem]{Corollary}

\newtheorem*{remark*}{Remark}

\numberwithin{equation}{section}

\newcommand{\eps}{\varepsilon}

\newcommand{\dt}{\delta}
\newcommand{\al}{\alpha}

\newcommand{\bn}{\mathbf{n}}
\newcommand{\br}{\mathbf{r}}

\newcommand{\cN}{\mathcal{N}}
\newcommand{\cP}{\mathcal{P}}

\newcommand{\cH}{\mathcal{H}}

\newcommand{\cS}{\mathcal{S}}
\newcommand{\cC}{\mathcal{C}}
\newcommand{\cL}{\mathcal{L}}
\newcommand{\cK}{\mathcal{K}}
\newcommand{\cT}{\mathcal{T}}
\newcommand{\cR}{\mathcal{R}}
\newcommand{\cF}{\mathcal{F}}

\newcommand{\R}{\mathbb{R}}

\DeclareMathOperator{\Ric}{Ric}
\DeclareMathOperator{\Hess}{Hess}

\DeclareMathOperator{\diam}{diam}
\DeclareMathOperator{\loc}{loc}

\title[Global perturbation of MCF II]{Generic Dynamics of Mean Curvature Flows with Asymptotically Conical Singularities}
\author{Ao Sun, Jinxin Xue}
\address{Lehigh University, Department of Mathematics, Bethlehem, PA 18015}
\email{aos223@lehigh.edu}
\address{New Cornerstone Science Laboratory, Department of Mathematical Sciences,  Tsinghua University, Beijing, 100084}
\email{jxue@tsinghua.edu.cn}
\date{\today}

\begin{document}
	\begin{abstract}This is the second paper in the series to study the generic dynamics of mean curvature flows. We study the initial perturbation of mean curvature flows, whose first singularity is modeled by an asymptotically conical shrinker. The noncompactness of the limiting shrinker creates essential difficulties. We introduce the Feynman-Kac formula to get precise asymptotic behaviour of the linearized rescaled mean curvature equation along an orbit. We also develop the invariant cone method for the noncompact setting for the local dynamics near the shrinker. As a consequence, we prove that after a generic initial perturbation, the perturbed rescaled mean curvature flow avoids the conical singularity.
	\end{abstract}
	\maketitle

	\section{Introduction}
	In this paper, we extend the idea in our previous paper \cite{SX} to study the initial perturbation of mean curvature flows whose first singularity is modeled by an asymptotically conical self-shrinker. 
	
	A \emph{mean curvature flow} (MCF) is a family of closed embedded hypersurfaces $\{\mathbf M_t\}$ in $\R^{n+1}$ satisfying the equation
	$
	\partial_t x=\vec{H}.
	$
	Here $x$ is the position vector, $\vec{H}$ is the mean curvature vector, which is the trace of the second fundamental forms. It is known that an MCF always develops singularities within a finite time, so the analysis of the singularity blowup becomes a central topic when we study MCFs. After a spacetime rescaling, an MCF can be turned to a \emph{rescaled mean curvature flow} (RMCF), satisfying the equation
	\begin{equation}\label{EqRMCF}
		\partial_t x=\vec{H}+\frac{x^\perp}{2},
	\end{equation} where 
	$\vec{H}+\frac{x^\perp}{2}$ is called the rescaled mean curvature vector. Here $x^\perp$ is the projection of the position vector to the normal bundle. The MCF and its corresponding RMCF are related as follows: if $\{\mathbf{M}_\tau\}_{\tau\in[-1,0)}$ is an MCF, then 
	\begin{equation}\label{EqMCFRMCF}
		M_t=\mathrm{e}^{t/2}\mathbf{M}_{-\mathrm{e}^{-t}}\end{equation} 
	is its corresponding RMCF zooming at the spacetime origin, while $t\in[0,\infty)$.
	A hypersurface that is static under the RMCF is called a \emph{self-shrinker}, which satisfies the equation $\vec{H}+\frac{x^\perp}{2}=0$. The RMCF was first introduced by Huisken in \cite{H} to study the singularity of MCFs. We say a singularity of the MCF is modeled by a self-shrinker $\Sigma$, if the corresponding RMCF converges to $\Sigma$ smoothly as time goes to infinity. If $\Sigma$ is noncompact, then the convergence is in the sense of $C_{\loc}^\infty$, i.e., on any compact subset the convergence is smooth.

	Colding-Minicozzi introduced the ideas from dynamical systems to study MCFs (c.f. \cite{CM1,CM2,CM3,CM4} etc.). Their dynamical approach views the RMCF \eqref{EqRMCF} as the negative gradient flow of the $F$-functional
	$$\cF(M_t)=\frac{1}{(4\pi)^{n/2}}\int_{M_t}\mathrm{e}^{-\frac{|x|^2}{4}}d\mu $$
	and shrinkers as the critical points of $F$.  So it is natural to anticipate that generic RMCF avoids those shrinkers that are saddles in the second variation, modulo translations and dilations.  In \cite{CMP}, Colding-Minicozzi-Pedersen proposed the conjecture that one can perturb the initial data of an MCF so that the MCF will only encounter singularities modeled by generic self-shrinkers (c.f. \cite[Conjecture 8.2]{CMP}).
	Since Euclidean rigid motion does not change the MCF essentially, in \cite{CM1} Colding-Minicozzi also introduced the entropy  $$\lambda(M)=\sup_{x\in\R^{n+1},t\in(0,\infty)}\cF(t^{-1}(M-x))$$ modulo the translations and dilations. 
	In \cite{SX}, we made progress to Colding-Minicozzi's dynamical program. We studied the initial perturbation of an MCF whose first singularity is modeled by a closed embedded self-shrinker. 
	
	In this paper, we make further progress to Colding-Minicozzi's program. We study the initial perturbation of MCFs whose first singularity is unique and modeled by an asymptotic conical self-shrinker. A self-shrinker $\Sigma$ is called \emph{asymptotically conical} if it converges to a cone after blowing down. More precisely, $\tau^{-1}\Sigma$ converges to a cone $\Gamma$ smoothly on any compact subset of $\R^{n+1}\backslash\{0\}$ as $\tau\to\infty$. We make the following standing assumption throughout the paper if not otherwise mentioned. 
	
	($\star$) {\it Let $(\mathbf M_\tau)_{\tau\in[-1,0)}$ be an MCF with a unique first-time singularity at the spacetime point $(0,0)$, and let $(M_t)_{t\in[0,\infty)}$ be the corresponding RMCF with $M_t\to\Sigma$ in the $C^\infty_{\loc}$ sense as $t\to\infty$, where $\Sigma$ is an asymptotically conical self-shrinker.}
	
	We remark that this assumption is natural: by the resolution of the Multiplicity One Conjecture by Bamler-Kleiner \cite{BK}, if the first-time singularity of an MCF of closed embedded surfaces in $\R^3$ is modeled by an asymptotically conical shrinker, then the corresponding RMCF indeed converges to the asymptotically conical shrinker in the $C_{\text{loc}}^\infty$ sense; recently Tang-Kai Lee and Xinrui Zhao \cite{LeZh} constructed examples of MCFs of closed hypersurfaces explicitly in $\R^{n+1}$ whose first-time singularity is modeled by an asymptotically conical shrinker.

	The main theorem tackles a given singularity. While the perturbation does modify the shape of the flow near the given singularity, it does not modify the region far away from the singularity. In fact, this is the consequence of the pseudolocality of MCFs. As entropy characterizes all the spacetime scales, it is not sensitive enough to capture the modification of a given singularity. In \cite{Su}, the first author introduced a finer quantity called \emph{local entropy}. More precisely, suppose $U\times I\subset\R^{n+1}\times(0,\infty)$ is a spacetime region,  define the local entropy to be
	\[
	\lambda^I_U(M)=\sup_{x\in U,t\in I}\cF(t^{-1}(M-x)).
	\]
	The benefit of introducing local entropy is that it only detects the scales that we are interested in. We refer the reader to \cite{Su} for further discussions.
	
	\begin{theorem}\label{thm:MainThmLocal}
		Assume $(\star)$. Then there exist $\dt_0>0$ and   an open dense subset $\cS$ of $\{u\in C^{2,\alpha}(M_0)\ |\ \|u\|_{C^{2,\alpha}}=1\}$, such that for any $0<\dt<\dt_0$ and any $u_0\in \mathcal S$, there exists $\eps_0:=\eps_0(u_0)$, such that for all $0<\eps<\eps_0$, there exists $T>0$, such that the RMCF $\{\widetilde{M}_t\}$ starting from $\widetilde{M}_0:=\{x+\eps u_0(x)\mathbf n(x)\ |\ x\in M_0\}$ satisfies $$\lambda^{(1-\delta,1+\delta)}_{B_\delta}(\widetilde{M}_T)<\lambda(\Sigma).$$
	\end{theorem}
	
	More quantitatively, we have the following theorem. 
	\begin{theorem}\label{thm:main1}
		
		Assume $(\star)$. Then there exist $\dt_0>0$ and   an open dense subset $\cS$ of $\{u\in C^{2,\alpha}(M_0)\ |\ \|u\|_{C^{2,\alpha}}=1\}$, such that for any $0<\dt<\dt_0$ and any $u_0\in \mathcal S$, there exists $\eps_0:=\eps_0(u_0)$, such that for all $0<\eps<\eps_0$, there exists $T>0$, such that the RMCF $\{\widetilde{M}_t\}$ starting from $\widetilde{M}_0:=\{x+\eps u_0(x)\mathbf n(x)\ |\ x\in M_0\}$ satisfies $$\cF( \widetilde{M}_T)<\lambda(\Sigma)-\dt^{2.5}.$$
		Moreover, there exists $R=R(\eps)\to\infty$ as $\eps\to 0,$ such that \begin{enumerate}
			\item $\cF(\widetilde{M}_T\setminus B_R)<\dt^3$,
			\item  $\widetilde{M}_T\cap B_R$ can be written as the graph of a function $\bar u(T):\ \Sigma\cap B_R\to \R$ with $\|\bar u(T)\|_{C^{2,\al}}=\dt$,
			\item $\cF(\mathcal R \widetilde{M}_T)<\lambda(\Sigma)-\dt^{2.5}$ for any translation and dilation $\cR$ of scale $\dt$.
		\end{enumerate}

	\end{theorem}
	
	{In Theorem \ref{thm:main1}, we first consider the part of the perturbed RMCF inside a large ball $B_R$. Item (1) says that the part outside this ball is negligible; item (2) says that inside the ball $B_R$, the perturbed RMCF can be written as a graph; item (3) says that the perturbed RMCF has strictly smaller entropy than the shrinker $\Sigma$.}
	By Huisken's monotonicity formula, we immediately obtain the following corollary.
	
	\begin{corollary}\label{CorMain}
		In the setting of Theorem \ref{thm:main1}, there exists a spacetime neighbourhood of $(0,0)$ with size $\dt$, such that the perturbed MCF $\{\widetilde{\mathbf{M}}_t\}$ starting from $\widetilde{M}_0:=\{x+\eps u_0(x)\mathbf n(x)\ |\ x\in M_0\}$ has no singularity modeled by $\Sigma$ in this neighbourhood.
	\end{corollary}

	In other words, after an initial generic perturbation, any singularity (if there is one) in a spacetime neighbourhood of the original singularity can not be modeled by the same asymptotically conical shrinker $\Sigma$. Moreover, the neighborhood is larger than that created by translations and dilations of the same order of magnitude of the initial perturbation.
	
	In \cite{SX}, we study the problem of initial perturbations for MCFs with compact singularities and obtained a stronger conclusion that after a generic initial perturbation, the perturbed MCF will never generate a singularity modeled by the original limit closed shrinker. Here we can only avoid a conical shrinker in a spacetime neighbourhood. The difference between these two results illustrates the nature of noncompactness. A similar issue appears in the study of higher multiplicity singularities (see \cite{Su}).

	As an application, we can study the behaviour of the perturbed RMCF near the original singularity. Our first application shows that after appropriate rescalings, the perturbed MCF will converge to an ancient solution.

	\begin{theorem}\label{thm:ancient} Assume $(\star)$. 
		Suppose $u_0$ is a generic smooth function on $M_0$ with the unit $C^{2,\al}$ norm and $\epsilon_i\to 0$ is a sequence of positive numbers. Suppose $\{\widetilde{M}^i_t\}$ is the RMCF starting from $M_0+\epsilon_i u_0\bn$. Then there exists a sequence $\{T_i\}_{i=1}^\infty$, $T_i\to\infty$ as $i\to\infty$, such that $\{\widetilde M_{t+T_i}\}$ smoothly converges to an ancient RMCF $\{N_t\}_{t\in(-\infty,0)}$ on any compact spacetime subset, and $N_t$ is not the static flow $\Sigma$.
	\end{theorem}
	
	In \cite{SX}, we have proved similar existence results of ancient solutions when the limit shrinker is compact. In \cite{CCMS1}, Chodosh-Choi-Mantoulidis-Schulze proved the existence of ancient solutions coming out from an asymptotically conical shrinker with dimension assumption $n\leq 6$ due to the requirement in geometric measure theory.  \cite{CCMS1} studied the problem of initial perturbation using the geometric measure theory method, while here we  purely use PDEs and the dynamical system method. Our approach has the disadvantage of being unable to handle multiple singularities, while it also enjoys some merits such as being free of dimension or low entropy assumptions and allowing generic perturbations, not necessarily only one-sided. 
	
	In the following, we discuss related backgrounds and give literature reviews. 
	\subsection{Generic MCFs}
	It is known that a closed MCF in $\R^{n+1}$ must generate finite-time singularities, and the singularities are modeled by shrinkers (c.f. \cite{H, I, Wh}). The shrinkers are minimal surfaces in the Gaussian metric space, and there are many constructions of shrinkers (see \cite{KKM}, \cite{Ngu} etc). It seems impossible to classify all embedded self-shrinkers even in $\R^3$. Therefore, it is very complicated to understand the singular behaviour of an MCF. The MCF has been proposed by Yau as a potential tool to approach some important problems in topology such as the Smale conjecture, the Schoenflies conjecture, etc.  However, the complicated singularity structure becomes a serious obstacle to the application of the MCF to topological problems. 
	
	Generic MCF is proposed to overcome this issue. The concept of generic MCF was first proposed by Huisken in \cite{H} (c.f. \cite{AIC} for similar ideas in the study of MCFs in $\R^3$). Colding-Minicozzi in \cite{CM1} formulated the notion of stability and classified the generic shrinkers that are spheres $\mathbb S^n$ and cylinders $\mathbb S^k\times \R^{n-k},\ k=0,1,\ldots,n-1$. We thus expect that a generic MCF avoids all singularities that are non-spherical and non-cylindrical. The idea of \cite{CM1} is to study the linearization of \eqref{EqRMCF} at a shrinker $\Sigma$. The linearized equation has the form $\partial_t u=L_\Sigma u$ where $L_\Sigma=\Delta_{\Sigma}-\frac{1}{2}\langle x,\nabla_\Sigma \cdot \rangle+(|A|^2+\frac12)$ is adjoint with respect to the Gaussian weighted inner product $\langle u,v\rangle=\int_{\Sigma} u(x)v(x)\mathrm{e}^{-\frac{|x|^2}{4}}d\mu.$ It is known that the mean curvature $H$ of $\Sigma$ is the eigenfunction of $L_\Sigma$ with eigenvalue 1, i.e., $L_\Sigma H=H$ (we remark that we use the different sign convention than \cite{CM1} for the definition of eigenvalues for the purpose of studying dynamics). Moreover, from elliptic operator theory, it is known that the leading eigenfunction does not change the sign. Shrinkers with positive $H$ are classified by Huisken and Colding-Minicozzi to be spheres and cylinders. So for a non-spherical and non-cylindrical shrinker, the leading eigenfunction $\phi_1$ cannot be $H$ and hence the leading eigenvalue $\lambda_1$ has to be larger than 1. The idea of \cite{CM1} is then to perturb $\Sigma$ in the direction of $\phi_1$, which can decrease the entropy strictly.  
	
	Note that the perturbations in \cite{CM1} are constructed on the shrinker $\Sigma$ and hence are local in nature. To avoid the shrinker by perturbing the initial condition, we have to control the perturbed RMCF all the way up to leaving a neighborhood of the shrinker. Our main theorem is a consequence of the following key estimate on the global dynamics, which shows that a generic initial perturbation realizes the local perturbation in the $\phi_1$ direction used by \cite{CM1}. 
	
	\begin{theorem}\label{ThmMain}
		In the setting of Theorem \ref{thm:main1}, we have 
		$$|\langle\bar u(T),\phi_1\rangle_{W^{1,2}}|\geq (1-o_\dt(1))\|\bar u(T)\|_{W^{1,2}}, $$ where $\phi_1$ is the first eigenfunction of the linearized operator on $\Sigma$ with the $L^2$-norm $1$. 
		
		In other words, $\widetilde M_t$ drifts to the most unstable direction on $\Sigma$. 
		
	\end{theorem}
	The proofs of the above theorems are given in Section \ref{S:Dynamics in a neighborhood of the shrinker}. 
	\subsection{Two dynamical problems}
	From our work in the compact case \cite{SX}, we see that one of the main ingredients is to study the asymptotic behavior of positive solutions to  the linearized RMCF equation (also called variational equation)
	\begin{equation}\label{EqMain}
		\partial_t u=L_{M_t} u,
	\end{equation}
	where $(M_t)$ is an RMCF converging to a limit shrinker $\Sigma$ as $t\to\infty$ in the $C^\infty_{\loc}$ sense. This equation can be considered as the Jacobi equation along an RMCF which governs how nearby orbits diverge. 
	Our dynamical approach to the problem of initial perturbations consists of the study of the following two dynamical problems:
	\begin{enumerate}
		\item The asymptotic behavior of the solution to \eqref{EqMain}. In particular, we want to find the initial condition $u(0):\ M_0\to \R$ such that $$\lim_{t\to\infty}\frac{1}{t}\log\|u(t)\|_{L^2(M_t)}=\lambda_1(\Sigma),\quad \frac{u(t)}{\|u(t)\|_{L^2(M_t)}}\to \phi_1.$$ One easily recognizes in this case that $\lambda_1(\Sigma)$ is the leading Lyapunov exponent. 
		\item The local (nonlinear) dynamics of the RMCF near the shrinker. 
	\end{enumerate}
	In \cite{SX}, we studied the case when $\Sigma$ is a compact shrinker. The first dynamics problem is addressed by a Harnack estimate given by the Li-Yau estimate. For the second dynamical problem, we write the RMCF equation as $\partial_t u=L_\Sigma u+Q(u,\nabla u,\nabla^2u)$ in a neighborhood of $\Sigma$ where each manifold corresponds to the graph of a function $u$.  This problem can be approached by the invariant manifold theory in hyperbolic dynamics, so the dynamics of the RMCF in a neighborhood of the shrinker is approximately that of the linear equation $\partial_t u=L_\Sigma u$. 
	
	When $\Sigma$ is noncompact, serious issues arise in both problems. For the first problem, we do not have a uniform Harnack estimate, and the Li-Yau estimate gets worse as time gets longer. In this paper, we introduce a Feynman-Kac representation of the solutions to \eqref{EqMain}, which enables us to prove the following theorem, and hence address the first problem (see Subsection \ref{SSFKCone}).
	\begin{theorem}\label{mainthm-FK}Let $\{M_t\}_{t\in[0,\infty)}$ be an RMCF with $M_t\to \Sigma$ as $t\to\infty$ in the $C^\infty_{\loc}$ sense, where $\Sigma$ is a shrinker that either is compact or satisfies \begin{enumerate}\item $\limsup_t\lambda_1(M_t)\leq \lambda_1(\Sigma);$
			\item there exists constant $D>0$ such that $\lambda_1(M_t)-\lambda_2(M_t)\geq D,$
		\end{enumerate}
		as $t\to\infty$ where $\lambda_1(M_t)$ $($resp. $\lambda_{1}(\Sigma))$ is the leading eigenvalue of $L_{M_t}$ $($resp. $L_\Sigma)$ and $\lambda_2(M_t)$ is the second. 
		Let $v^\star$ be the solution to the initial value problem \eqref{EqMain} with  initial condition $v^\star_0>0$. 
		
		Then we have
		\begin{enumerate}
			\item $\lim_{t\to\infty} \frac{1}{t}\log \|v^\star(t)\|_{L^2(t)}=\lambda_1(\Sigma)$, the leading eigenvalue of $L_\Sigma$;
			\item Let $\phi_1(t)$ be the first eigenfunction of $L_{M_t}$ on $M_t$ with the $L^2$-norm $1$. There exist constants $1>c>0,C>1$ and a sequence of times $t_i\to\infty$ such that $$\|v^\star\|_{Q(M_{t_i})}\leq C\|v^\star\|_{L^2(M_{t_i})},\ \mathrm{and}\ \frac{|\langle v^\star(t_i),\phi_1(t_i)\rangle_{L^2(M_{t_i})}|}{\|v^\star(t_i)\|_{Q(M_{t_i})}}>c.$$ 
		\end{enumerate}
	\end{theorem}
	Here the $L^2(M_t)$-norm is the Gaussian weighted $L^2$-norm for functions on $M_t$ and $Q(M_t)$ is a norm equivalent to the weighted $W^{1,2}$-norm on $M_t$. 
	\begin{figure}[h!]
		\centering
		\includegraphics[width=0.5\textwidth]{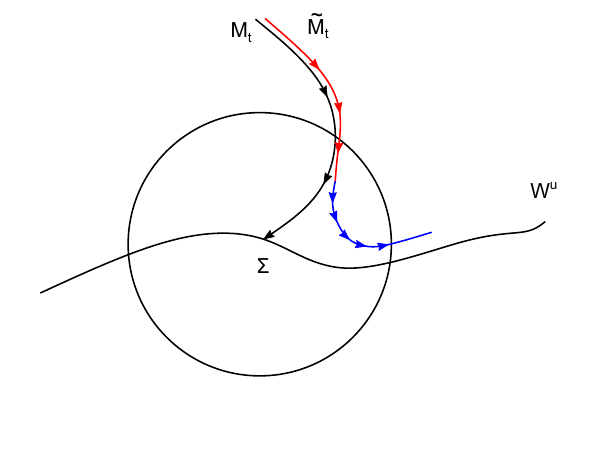}
		\caption{Dynamics of the perturbed RMCF $\widetilde{M}_t$}
		\label{Figure}
	\end{figure}
	
	We shall apply Theorem \ref{mainthm-FK} to control the perturbed RMCF $\widetilde M_t$ over a long time $T$ so that both $M_t$ and $\widetilde M_t$ are $\dt$-close to $\Sigma$ in the $C^{2,\al}$ norm over a large domain. See the red curve in Figure \ref{Figure}. Item (2) of Theorem \ref{mainthm-FK} gives that the difference of the two manifolds $\widetilde M_T$ and $M_T$ has a nontrivial projection to the $\phi_1$ direction. Here comes the second problem. We wish to approximate the local dynamics near $\Sigma$ by the linear equation $\partial_t u=L_\Sigma u$ and show that the $\phi_1$-component dominates all other Fourier modes when the perturbed flow leaves a $\dt$-neighborhood of $\Sigma$. See the blue curve in Figure \ref{Figure}. The main difficulty is that the flow $\{M_t\}$ cannot be written as a global graph over $\Sigma$ for any time $t$ so that the linear approximation of the local dynamics can only be done by restricting the flow to a compact domain, which makes the system nonautonomous, since the information outside the compact domain is discarded. Moreover, the time span for the blue curve in Figure \ref{Figure}, even though is only finite, is rather long depending on the smallness of the initial perturbation.  
	
	To overcome this difficulty, one important ingredient in the proof is the pseudo-locality property of MCFs. Pseudo-locality is a special property for the nonlinear geometric heat equation, first discovered by Perelman in the setting of Ricci flows, then by Ecker-Huisken in \cite{EH}, and later studied by \cite{INS} in the setting of MCFs. Roughly speaking, pseudo-locality says that if the MCF is graphical in a small neighbourhood, then it keeps being graphical for a short time. Using the correspondence \eqref{EqMCFRMCF} between the MCF and the 
	RMCF we get that the region close to the singularity in the MCF will be expanded to infinity at an exponential rate in the RMCF. Thus pseudolocality enables us to get control of the dynamics of RMCF over an exponentially growing domain where we can approximate the RMCF using the linear equation $\partial_t u=L_\Sigma u$ over a sufficiently long time.  This part will be elaborated in Section \ref{S:Dynamics in a neighborhood of the shrinker}, where we prove Theorem \ref{ThmMain}.

	\subsection{The Feynman-Kac formula}
	
	When the limiting shrinker $\Sigma$ is compact, the RMCF $M_t$ will be very close to $\Sigma$ when $t$ is sufficiently large. In particular, the geometry of $M_t$ will be uniformly close to $\Sigma$, and we can identify the function space of $M_t$ with the function space of $\Sigma$. In \cite{SX}, this fact is crucial, and it allows us to use a Li-Yau type Harnack inequality, to show that the positive solutions to \eqref{EqMain} satisfy the estimates in the conclusion (2) of Theorem \ref{mainthm-FK}.
	
	In the setting of noncompact shrinkers, the Li-Yau estimate does not meet our purpose. The new tool we introduce to prove Theorem \ref{mainthm-FK} is a Feynman-Kac formula in the setting of RMCFs. The Feynman-Kac formula views the equation \eqref{EqMain} from a dynamical and probabilistic perspective.  Indeed, if we consider the dynamical system $\partial_t u=\mathcal L_\Sigma u:=(\Delta_\Sigma-\frac12\langle x,\nabla_\Sigma\cdot \rangle)u$ on a shrinker $\Sigma$, the fact (following from the adjointness of $\mathcal L_\Sigma$)$$\frac{d}{dt}\int_{\Sigma}u(t)\mathrm{e}^{-\frac{|x|^2}{4}}d\mu=\int_{\Sigma}\mathcal L_\Sigma u(t)\mathrm{e}^{-\frac{|x|^2}{4}}d\mu=-\int_{\Sigma}\nabla 1\cdot \nabla u(t)\mathrm{e}^{-\frac{|x|^2}{4}}d\mu=0$$
	means that the dynamical system has $\mathrm{e}^{-\frac{|x|^2}{4}}d\mu$ as the invariant measure. The situation is then rather similar to the well-known Ornstein-Uhlenbeck process in $\R^n$. The stochastic differential equation $dX=-Xdt+dW$, where $W$ is the Brownian motion, has the Ornstein-Uhlenbeck operator $\mathcal L=\Delta-\langle x,\nabla\cdot \rangle$ as the generator and Gaussian as the invariant measure. When there is a potential $V$ added to the Ornstein-Uhlenbeck operator, there is no underlying stochastic differential equation, instead, the Feynman-Kac formula gives a representation of  the solution to a linear heat equation in the presence of a potential term, for example, the equation $\partial_t u=\Delta u+Vu$ on $\R^n$, with the initial condition $u(x,0)=f(x)$ as
	\[
	u(x,t)=\int_\Omega f(\omega(0))\exp\left(\int_0^t V(\omega(s))ds\right)d\nu_{x,t}(\omega),
	\]
	where  $\nu_{x,t}$ is a  probability measure on the ``space of all paths" $\Omega$ ending at $x$ at the time $t$. We establish a Feynman-Kac type representation of solutions to \eqref{EqMain} in Section \ref{S:FK}. 
	The path integral feature is useful for localizing the linearized equation \eqref{EqMain} to a neighborhood of the shrinker and to a bounded domain, and enables us to establish the correct exponential growth of the solution in Theorem \ref{mainthm-FK}. Details are presented in Section \ref{S:Asymvariation}.

	\subsection{Asymptotically conical shrinkers}
	
	Our analysis in this paper in principle should apply to general shrinkers and even singularities of other flows. However, we choose to study the singularity modeled by asymptotically conical shrinkers for the sake of concreteness and simplicity. 
	
	Firstly, asymptotically conical shrinkers form an important class of shrinkers. In fact, it is not known whether other types of singularities really exist. Particularly in $\R^3$, L. Wang \cite{Wa2} shows that all noncompact shrinkers are only those with finitely many cylindrical or conical ends, and the asymptotics are smooth. Ilmanen \cite{I} conjectured that in $\R^3$, any asymptotic cylindrical shrinker is actually a standard cylinder. If this conjecture is true, then the shrinkers in $\R^3$ can be classified into three classes: compact, cylinder, and asymptotically conical. The cylinder is known to be generic. Therefore, together with \cite{SX}, we know how to perturb away all types of first nongeneric singularities of MCFs in $\R^3$ by generic initial perturbations.
	
	Secondly, Chodosh-Schulze \cite{CS} proved that the tangent flow of an asymptotically conical shrinker is unique. Therefore, over any compact region, the RMCF can be written as a graph over the limit shrinker for a sufficiently large time. 
	
	Thirdly, asymptotically conical shrinkers have some nice properties themselves. For example, in \cite{BW} Bernstein-Wang analyzed the spectrum and eigenfunctions on an asymptotically conical shrinker and proved certain nice bounds. In Section \ref{SConvergence}, we proved that if the RMCF converges to an asymptotically conical shrinker and it models the unique singularity, some geometric quantities converge.

	Examples of asymptotically conical shrinkers are firstly constructed by Angenent-Ilmanen-Chopp in \cite{AIC} using numerical methods, and later Nguyen \cite{Ngu} and Kapouleas-Kleene-M$\phi$ller \cite{KKM} constructed examples theoretically. The theory of asymptotically conical shrinkers is interesting and has attracted mathematicians. 
	We refer the reader to \cite{BW} for further detailed discussions on asymptotically conical shrinkers. 
	
	\subsection{Convergence of eigenvalues and eigenfunctions}
	
	To apply Theorem \ref{mainthm-FK}, we have to verify the assumptions on the convergence of eigenvalues. The problem of spectral flow, i.e., how eigenvalues and eigenfunctions of parameter-dependent elliptic operators depend on the parameter, is important in many applications and has been studied widely in literature (c.f. \cite{A,HM,U,Z} etc). However, our setting is rather special since a family of compact manifolds $M_t$ converges to a noncompact one in the $C^\infty_{\loc}$ sense and the $L^2$-norm is defined with a Gaussian weight. 
	
	In Section \ref{SConvergence}, we prove the following convergence result assuming ($\star$), which may have an independent interest.  
	\begin{theorem}Assume $(\star)$, then we have as $t\to\infty$ 
		\begin{enumerate}
			\item $\lambda_1(M_t)\to \lambda_1(\Sigma)$;
			\item there is a constant $D>0$ such that $\lambda_1(M_t)-\lambda_2(M_t)>D. $
		\end{enumerate}
	\end{theorem}
	Section \ref{SConvergence} also contains further information on the convergence of eigenfunctions, etc. 
	
	\subsection{Organization of paper}
	The paper is organized as follows. In Section \ref{S:FK}, we establish the Feynman-Kac formula in the setting of RMCFs. In Section \ref{S:Asymvariation}, we study the asymptotic behavior of the solution to the variational equation using the Feynman-Kac formula. In Section \ref{S:Dynamics in a neighborhood of the shrinker}, we study the dynamics in a neighborhood of the shrinker, and hence completing the proofs of the main theorems stated above. In Section \ref{SEstimate}, we give some graphical estimates for the RMCF close to the shrinker using pseudolocality and the Ecker-Huisken estimate, etc. In Section \ref{SConvergence}, we study the convergence of the leading eigenvalue and eigenfunctions for the $L$-operator on $M_t$ as $t\to\infty$. Finally, we have four appendices containing some technical ingredients. In Appendix \ref{SS:transplantation}, we give the estimates for functions on $M_t$ pulled back to $\Sigma$ on a compact set that we call transplantations. In Appendix \ref{AppP}, we introduce polar-spherical transplantations adapted to conical shrinkers and give the estimates for the pullback of the $L$-operator. In Appendix \ref{AppPropCM}, we give the proof of Proposition \ref{LmCM3}. In Appendix \ref{AppKernel}, we prove the existence of the heat kernel of $L$ on a conical shrinker. 
	
	\subsection*{Acknowledgement}
	A.S. wants to thank his advisor Professor Bill Minicozzi for his encouragement and support, as well as many helpful comments. J. X. would like to thank Professor Tobias Colding for introducing him to the subject. J. X. is supported by NSFC grant No. 12271285, the New Cornerstone Investigator Program and the Xiaomi Foundation.

	\section{The Feynman-Kac formula}\label{S:FK}
	In this section, we derive a Feynman-Kac formula adapted to the variational equation $u_t=L_{M_t} u$ over the RMCF $\{M_t\}_{t\in[0,\infty)}$. Throughout this section, $u,v$ are functions defined on the RMCF $M_t$. 
	
	\subsection{Heat kernels}\label{SSHeat}
	
	In this section we study the heat kernel of the equation
	$$\partial_tu=L_{M_t}u=\mathcal L_{M_t} u+V(x,t)u,$$
	where $\{M_t\}_{t\in[0,\infty)}$ is the RMCF with the limit $M_t\to \Sigma$ in $C_{\loc}^\infty$ as $t\to \infty$, $V$ is a smooth potential which is always chosen as 0 or $|A|^2+\frac12 $ in the paper and $\cL_{M_t}$ is the drifted Laplacian, defined by $\cL_{M_t}u=\Delta_{M_t}u-\frac{1}{2}\langle x,\nabla u\rangle$.
	\begin{definition}\label{DefHeatKernel}
		The heat kernel of the above equation is a function of the form $\cH(x,t; y,s)$, where $x\in M_t$ and $y\in M_s$ and we always assume $s<t$, satisfying
		\begin{enumerate}
			\item $\partial_t \cH=\cL_{x,t} \cH+V\cH,$
			\item $\lim_{t\searrow s}\cH(\cdot ,t,y,s)\mathrm{e}^{-\frac{|y|^2}{4}}=\delta_y$.
		\end{enumerate}
	\end{definition}
	With the heat kernel, we can express the solution to the initial value problem \begin{equation}\label{EqCauchy}\begin{cases}\partial_t u=L_{M_t} u,\\
			u(\cdot, s)=f(\cdot)\end{cases}\end{equation} as \begin{equation}\label{EqReproducing}u(x,t)=\int_{M_s} \cH(x,t;y,s) f(y)\mathrm{e}^{-|y|^2/4}\,d\mu_s(y).\end{equation}
	Of particular importance for us is the following cocycle property. 
	\begin{theorem}[cocycle property]
		Let $\cH$ be the heat kernel for the equation  $\partial_t u=\cL u+Vu$, then for all $x\in M_r, y\in M_s, z\in M_t$ with $ r<s<t$, we have
		\begin{equation}\label{EqCocycle}\int_{M_s} \cH(z,t;y,s) \cH(y,s;x,r)\mathrm{e}^{-|y|^2/4}\,d\mu_s=\cH(z,t;x,r). \end{equation}
	\end{theorem}
	The proof is to take $s\to 0$ and apply item (2) of Definition \ref{DefHeatKernel}. We refer the reader to \cite[Chapter 26]{CC} for the existence of heat kernel and related properties.
	\subsection{The Trotter product formula for evolving manifolds}
	The next ingredient is the Trotter product formula. The classical Trotter product formula is as follows: let $A$ and $B$ be two adjoint operators bounded from below on a Hilbert space $\mathbb H$ and suppose that $A+B$ is adjoint on $\mathcal D(A)\cap \mathcal D(B)$, where $\mathcal D(\bullet)$ is the domain of $\bullet$, then $\mathrm{e}^{-t(A+B)}=\lim_{n} (\mathrm{e}^{-\frac{t}{n}A}\mathrm{e}^{-\frac{t}{n}B})^n$ in the strong operator norm on $\mathbb H$. 
	
	To adapt this formula to our setting of evolving manifolds, we need a non-autonomous Trotter formula, which was developed in  \cite{V, VWZ}. 
	Let $\mathcal T(t,s)$ be the fundamental solution to the variational equation \eqref{EqMain}, i.e., for all $u(s)\in C^\infty(M_s)$, the function $\mathcal T(t,s)u(s):=u(t):\ M_t\to \R$ satisfies \eqref{EqMain} with the initial condition $u(s)$. The fundamental solution is related to the heat kernel by $$u(x,t)=\mathcal T(t,s)u(y,s)=\int_{M_s}\mathcal H(x,t;y,s)u(y,s)\mathrm{e}^{-\frac{|y|^2}{4}}\,d\mu_s(y).$$ 
	So the fundamental solution $\mathcal T(t,s)$ can be extended to a linear operator from $L^2(M_s)$ to $L^2(M_t)$. 
	Similarly, we let $\bar{\mathcal T}(t,s)$ be the fundamental solution to the equation $\partial_t u=\cL_{M_t} u$. Then we get the following result by applying the non-autonomous Trotter product formula in \cite{V} in our setting. 
	
	\begin{proposition}Let $V(\cdot, \tau):\ M_\tau\to \R,\ 0\leq s\leq \tau\leq t<\infty$ be a smooth function that is also smooth in $\tau$ and let $\mathcal T(t,s)$ and $\bar{\mathcal T}(t,s)$ be as above. Let $t_k=k(t-s)/n+s$. Then we have 
		$$\mathcal T(t,s)=\lim_n\prod_{k=0}^{n-1}\left(\bar{\mathcal T}(t_{k+1}, t_k) \mathrm{e}^{\frac{t-s}{n}V(\cdot ,t_k)}\right)$$
		in the strong operator topology as linear operators from $L^2(M_s)$ to $L^2(M_t)$. 
	\end{proposition}
	
	\subsection{The Feynman-Kac formula}
	We next prove the Feynman-Kac formula in the setting of RMCFs.
	
	\begin{theorem}\label{ThmFK}
		Let $V(\cdot, \tau):\ M_\tau\to \R,\ 0\leq s\leq \tau\leq t<\infty$ be a smooth function that is also smooth in $\tau$. Then there exists a positive measure $\nu_{x,t}$ on the infinite product space $\Omega:=\prod_{0\leq \tau\leq t} M_\tau $ such that the solution of the Cauchy problem \eqref{EqCauchy}
		is represented as follows in the $L^2$-norm: 
		$$u(x,t)=\int_{\Omega} f(\omega(0))\exp\left(\int_0^t V(\omega(s),s)\,ds\right)d\nu_{x,t}(\omega).$$
	\end{theorem}
	\begin{proof}
		We follow the argument in Chapter X of \cite{RS}, which avoids using the probabilistic language. 
		We introduce the space $\Omega:=\prod_{0\leq \tau\leq t} M_\tau $ of all paths $\{\omega_\tau\}$ along the RMCF such that $\omega_\tau\in M_\tau$ endowed with the product topology. For fixed $\mathbf t:=(t_1,\ldots,t_m )$ with $t_1<t_2<\ldots<t_m$, we introduce a subspace $\Omega(\mathbf t)$ of $\Omega$ defined as $\Omega(\mathbf t)=\prod_{i=1}^m M_{t_i}$. Let $F:\ \Omega(\mathbf t)\to \R$ be a continuous function on $\Omega(\mathbf t)$ and $\omega\in \Omega$ be a path along $(M_t)$. The restriction map $\varphi:\ \Omega\to \R$ is defined as $\varphi(\omega):=F(\omega(t_1),\ldots,\omega(t_m))$. 
		
		We denote by $C_{\mathrm{fin}}(\Omega)$ the set of all such functions on $\Omega$ with all possible choices of time slices $\mathbf t$ and introduce a linear functional $\mathbb L_{x_{m+1},t_{m+1}}$ for each given point $x_{m+1}$ on the $t_{m+1}$-slice $M_{t_{m+1}}$:
		\begin{equation*}
			\begin{aligned}
				\mathbb L_{x_{m+1},t_{m+1}}(\varphi)&=\int_{M_{t_m}}\cdots \int_{M_{t_1}}F(x_1,\ldots,x_m) \bar\cH(x_{m+1},t_{m+1};x_m,t_m)\cdots \bar\cH(x_2,t_2;x_1,t_1)\\
				&\mathrm{e}^{-\frac14(|x_0|^2+\cdots+|x_{m}|^2)}\,d\mu_{x_1}\ldots d\mu_{x_m},\end{aligned}\end{equation*}
		where $ \bar\cH$ is the heat kernel for the heat equation $\partial_t u=\mathcal L_{M_t}u$ defined as in Section \ref{SSHeat} with $V\equiv0$. 
		
		The linear functional $\mathbb L$ is well-defined on $C_{\mathrm{fin}}$. Indeed, let $\mathbf t'$ be a finite superset of $\mathbf t$, then $F:\ \Omega(\mathbf t)\to \R$ can be considered as a function $F':\ \Omega(\mathbf t')\to R$ which agrees with $F$ on the time slices in $\mathbf t$ and constant on the time slices in $\mathbf t'\setminus\mathbf t$. Then the cocycle property of the heat kernel enables us to integrate out the variables on the slices in $\mathbf t'\setminus\mathbf t$. 
		
		Then we obtain a bounded positive linear functional $\mathbb L_{x_{m+1}, t_{m+1}}$ on $C_{\mathrm{fin}}(\Omega)$, and by Stone-Weierstrass theorem there exists a unique extension to $C(\Omega)$. Then by the Riesz representation, we obtain a unique Borel  measure $\nu_{x_{m+1}, t_{m+1}}$ such that for all $\varphi\in C(\Omega)$
		$$\mathbb L_{x_{m+1},t_{m+1}}(\varphi)=\int_\Omega \varphi(\omega) \, d\nu_{x_{m+1}, t_{m+1}}(\omega).$$
		This gives us a representation of solutions of the Cauchy problem \eqref{EqCauchy} as follows: for each $x\in M_t$, we have 
		$$\mathbb L_{x,t} f=\int_{\Omega} f(\omega(0))\,d\nu_{x,t}(\omega)=\int_{M_0} f(y)\bar\cH(x,t;y,0) \mathrm{e}^{-\frac{|y|^2}{4}}d\mu_{y}.$$
		
		Finally, the Trotter product formula shows that (denoting $t_k=tk/n,\ x_n=x,\ t_n=t$)
		\begin{equation*}
			\begin{aligned}
				u(x,t)&=\lim_n\prod_{k=0}^{n-1}\left(\bar{\mathcal T}(t_{k+1},t_k) \mathrm{e}^{\frac{t}{n}V(x_k,t_k)}\right)f\\
				&=\lim_n\int_{M_{t_{n-1}}}\cdots\int_{M_0}\mathrm{e}^{\frac{t}{n}\sum_k V(x_k,t_k)} \bar\cH(x, t;x_{n-1},t_{n-1})\cdots \bar\cH(x_1, t_1;x_0,0)f(x_0)\\
				&\mathrm{e}^{-\frac14(|x_0|^2+\cdots+|x_{n-1}|^2)}\,d\mu_{x_0}\cdots d\mu_{x_{n-1}}\\
				&=\int_{\Omega}\mathrm{e}^{\int_0^t V(\omega(s),s)\,ds}f(\omega(0))\,d\nu_{x,t}(\omega). 
			\end{aligned}
		\end{equation*}
		This gives the Feynman-Kac formula. 
	\end{proof}
	
	\begin{remark}
		The measure $\nu$ is called the Wiener measure.
	\end{remark}
	
	\subsection{The localization}
	One notable difficulty in the study of singularities modeled by noncompact shrinkers in the MCF theory is that in general the manifold $M_t$ cannot be written as a global graph over the limiting shrinker $\Sigma$ no matter how large $t$ is, where $\{M_t\}_{t\in[0,\infty)}$ is the RMCF with $M_t\to \Sigma$ in $C_{\loc}^\infty$ as $t\to\infty. $
	Therefore it is natural to consider the localized Dirichlet boundary value problem. Let $B_R$ be a big open ball in $\R^{n+1}$. When $t$ is sufficiently large, we can write part of $M_t$ as a normal graph over $B_R\cap \Sigma$. We denote by $M_t^R$ this part of $M_t$ and by $L^R_{M_t}$ the restriction of the $L_{M_t}$-operator  to $M_t^R$, and introduce the evolutionary Dirichlet boundary value problem
	\begin{equation}\label{EqDiri}
		\begin{cases}
			\partial_t u&=L^R_{M_t} u\\
			u(t,\cdot)|_{\partial M_t^R}&=0
		\end{cases}.
	\end{equation}
	For problem \eqref{EqDiri}, we can also introduce the heat kernel as in Section \ref{SSHeat}, which also has the cocycle property (c.f. Lemma 26.12 of \cite{CC}). Let us denote by $\mathcal H^R$ its heat kernel,  and by $\bar\cH^R$ the heat kernel for the Dirichlet boundary value problem with $V\equiv0$. Then we can repeat the argument of Theorem \ref{ThmFK} to obtain a Wiener measure $\nu^R$ that is supported on $\Omega^R:=\prod_{0\leq \tau\leq t}M_\tau^R$ such that the solution to \eqref{EqDiri} can be represented as
	$$[\cT^R(t,s)u(s,\cdot )](x)=\int_{\Omega} u(s,\omega(s))\exp\left({\int_s^t V(\omega(\tau),\tau)\,d\tau}\right) d\nu^R_{x,t}(\omega),$$
	where we use $\cT^R(t,s)$ to denote the fundamental solution to \eqref{EqDiri}. 
	
	One remarkable property of the Feynman-Kac representation is that we can compare the solution of \eqref{EqDiri} to that of the original Cauchy problem without cutoff. Using the parabolic maximum principle, we have $\bar\cH(x,t;y,s)\geq \bar\cH^R(x,t;y,s)$ pointwise. Thus we get the following proposition by comparing the argument of Theorem \ref{ThmFK}. 
	\begin{proposition}\label{PropLocalOp}
		Let $\nu$ be the Wiener measure constructed in Theorem \ref{ThmFK} and $\cT^R(t,s)$ be as above. Then we have pointwise for any positive function $u$,
		$$\int_{\Omega}  u(\omega(s),s)\exp\left({\int_s^t V(\omega(\tau),\tau)\,d\tau}\right) d\nu_{x,t}(\omega)\geq [\cT^R(t,s)u(\cdot,s )](x).$$
	\end{proposition}
	
	\section{The asymptotic behavior of the solution to the variational equation}\label{S:Asymvariation}
	In this section, we use the Feynman-Kac formula to study the asymptotic behavior of positive solutions to the variational equation $\partial_t v^\star=L_{M_t}v^\star$. We shall give the proof of Theorem \ref{mainthm-FK} in Section \ref{SSFKCone}.
	We introduce the  $L^2$-norm and the $Q$-norm  on $M_t$ as follows. 
	
	\begin{definition}
		\begin{enumerate}
			\item The $L^2(M_t)$-norm is defined by
			$\|u\|_{L^2(M_t)}=\left(\int_{M_t} |u(x)|^2  \mathrm{e}^{-|x|^2/4} d\mu\right)^{1/2}$ for a function $u:\ M_t\to\R$;
			\item the $Q$-norm is defined by 
			$$\|u\|_{Q(M_t)}=\left(\int_{M_t} \left(|\nabla u(x)|^2+ \Lambda u(x)^2- (|A|^2+\frac12)u(x)^2\right)  \mathrm{e}^{-|x|^2/4} d\mu\right)^{1/2}.$$
			Here we pick $\Lambda> \sup_t\lambda_1(t)$ where $\lambda_1(t)$ is the first eigenvalue of $L_{M_t}$ on $M_t$. In the case we study later, $\lambda_1(t)$ is uniformly bounded from above, so we can always pick such a $\Lambda$. 
			\item We abbreviate $L^2(t)$ to $L^2(M_t)$, and $Q(t)$ to $Q(M_t)$.
			\item We also introduce $L^2(\Sigma)$ and $Q(\Sigma)$ similarly and abbreviate them as $L^2$ and $Q$ respectively. 
		\end{enumerate}
	\end{definition}

	\subsection{The cone-preservation property}\label{SSConeLinear}
	In this section, we study the asymptotic dynamics of the evolutionary Dirichlet boundary value problem \eqref{EqDiri} with some large $R$ fixed. For each $R>0$, there exists $T=T(R)$ such that for all $t>T$, we can write $M_t^R$ as a normal graph of a function $m_t$ over $\Sigma^R:=\Sigma \cap B_R$, i.e., 
	$$M_t^R=\mathrm{Graph}\{x+m_t(x)\bn(x),\quad x\in \Sigma^R\}. $$
	This provides a diffeomorphism $\varphi_t:\ \Sigma^R\to M_t^R$,  via $x\mapsto x+m_t(x)\bn(x)$,  that converges to identity in the $C^2$-norm as $t\to\infty$.
	
	A function $f:\ M_t^R\to \R$ is pulled back by $\varphi_t$ to a function $f^*=f\circ \varphi_t:\ \Sigma^R\to \R$. This is called ``transplantation" in \cite{SX}.
	From the equation $\partial_tv^\star=L^R_{M_t} v^\star$, we obtain the equation for $v^*:=v^\star\circ\varphi_t$ as 
	\begin{equation}\label{EqPullBack}
		\partial_t v^*=L^R_\Sigma v^*+P(v^*,t),
	\end{equation}
	where $L^R_\Sigma$ is the restriction of $L_\Sigma$ to $\Sigma^R$, and $P(v^*,t)\to 0$ as $t\to\infty$ uniformly for all $\|v^*\|_{C^2}\leq 1$. Because $v^\star$ satisfies a linear equation, we have $P(\lambda v^*, t)=\lambda P(v^*,t)$. Moreover, we have the following estimate for $P$ 
	
	\begin{equation}\label{EqP}
		\begin{aligned}
			|P(v^*,t)|
			\leq C\|m\|_{C^2}(|\nabla^2_\Sigma v^*|+|\nabla_\Sigma v^*|+|v^*|).
		\end{aligned}
	\end{equation}
	\begin{lemma}
		The $L^2(M_t^R)$ and the $L^2(\Sigma^R)$-norms are equivalent under the transplantation.
	\end{lemma}
	\begin{proof}
		
		Given $v^\star\in L^2(M^R_t)$, we have that $v^*\in L^2(\Sigma^R)$. Indeed, in the definition of the $L^2_R(M_t)$-norm, we introduce a coordinate change $x\mapsto \varphi_t(x)$. The Jacobian is close to identity. We next consider the Gaussian weight, we have $\mathrm{e}^{-\frac{|x+m_t(x)\bn(x)|^2}{4}}=\mathrm{e}^{-\frac{|x|^2}{4}-\frac{m_t(x)^2}{4}-\frac{m_t(x)}{2} x\cdot \bn(x)}.$
		Since $x\in \Sigma$ and $\bn(x)$ is the unit normal at $x$, we get $\frac12x\cdot \bn(x)=H(x)$.
		Restricted to $B_R$, we have $|H|<C$ and then we get the following estimate for $|m_t|$ sufficiently small over $B_R$ $$(1-\eps )\mathrm{e}^{-\frac{|x|^2}{4}}\leq \mathrm{e}^{-\frac{|x+m_t(x)n(x)|^2}{4}}\leq (1+\eps) \mathrm{e}^{-\frac{|x|^2}{4}}. $$
	\end{proof} 
	
	In this way, we have converted the fundamental solution $\mathcal T^R(t,s)$ to \eqref{EqDiri} over the RMCF $M_t^R$ into a non-autonomous dynamical system \eqref{EqPullBack} on the fixed manifold $\Sigma^R$.  Abusing notation slightly, we still use $\mathcal T^R(t,s)$ to denote the fundamental solution to the system \eqref{EqPullBack}. 
	
	Let $\alpha>0$ be a positive number. We introduce the cone of functions:
	$$\mathcal K^R(\alpha):=\{ f\in  L^{2}(\Sigma^R)\ |\ \|\pi_1 f\|_{L^{2}(\Sigma^R)}\geq \alpha\|\pi_2 f\|_{L^{2}(\Sigma^R)}\},$$
	where $\pi_1$ means the $L^2(\Sigma^R)$-projection to the direction of the first eigenfunction of $L^R_{\Sigma}$ and $\pi_2$ is the $L^2(\Sigma^R)$-orthogonal complement of $\pi_1$. The larger $\al$ is, the narrower the cone is around $\phi_1^R$, the leading eigenvector of $L_\Sigma^R$. 
	
	\begin{proposition}\label{PropCone}
		\begin{enumerate}
			\item For each $\alpha$ and $R$, there exists $T$ sufficiently large such that for all $t>s>T$ we have
			$\cT^R(t,s)\mathcal K^R(\alpha)\subsetneqq \mathcal K^R(\alpha).$
			\item  For all $\eps>0,\al>0$, and $R>0$, there exists $T$ sufficiently large,   such that for all  $t>s\geq T$ and all $v^*\in \mathcal K^R(\alpha)$ we have $$\|\pi_1 (\cT^R(t,s)v^*)\|_{L^2(\Sigma^R)}\geq \mathrm{e}^{(t-s)(\lambda^R_1-\eps)}\|\pi_1(v^*)\|_{L^2(\Sigma_R)},$$
			where $\lambda^R_1$ is the leading eigenvalue of $L_\Sigma^R$. 
		\end{enumerate}
	\end{proposition}
	\begin{proof}We first consider $t=s+1$ and the initial condition to the equation \eqref{EqPullBack} to be $v^*(0)$ at the time $s$.
		With the following Lemma \ref{LmApproximation}, we get (denoting $\|\cdot\|=\|\cdot\|_{L^2(\Sigma^R)} $)
		\begin{equation*}
			\begin{aligned}
				\|\pi_1v^*(1)\|&\geq \mathrm{e}^{\lambda_1-\dt}\|\pi_1v^*(0)\|,\quad \|\pi_2v^*(1)\|\leq \mathrm{e}^{\lambda_2+\dt}\|\pi_2v^*(0)\|.\\
			\end{aligned}
		\end{equation*}
		Taking the quotient, we get that $\frac{\|\pi_1v^*(1)\|}{\|\pi_2v^*(1)\|}\geq \mathrm{e}^{\lambda_1-\lambda_2-2\dt}\frac{\|\pi_1v^*(0)\|}{\|\pi_2v^*(0)\|}$. The statement then follows from taking $\dt$ small and iterating the argument. 
		
	\end{proof}
	
	\begin{lemma}\label{LmApproximation}
		Let $v^*(t)$ be a solution to the non-autonomous system \eqref{EqPullBack} with the initial value $v^*(T)$ at the initial time $t=T$ and the Dirichlet boundary value. Then for any $\dt>0$, there exists $T_0>0$ such that for all $T>T_0$  we have
		$$\|v^*(T+1)-\mathrm{e}^{L_\Sigma^R} v^*(T)\|_{L^2(\Sigma^R)}\leq \dt\|v^*(T)\|_{L^{2}(\Sigma^R)}.$$
		
	\end{lemma}
	
	\begin{proof}
		Suppose $M^R_t$ is written as the graph of the function $m_t$ over $\Sigma^R$, then we have the bound \eqref{EqP}, where the constant $C$ is a constant depending on the geometry of $\Sigma^R$. Here we restrict the problem to a fixed ball of radius $R$, so $C$ is uniformly bounded. 
		When $t$ is large, we have that $M_t$ is sufficiently close to $\Sigma$, and hence $\|m_t\|_{C^2}$ is very small.
		
		To prove the lemma, let $w(t)=v^*(t)-v^\star(t)$ where $\partial_tv^*=L_\Sigma^R v^*+P(v^*,t)$ and $\partial_tv^\star=L_\Sigma^R v^\star$ with the same initial value $v^*(T)$ at the time $t=T$, then we have $\partial_t w=L_\Sigma^R w+P(v^*,t)$ and 
		\begin{equation*}
			\begin{aligned}
				\partial_t\int_{\Sigma^R} |w|^2 \mathrm{e}^{-\frac{|x|^2}{4}}d\mu&=-\int_{\Sigma^R}  |\nabla w|^2  \mathrm{e}^{-\frac{|x|^2}{4}}d\mu+\int_{\Sigma^R}  wP(v^*,t)  \mathrm{e}^{-\frac{|x|^2}{4}}d\mu\\
				&\leq \int_{\Sigma^R}  w^2  \mathrm{e}^{-\frac{|x|^2}{4}}d\mu+\int_{\Sigma^R}  P(v^*,t)^2  \mathrm{e}^{-\frac{|x|^2}{4}}d\mu\\
				\mathrm{e}^{-1}\int_{\Sigma^R}  |w(T+1)|^2 \mathrm{e}^{-\frac{|x|^2}{4}}d\mu&\leq \int_0^1 \mathrm{e}^{-t} \int_{\Sigma^R}  P(v^*,T+t)^2  \mathrm{e}^{-\frac{|x|^2}{4}}d\mu\,dt.
			\end{aligned}
		\end{equation*}
		Next,  for any $\dt$, we can choose $t$ large enough so that $C\|m_t\|_{C^2}<\dt$. Then  we get from \eqref{EqP} that $|P(v^*,t)|\leq \dt(|\mathrm{Hess}_{v^*}|+|\nabla v^*|+|v^*|)$. Then following Corollary A.9 in \cite{SX} (also see \cite{CM2} Lemma 5.4), we bound 
		$\int_0^1\int (|\mathrm{Hess}_{v^*}|+|\nabla v^*|)^2\mathrm{e}^{-\frac{|x|^2}{4}}d\mu\,dt$ by a multiple of $\int |v^*(T)|^2\mathrm{e}^{-\frac{|x|^2}{4}}d\mu$. This completes the proof. 
	\end{proof}
	
	\subsection{The Feynman-Kac formula and the cone condition}\label{SSFKCone}
	
	With the Feynman-Kac formula, we give the proof of our main result in this subsection. 
	\begin{proof}[Proof of Theorem \ref{mainthm-FK}]
		We express $v^\star(x,t)$, the solution to $\partial_t v^\star=L_{M_t} v^\star$ with the initial condition $v^\star_0>0$, by Feynman-Kac formula as
		$$v^\star(x,t)=\int_{\Omega}v^\star_0(\omega(0))\exp\left(\int_0^t V(\omega(s),s)\,ds\right)d\nu_{x,t}(\omega).$$
		Here $\Omega$ is the set of all paths $\omega:\ [0,t]\to (M_s)_{0\leq s\leq t}$ with the endpoint $\omega(t)=x$. 
		
		We then pick a large $T$ such that for $t\geq T$, $M_t$ is sufficiently close to $\Sigma$ over a ball $B_R$ for some $R$ large, and such that we have the preservation of the cone condition (see Proposition \ref{PropCone}). 
		We next modify $V(\cdot, t)$ to the function $\widetilde V(\cdot,t)=\begin{cases}0,\ t<T,\\
			V(\cdot,t),\ t\geq T
		\end{cases}$. Then we get $$\exp\left(\int_0^t V(\omega(s),s)\,ds\right)\geq \exp\left(\int_0^t \widetilde V(\omega(s),s)\,ds\right)=\exp\left(\int_T^t V(\omega(s),s)\,ds\right),\ \mathrm{and}\ $$
		\begin{equation*}
			\begin{aligned}
				v^\star(x,t)&\geq c_T\int_{\Omega}v^\star_0(\omega(0))\exp\left(\int_T^t V(\omega(s),s)\,ds\right)d\nu_{x,t}(\omega)\\
				&\geq \frac{c_T\min v^\star_0}{\|\phi_1(T)\|_{L^\infty}}\int_{\Omega}\phi_1(T)(\omega(0))\exp\left(\int_T^t V(\omega(s),s)\,ds\right) d\nu_{x,t}(\omega),
			\end{aligned}
		\end{equation*}
		where $\phi_1(T)$ is the eigenfunction associated with the leading eigenvalue of the operator $L_{M_T}$.
		
		By the cocycle property of the heat kernel applied to the time interval $[0,T]$, the above integral becomes $$\int_{\Omega}\phi_1(T)\exp\left(\int_T^t V(\omega(s),s)\,ds\right)d\nu_{x,t}(\omega)=\int_{\Omega_1}\phi_1(T)\exp\left(\int_T^t V(\omega(s),s)\,ds\right)d\nu_{x,t}(\omega),$$
		where $\Omega_1$ is the set of paths $\omega:\ [T,t]\to (M_s)_{T\leq s\leq t}$.
		We next use Proposition \ref{PropLocalOp} to obtain 
		\begin{equation*}
			\begin{aligned}
				&\int_{\Omega_1}\phi_1(T)\exp\left(\int_T^t V(\omega(s),s)\,ds\right)d\nu_{x,t}(\omega)
				\geq\int_{\Omega_1}\phi_1(T)\exp\left(\int_T^t V(\omega(s),s)\,ds\right)d\nu^R_{x,t}(\omega).
			\end{aligned}
		\end{equation*}
		By the Feynman-Kac formula, the right-hand side of the above formula is exactly the solution to the variational equation with the initial condition $v^\star(\cdot ,T)=\phi_1(T)|_{M_T^R}$ and the Dirichlet boundary condition $v^\star(t,\cdot)|_{M_t^R}=0$ for all $t>T$. 
		
		Now we transplant the problem to $\Sigma^R$. By the cone preservation property Proposition \ref{PropCone}(1), if $\phi_1(T)$ lies in the cone $\mathcal K(\al)$, then its future orbit lies in the cone. This implies for all $\eps>0$, $\lim_{t\to\infty} \frac{1}{t}\log\|v^\star(t)\|_{L^2_R(\Sigma)}\geq \lambda^R_1-\eps$ (see Proposition \ref{PropCone}(2)).  By Lemma 9.25 of \cite{CM1} (see also Theorem \ref{thm:conv1steigenfun} for $\Sigma$ being asymptotically conical), we have that $\lambda_1^R\to \lambda_1$ which is the first eigenvalue of $L_\Sigma$ and the leading eigenfunction $\phi^R_1$ of $L_\Sigma^R$ converges in the $C^\infty_{\loc}$ sense to $\phi_1$, the leading eigenfunction of $L_\Sigma$. Thus $\lim_{t\to\infty} \frac{1}{t}\log\|v^\star(t)\|_{L^2(M_t)}\geq \lambda^R_1-\eps$. On the other hand, we have
		$$\frac12\partial_t\int_{M_t} (v^\star)^2 \mathrm{e}^{-\frac{|x|^2}{4}}d\mu=\int_{M_t} \left( v^\star L_{M_t}v^\star -\widetilde H^2 (v^\star)^2\right)\mathrm{e}^{-\frac{|x|^2}{4}}d\mu\leq \lambda_1(t)\int_{M_t} (v^\star)^2 \mathrm{e}^{-\frac{|x|^2}{4}}d\mu,$$
		where $\widetilde H$ is the rescaled mean curvature $\left(H-\frac{\langle x,\bn\rangle}{2}\right)$ and $\lambda_1(t)$ is the leading eigenvalue of $L_{M_t}$. By assumption, we have $\limsup_t\lambda_1(t)\leq\lambda_1(\Sigma)$. This gives $$\|v^\star(t)\|_{L^2(M_t)}\leq \mathrm{e}^{(\lambda_1(\Sigma)+\eps)t}\|v^\star_0\|_{L^2(M_0)}$$ for $t$ sufficiently large, for any $\eps>0$. Thus we get item (1). 
		
		We next work on item (2) of the statement.  
		For any time $t$, we decompose $v^\star(t)=v^\star_1(t)+v^\star_2(t)+v^\star_3(t)$, where $v^\star_1$ is the projection of $v^\star$ to the eigenspace corresponding to the leading eigenvalue of $L_{M_t}$,  $v^\star_3$ consists of Fourier modes with negative eigenvalues and $v^\star_2$ consists of the rest,  and write $$  g_1(t)\|v^\star\|^2_{L^2(M_t)}=\|v^\star_1\|_{L^2(M_t)}^2,\quad
		g_2(t)\|v^\star\|^2_{L^2(M_t)}=\|v^\star_2\|_{L^2(M_t)}^2,$$
		\[g_3(t)\|v^\star\|^2_{L^2(M_t)}=-\int_{M_t} v^\star_3L_{M_t}v^\star_3 \mathrm{e}^{-\frac{|x|^2}{4}}d\mu.\]
		By Pythagoras' theorem we have $g_1(t)+g_2(t)\leq 1$. 
		Notice that with this notion, we have
		\[\|v^\star\|_{Q(M_t)}^2\leq (\Lambda+g_3(t))\|v^\star\|_{L^2(M_t)}^2.\]
		Then we can write the time derivative of $\|v^\star\|_{L^2(M_t)}^2$ as follows:
		\[\begin{split}
			\frac{1}{2}\partial_t \|v^\star\|_{L^2(M_t)}^2
			\leq& \int_{M_t} v^\star L_{M_t} v^\star \mathrm{e}^{-\frac{|x|^2}{4}}d\mu\\
			\leq&\max\{0,\lambda_2(t)\} g_2(t)\| v^\star(t)\|^2_{L^2(M_t)}
			+\lambda_1(t) g_1(t)\| v^\star(t)\|^2_{L^2(M_t)}-g_3(t)\|v^\star\|^2_{L^2(M_t)}
			\\
			=&
			[\max\{0,\lambda_2(t)\}g_2(t)+g_1(t)\lambda_1(t)-g_3(t)]\|v^\star\|^2_{L^2(M_t)}.
		\end{split}\]
		Then we claim that there exist $c>0$ and infinitely many $t_i\to\infty$ such that $\frac{g_1(t_i)}{\Lambda+g_3(t_i)}\geq c$. In fact, there exist $c_1>0$ and infinitely many $t_i\to\infty$ such that $$(\lambda_1(t_i)-\lambda_2(t_i)) g_1(t_i)-g_3(t_i)\geq c_1>\lambda_1(t_i)-\lambda_2(t_i).$$ Otherwise it violates the exponential growth in item (1). Then we get $g_1(t_i)\geq 1+(\lambda_1(t_i)-\lambda_2(t_i))^{-1}(g_3(t_i)-\epsilon)$. For $i$ large, this forces $|g_3(t_i)|<\epsilon$. This shows that there exist $c>0$ and infinitely many $t_i\to\infty$ such that
		$\frac{|\langle v^\star(t_i),\phi_1(t_i)\rangle_{L^2({t_i})}|}{\|v^\star(t_i)\|_{Q({t_i})}}>c$.
	\end{proof}
	
	We need the following approximation of the RMCF equation and the linearized RMCF equation in \cite{SX}.
	
	\begin{proposition}[Proposition 3.3 in \cite{SX}]\label{Prop:DiffRMCFandLinear}\label{prop:approx}
		Given an RMCF $\{M_t\}_{t\in[0,T]}$, there exist $\delta_0>0$, $\alpha'>0$, $\eps_0>0$ and $C$ so that if $v_0:M_0\to\R$ satisfies
		\[
		|v_0|+|\nabla v_0|\leq \delta\leq\delta_0,\quad |\Hess_{v_0}|\leq 1,
		\]
		$v:[0,T]\to\R$ satisfies the RMCF equation, and $v^\star:[0,T]\to\R$ satisfies the linearized equation, with $v(\cdot,0)=v^\star(\cdot,0)=v_0$, then we have 
		\begin{equation*}
			\|(v-v^\star)(\cdot,T)\|_{C^{2,\alpha'}}\leq C\delta^{1+\eps_0}.
		\end{equation*}
	\end{proposition}
	
	This proposition is a consequence of the fact that the nonlinear term of the RMCF equation is quadratic. With this proposition, we can prove that if the initial positive perturbation is sufficiently small, then the solution to the RMCF equation on $M_t$ will also drift to the first eigenfunction direction. 
	
	\begin{proposition}\label{thm:perturbedRMCFto1st}
		Assume the setting of Theorem \ref{mainthm-FK}.	Suppose $v_0>0$ is a $C^{2}$ positive function on $M_0$. Then there exist $1>c>0$ and $C>1$ and a sequence $t_i\to\infty$, such that for any fixed $t_i$, there exists $\epsilon_i>0$, such that for $\epsilon<\epsilon_i$, the perturbed RMCF starting from $\{x+\epsilon v_0(x)\bn(x)\ :\ x\in M_0\}$ can be written as a graph of the function $v(\cdot,t)$ over $M_t$ for $t\in[0,t_i]$, such that 
		$$\|v\|_{Q({t_i})}\leq C\|v\|_{L^2({t_i})},\ \mathrm{and}\ \frac{|\langle v(t_i),\phi_1(t_i)\rangle_{L^2({t_i})}|}{\|v(t_i)\|_{Q({t_i})}}>c.$$ 
	\end{proposition}
	
	\subsection{The cone condition with cutoff}
	In this subsection, we show that conclusion (2) in Theorem \ref{mainthm-FK} also holds if we restrict it to a bounded part of the manifold $M_{t_i}$. This will be useful in the next section when we study the dynamics in a neighborhood of $\Sigma$, considering that $M_t$ can not be written as a global graph over $\Sigma$.

	\begin{theorem}\label{thm:vstarcutsmall}
		there exist $R_0>0$ and $C>0$ such that for $R>R_0$ and the sequence of $t_i$ in Theorem \ref{mainthm-FK}, we have 
		$\|v^\star(t_i)\|_{L^2(M_{t_i})}\leq C\|v^\star(t_i)\|_{L^2(M_{t_i}^R)}.$
	\end{theorem}
	
	\begin{proof}
		Theorem \ref{thm:conv1steigenfunL2} shows that there exist $R_0>0$ and $\eta>0$ such that when $R>R_0$, $\|\phi_1(t)-\eta\phi_1\|_{L^2(\Sigma^R)}\to 0$ as $t\to\infty$. Here $\phi_1(t)$ is transplanted to $\Sigma^R$ when $t$ is sufficiently large. From Theorem \ref{mainthm-FK}, we have
		$$\|v^\star\|_{Q({t_i})}\leq C\|v^\star\|_{L^2({t_i})},\ \mathrm{and}\ \frac{|\langle v^\star(t_i),\phi_1(t_i)\rangle_{L^2({t_i})}|}{\|v^\star(t_i)\|_{Q({t_i})}}>c.$$ 
		Therefore (we choose $\phi_1(t)>0$)
		\[\|v^\star\|_{L^2({t_i})}< c\langle v^\star(t_i),\phi_1(t_i)\rangle_{L^2({t_i})}= c\langle v^\star(t_i),\phi_1(t_i)\chi_{B_R}\rangle_{L^2({t_i})}+c\langle v^\star(t_i),\phi_1(t_i)(1-\chi_{B_R})\rangle_{L^2({t_i})}.\]
		
		Corollary \ref{cor:1stcutsmall} implies that $\|\phi_1(t)(1-\chi_R)\|_{L^2(M_t)}\leq CR^{-1}$ for sufficiently large $R,t$. Thus, 
		\[\|v^\star\|_{L^2(t_i)}<c\langle v^\star(t_i),\phi_1(t_i)\chi_{B_R}\rangle_{L^2({t_i})}+CR^{-1}\|v^\star\|_{L^2({t_i})}.\]
		So if $R$ is sufficiently large, 
		\[\|v^\star\|_{L^2({t_i})}\leq C\langle v^\star(t_i),\phi_1(t_i)\chi_{B_R}\rangle_{L^2({t_i})}\leq C \|v^\star\|_{L^2(M_{t_i}^R)}.\]
	\end{proof}
	
	\begin{corollary}\label{cor:vstarcutsmall}
		there exist $R_0>0$ and $C>0$ with the following significance: in the setting of Theorem \ref{mainthm-FK}, we have 
		$$\|v^\star\|_{Q(M_{t_i}^R)}\leq C\|v^\star\|_{L^2(M_{t_i}^R)},\ \mathrm{and}\ \frac{|\langle v^\star(t_i),\phi_1(t_i)\rangle_{L^2(M_{t_i}^R)}|}{\|v^\star(t_i)\|_{Q(M_{t_i})}}>c.$$ 
	\end{corollary}
	
	\section{The dynamics in a neighborhood of the shrinker}\label{S:Dynamics in a neighborhood of the shrinker}
	
	In this section, we show that the RMCF in a neighborhood of the shrinker is approximated by the linear equation $\partial_t u=L_\Sigma u$. 
	
	Let $M_t$ be an RMCF converging to a conical shrinker $\Sigma$ in the $C^\infty_{\loc}$ sense as $t\to\infty$ and we perturb the initial condition slightly to give a new RMCF $\widetilde M_t$. We choose the initial perturbation so small that the perturbed flow will stay close to the unperturbed flow for a long time until they both enter a $\dt$-neighborhood of the shrinker in the $C^{2,\al}$ norm over a large compact ball $B_R$. This is the red curve in Figure \ref{Figure}, which is controlled by the variational equation. The main body of this section is devoted to the blue curve which is the dynamics in a neighborhood of the shrinker to be approximated by the linear equation $\partial_t u=L_\Sigma u$. As we have discussed in the introduction, the main difficulty is created by the fact that $M_t$ and $\widetilde M_t$ cannot be written as global graphs over $\Sigma$ so that a cutoff is not avoidable. The time span of the blue curve of local dynamics depends on the smallness $\eps$ of the initial perturbation (roughly $\log\eps^{-1}$) so that the linear approximation has to be done over the growing domain since otherwise the discarded information by the cutoff will accumulate large errors over  $\log\eps^{-1}$ long time. In Section \ref{SSDomain}, we state a result on the exponential growth of the graphical domain to be proved in Section \ref{SEstimate}. In Section \ref{SSDifferenceEq}, we derive the equation governing the dynamics of the blue curve. In Section \ref{SSInitiateCone}, we give the necessary cone condition and the boundary condition to initiate the blue curve. In Section \ref{SSApproximate}, we show that the dynamics on the blue curve can be approximated by the solution to the linear equation. In Section \ref{SS:The cone preservation property}, we show that the cones become narrower and narrower under the dynamics, which enables us to complete the proof of the main theorem in Section \ref{SSProofMain}. In Section \ref{SSGeneric}, we show the perturbation can be generic and in Section \ref{SSAncient}, we consider ancient solutions.

	\subsection{The exponential growth of the graphical domain}\label{SSDomain}
	
	Denote by $\mathbb A_{r_1,r_2}$ the closure of the annulus region $B_{r_2}\setminus B_{r_1}$ for $r_2>r_1.$
	\begin{definition}\label{DefGraphical}
		Let us fix an integer $\ell>3$ and $r>0$, $\eps_0>0$. We define \emph{the graphical scale $\mathit r(M_t)$} as the largest radius $R$ such that $  M_t$ can be written as a graph of a function $u:\ \Sigma^R\to \R$ satisfying
		\begin{enumerate}
			\item $\|u\|_{C^{2,\alpha}(\Sigma^r)}<\eps_0$,
			\item $\|\nabla^i u\|_{C^{0}(\Sigma\cap \mathbb A_{s-1,s})}< s^{-i+1}\eps_0$, $i=0,1,2,\ell$, $r<s<R$.
		\end{enumerate}
		Here $\eps_0$ is a uniform constant independent of $R$.

	\end{definition}
	
	We will fix $\eps_0$ later, such that the following Proposition \ref{PropPL} holds. Proposition \ref{PropPL} will be proved in Section \ref{SSPL}.
	
	\begin{proposition}\label{PropPL}
		There exist $\eps_0$, $T>0$ and $C>0$ such that for all $t>T$ we have 
		$\mathit r(M_t)\geq C \mathrm{e}^{t/2}. $
	\end{proposition}
	
	We define ${\bf r}(M_t)$ to be the minimum of the above $\mathit r(M_t)$ and $C\mathrm{e}^{t/2}$. Sometimes we write $\mathbf{r}(t)$ for simplicity. On $\Sigma\cap B_ {\br(t)}$, $M_t$ can always be written as a graph with an appropriate decay rate. 
	
	\subsection{The evolution equation governing the difference of two nearby RMCFs}\label{SSDifferenceEq}
	We consider the following setting: Let $\{M_t\}_{t\in[0,\infty)}$ be an RMCF as in ($\star$). Suppose $\widetilde M_0$ is a small perturbation of $M_0$ and $\widetilde M_t$ is the rescaled MCF with the initial condition $\widetilde M_0$.

	We first write $\widetilde M_t$ as the normal graph of a function  $v(\cdot,t)$ over $M_t$, i.e., 
	\begin{equation}\label{EqGraphRMCF}\widetilde M_t=\{x+v(x,t)\bn(x)\ |\ x\in M_t\}.\end{equation}
	
	Then taking the difference of the RMCF equations for $\widetilde M_t$ and $M_t$ we get
	\begin{equation}\label{EqDifference2}\partial_t v=L_{M_t}v+ Q(v),\end{equation}
	where $Q(v)$ is quadratically small in $\|v\|_{C^2}$ (see \cite[Lemma 4.3]{CM2} and \cite[Appendix A]{SX}). 
	
	Suppose $t$ is sufficiently large so that the graphical scale of $M_t$ can be defined as in Definition \ref{DefGraphical}. We can introduce a diffeomorphism $\varphi$ between $\Sigma^{\mathbf r(t)}$ and $M_t^{\mathbf r(t)}$ so that a function $v$ on $M_t^{\mathbf r(t)}$ can be transplanted to a function on $\Sigma^{\mathbf r(t)}$ as the pullback $\varphi^* v$. Note that the difference between the two manifolds $M_t^{\mathbf r(t)}$ and $\Sigma^{\mathbf r(t)}$ grows linearly in the radial direction, since the $C^0$-norm in item (2) of Definition \ref{DefGraphical} does so. Instead of using the normal graphical function in Definition \ref{DefGraphical} to define the diffeomorphism $\varphi$, we adopt a polar-spherical coordinates approach. Let $\mathcal C:=\{r\theta\ |\ r\geq0,\ \theta\in \mathcal S\subset \mathbb S^n(1)\}$ be the cone such that $\lambda \Sigma\to \mathcal C$ as $\lambda\to 0_+$, where $\mathcal S$ is a codimension-$1$ submanifold of $\mathbb S^n(1)$. On each spherical slice $\mathbb S^n(r)$, both $M_t\cap \mathbb S^n(r)$ and $\Sigma\cap  \mathbb S^n(r)$ can be written as a normal graph over $r\mathcal S$ within $\mathbb S^n(r)$.
	
	\begin{definition}\label{DefTransplant}
		We define the diffeomorphism $\varphi_t:\ \Sigma^{\mathbf r(t)}\to M_t^{\mathbf r(t)}$ by mapping each point in $\Sigma^{\mathbf r(t)}\cap  \mathbb S^n(r)$ to the point of $M_t^{\mathbf r(t)}\cap  \mathbb S^n(r)$ if these two points are on the same normal line of $r\mathcal S$. Here the normal line can be understood as the exponential map from the normal vector of $r\mathcal S$ within $\mathbb S^n(r)$.
	\end{definition}
	Note that the diffeomorphism preserves the Gaussian weight $\mathrm{e}^{-\frac{|x|^2}{4}}$. 
	We next use  $\varphi_t$ to pull back functions on $M_t$ to $\Sigma$ to rewrite the equation as a nonautonomous system over $\Sigma^{\mathbf r(t)}$ as
	\begin{equation}\label{EqDifference1}\partial_t  v^*=L_{\Sigma} v^*+\mathcal Q (v^*,t),\quad \mathrm{with}\ v^*=(\varphi_t)^*v,\ \mathrm{and}\end{equation} 
	\begin{equation}\label{EqQ}\mathcal Q(v^*,t)=P( v^*,t)+((\varphi_t)^{*}Q)(v^*)+(\varphi_t)^*(\langle \partial_t \varphi_t,\nabla v\rangle),\quad P( v^*,t)=(\varphi_t)^*(L_{M_t} v)-L_\Sigma v^*.\end{equation} We emphasize that \eqref{EqDifference1} is the equation satisfied by the restriction of $v^*=(\varphi_t)^*v$ to the graphical part $\Sigma^{\mathbf r(M_t)} $ where $v$ satisfies \eqref{EqDifference2}. Equation \eqref{EqDifference1} itself is not an autonomous evolutionary equation since the solution over the region $ \Sigma^{\mathbf r(M_t)}$ is influenced by the part outside $ B_{\mathbf r(M_t)}$ according to \eqref{EqDifference2}. Among the three terms of $\mathcal Q$ in \eqref{EqQ}, the error term $P$ is estimated in Appendix \ref{AppP} (see Lemma \ref{LmAppP}). The term $(\varphi_t)^*Q$ is bounded by $Q$ and smallness of $\varphi_t$, and the term $(\varphi_t)^*(\langle \partial_t \varphi_t,\nabla v\rangle)$ is bounded by the smallness of $\varphi_t$ and $\partial_t\varphi_t$. We refer the reader to Appendix \ref{AppP} as well.
	
	Similar to Proposition \ref{PropPL}, we have the following estimate for the perturbed RMCF. 
	
	\begin{proposition}\label{PropPL1}
		For any large $r>0$ and small $\dt_1>0$, there exists $\epsilon_2>0$ with the following significance. Suppose $\widetilde{M}_t$ is a perturbed RMCF where $\widetilde{M}_0=\{x+v_0(x)\bn(x)\ |\ x\in M_0\}$. If $\|v_0\|_{C^{2,\al}}<\epsilon_2$, then there exists $T>0$ such that $\widetilde{M}_{t}\cap B_{\mathbf{r}(t)}$ can be written as a graph of the function $u(\cdot,t)$ on $\Sigma^{\mathbf{r}(t)}$ with $\|u(\cdot,t)\|_{C^2(\Sigma^r)}<\dt_1$ on $\Sigma^r$ for $t\in[T,T+1]$.
	\end{proposition}
	
	\begin{proof}
		$M_t$ converging to $\Sigma$ in $C_{\loc}^\infty$ sense implies that for any fixed $r$, when $T$ is sufficiently large, $M_t$ can be written as a small graph over $\Sigma^r$. If the initial perturbation is sufficiently small, then by the smooth dependence on the initial data of the RMCF equation, we have $\widetilde{M}_T$ can be written as a graph over $M_T$ which is sufficiently small. Then Theorem \ref{Lem;transplantation} implies that $\widetilde{M}_T$ can be written as a small graph over $\Sigma^r$. The fact that $\widetilde{M}_{t}\cap B_{\mathbf{r}(t)}$ can be written as the graph of $u$ over $\Sigma^{\mathbf{r}(t)}$ follows from  the following Proposition \ref{LmT}(2) and Theorem \ref{Lem;transplantation}. \end{proof}
	
	\subsection{The cone condition to initiate the local dynamics}\label{SSInitiateCone}
	The main result of this subsection is the following proposition which gives the cone condition and the estimate of the boundary behavior for a small initial positive perturbation. 
	
	\begin{proposition} \label{LmT} Let $\dt>0$ be a small number and $r$ be a large number, then there exist $T_\sharp$ and $\eps_1$ with the following significance:  
		\begin{enumerate}
			\item For $t\geq T_\sharp$, $M_{t}$ can be written as the graph of a function $m(\cdot,t):\ \Sigma^{\mathbf r(t)}\to \R$ with $\|m(\cdot,t)\|_{C^{2,\al}(\Sigma^r)}\leq \dt/2;$ 
			\item For all $0<\eps<\eps_1$, suppose $\widetilde M_t$  is the RMCF written as the normal graph of $v(\cdot,t)$ over $M_t$ as in \eqref{EqGraphRMCF} and with the initial condition $v_0>0$ and $\|v_0\|_{C^{2,\al}(M_0)}\leq \eps$. Let $v^*$ be the transplantation of $v$ to $\Sigma^{\mathbf{r}(t)}$ which satisfies \eqref{EqDifference1}. Then at the time $T_\sharp$, the transplanted function $v^*(\cdot,T_\sharp)$ is defined on $\Sigma^{\mathbf{r}(T_\sharp)}$, and
			$$\|v^*(T_\sharp)\chi_{\mathbf{r}({T_\sharp})}\|_Q\geq  \sup_{\substack{ t'\in[T_{\sharp}-1,T_{\sharp}+1]}}
			\|v^*(t',\cdot)\|_{C^2(\mathbb A_{\mathbf r({t'})-1,\mathbf r({t'})})} \mathrm{e}^{-\frac{{\mathbf r({t'})^{3/2}}}{4}}; $$
			\item $\frac{\left|\left\langle v^*(T_\sharp)\chi_{\br(T_\sharp)},\phi_1\right\rangle_{L^2(\Sigma)}\right|}{\|v^*(T_\sharp)\chi_{\br(T_\sharp)}\|_{Q(\Sigma)}}>c/3>0$, where $c$ is in Proposition \ref{thm:perturbedRMCFto1st}. 
		\end{enumerate}
	\end{proposition}
	\begin{proof}
		The item (1) follows from the definition of the graphical scale and $\mathbf{r}(t)$. We only need to choose $\eps_0$ in Definition \ref{DefGraphical} to be $\dt/2$. 
		
		Next we prove the item (2). We first prove the desired inequality for $v^\star$, the solution to the linearized RMCF equation. From Theorem \ref{thm:vstarcutsmall} and Corollary \ref{cor:vstarcutsmall} we know that there exist $R_0>0$ and a sequence of $t_i$ such that the $Q_{R_0}({t_i})$-norm of $v^\star$ dominates the $L^2({t_i})$-norm of $v^\star$, which grows exponentially by Theorem \ref{mainthm-FK} (1). On the other hand, the parabolic maximum principle shows that $v^\star$ grows exponentially in time, and our discussion of the time range assures that $v^\star$ is bounded in space (in fact $\|u\|_{}C^{2,\alpha}$ is bounded in space); meanwhile $\mathrm{e}^{-{\mathbf r({t})^{3/2}}}$ decays superexponentially since $\mathbf r({t})$ grows exponentially by Proposition \ref{PropPL}. So we can choose a $T_\sharp$ such that
		$$\|v^\star(T_\sharp)\chi_{R_0}\|_{Q(t_\sharp)}\geq  
		\frac12 \sup_{\substack{ t'\in[T_{\sharp}-1,T_{\sharp}+1] }} \|v^\star(t',\cdot)\|_{C^2(\mathbb A_{\mathbf r({t'})-3,\mathbf r({t'})+2})}
		\mathrm{e}^{-\frac{{\mathbf r({t'})^{3/2}}}{4}}.$$
		
		Now we fix such $T_\sharp$. From Proposition \ref{prop:approx}, namely the approximation of $v^\star$ and $v$, we know that if $\eps$ is sufficiently small, $v$ and $v^\star$ will be sufficiently close to each other, and thus
		$$\|v(T_\sharp)\chi_{R_0}\|_{Q(t_\sharp)}\geq  
		\frac23 \sup_{\substack{ t'\in[T_{\sharp}-1,T_{\sharp}+1]}}
		\|v(t',\cdot)\|_{C^2(\mathbb A_{\mathbf r({t'})-2,\mathbf r({t'})+1})}\mathrm{e}^{-\frac{{\mathbf r({t'})^{3/2}}}{4}}.$$
		Finally, Lemma \ref{Lem:transplantation-cone} implies that after the transplantation, we will have 
		$$\|v^*(T_\sharp)\chi_{R_0}\|_{Q(\Sigma)}\geq  
		\sup_{\substack{ t'\in[T_{\sharp}-1,T_{\sharp}+1]}}
		\|v^*(t',\cdot)\|_{C^2(\mathbb A_{\mathbf r({t'})-1,\mathbf r({t'})})}\mathrm{e}^{-\frac{{\mathbf r({t'})^{3/2}}}{4}}.$$
		When $T_\sharp$ is sufficiently large, $\br({T_\sharp})>R_0$. Then we obtain item (2).

		Item (3) follows from Theorem \ref{thm:vstarcutsmall} and Corollary \ref{cor:vstarcutsmall}. Notice that the choice of $T_\sharp$ satisfies $\frac{\left|\left\langle v^\star(T_\sharp)\chi_{R_0},\phi_1(T_\sharp)\right\rangle_{L^2({T_\sharp})}\right|}{\|v^\star(T_\sharp)\|_{Q({T_{\sharp}})}}>c$. Meanwhile, from Theorem \ref{thm:conv1steigenfun}, $\phi_1(t_i)\to\phi_1$ on $B_R$. So when $T_\sharp$ is chosen sufficiently large, we have $\frac{\left|\left\langle v^\star(T_\sharp)\chi_{R_0},\phi_1\right\rangle_{L^2(\Sigma)}\right|}{\|v^\star(T_\sharp)\chi_{\br(T_\sharp)}\|_{Q(\Sigma)}}>\frac{1}{2}c$ applying \eqref{EqNormCompare}. Again, if $\eps$ is sufficiently small, Proposition \ref{thm:perturbedRMCFto1st} and Lemma \ref{Lem:transplantation-cone} imply that $\frac{\left|\left\langle v^*(T_\sharp)\chi_{\br(T_\sharp)},\phi_1\right\rangle_{L^2(\Sigma)}\right|}{\|v^*(T_\sharp)\chi_{\br(T_\sharp)}\|_{Q(\Sigma)}}>\frac{1}{3}c$.

	\end{proof}
	
	\subsection{Approximating the local dynamics by the linear equation}\label{SSApproximate}
	We next consider the orbit of $\widetilde M_t$ for $t\geq T_\sharp$.  We will compare $v^*(t)$ with the solution of the autonomous equation $\partial_t v= L_\Sigma v$. Note that the linear equation is globally defined on $\Sigma$, but $M_t$ is not a global graph, so we choose the initial condition $v(0)$ for the linear equation satisfying $v(0)=v^*(T_\sharp+n)\chi_R$, $R=\mathbf r({T_\sharp+n})$, where $\chi_R:\ \R\to \R$ is a smooth function that is 1 for $|x|\leq R-1$, $0$ for $|x|>R$ and $|\chi'_R|<2$. Then we solve the initial value problem
	\begin{equation}\label{EqLinearAuto}
		\begin{cases}
			\partial_t v_n&= L_\Sigma v_n\\
			v_n(0)&=v^*(T_\sharp+n)\chi_{\mathbf r({T_\sharp+n})}.\\
		\end{cases}
	\end{equation}
	The following lemma shows that the solution of the above equation \eqref{EqDifference1} can be well-approximated by the linearized equation. 
	
	\begin{proposition}\label{LmCM3}
		Let $v^*$, $v_n$ and $T_\sharp$ be as defined above. Suppose that $\|v^*(T_\sharp+n+t)\|_{C^{2,\al}(\Sigma^{\mathbf r({T_\sharp+n+t})})}<\dt $ for $t\in [0,1]$. Then we get 
		\begin{enumerate}
			\item $\|v^*(T_\sharp+n+1)\chi_{\mathbf r({T_\sharp+n+1})}-v_n(1)\|_{Q}\leq \dt\|v^*(T_\sharp+n)\chi_{\mathbf r({T_\sharp+n})}\|_{Q};$
			\item $\displaystyle\|v^*(T_\sharp)\chi_{\mathbf{r}({T_\sharp})}\|_Q\geq  \sup_{\substack{ t'\in[T_{\sharp}-1,T_{\sharp}+1]}}
			\|v^*(t',\cdot)\|_{C^2(\mathbb A_{\mathbf r({t'})-1,\mathbf r({t'})})} \mathrm{e}^{-{\mathbf r({t'})^{3/2}}}. $
		\end{enumerate}
		
	\end{proposition}
	We postpone the proof to Appendix \ref{AppPropCM}. The proof is similar to that of Proposition 4.3 of \cite{CM2}. The main difficulty created by the noncompactness is that the boundary terms behave badly when integration by parts is performed. The key observation is that the graphical domain grows exponentially with respect to $t$ (Proposition \ref{PropPL}), so the Gaussian weight decays superexponentially like $\mathrm{e}^{-c\mathrm{e}^{t/2}}$, and on the other hand, the difference between two nearby RMCFs grows at most exponentially (Proposition \ref{PropPL1}). Thus we use item (2) in the definition of $T_\sharp$ in Proposition \ref{LmT} to absorb the boundary term. 
	\subsection{Iterating the local dynamics}\label{SS:The cone preservation property}
	
	Let $E$ be the Banach space of $C^{2,\al}$ functions on $\Sigma$. We consider an orthogonal decomposition $E=E_1\oplus E_2$ with respect to the $Q$-norm, where $E_1=\R \phi_1$ and $E_2$ is the orthogonal complement of $E_1$. Next, let $\kappa$ be a positive constant. We define the cone
	$$ \mathcal K(\kappa)=\left\{ u=(u_1,u_2)\in E_1\oplus E_2\ |\ \|u_1\|_Q\geq \kappa\|u_2\|_Q\right\}.$$
	This is a cone containing $E_1=\R\phi_1$ and larger $\kappa$ implies a narrower cone. 
	
	Proposition \ref{LmCM3} implies the following cone preservation property. 
	
	\begin{lemma} \label{LmSchlag} 
		Let $\kappa>0$ be a fixed number. For all $\eps$, there exists $\dt$ sufficiently small such that the following holds. 
		Let $v^*$ be as in  \eqref{EqDifference1}  with $v^*(T_\sharp)\chi_{\mathbf r({T_\sharp})}:=(v_1(0),v_2(0))\in \mathcal K(\kappa)$ and $\| v^*(T_\sharp+t)\|_{C^{2,\al}(\mathbf r({T_\sharp+t}))}<\dt $ for $t\in [0,m]$ for any $m\in\mathbb N$,  then we have $$v^*(T_\sharp+n)\chi_{\mathbf r({T_\sharp+n})}:= (v_1(n),v_2(n))\in\mathcal K(\kappa),\quad \forall \ 0\leq n\leq m.$$ Moreover, we have
		\begin{enumerate}
			\item $\|v_1(n)\|_Q\geq \mathrm{e}^{(\lambda_1-\eps)n}\|v_1(0)\|_Q;$
			\item $\frac{\|v_2(n+1)\|_Q}{\|v_1(n+1)\|_Q}\leq \mathrm{e}^{(\lambda_2-\lambda_1+\eps)}\frac{\|v_2(n)\|_Q}{\|v_1(n)\|_Q}+\dt\left(\frac{\kappa+1}{\kappa}\right)$. 
		\end{enumerate}
	\end{lemma}
	\begin{proof}
		We have the following from  Proposition \ref{LmCM3} and the assumption  $(v_1(0),v_2(0))\in \mathcal K(\kappa)$
		$$\|v_2(1)\|_Q\leq \mathrm{e}^{(\lambda_2+\eps/3)}\|v_2(0)\|_Q+\dt\left(\frac{1+\kappa}{\kappa}\right)\|v_1(0)\|_Q,$$$$\|v_1(1)\|_Q\geq \mathrm{e}^{(\lambda_1-\eps/3)}\|v_1(0)\|_Q-\dt\left(\frac{1+\kappa}{\kappa}\right)\|v_1(0)\|_Q.$$ 
		Taking the quotient, we get both items with $n=1$, which implies $(v_1(1),v_2(1))\in \mathcal K(\kappa)$. Then the lemma follows from iterations.  
	\end{proof}

	With this lemma, we prove the following result.
	
	\begin{proposition}\label{ThmDynamics}
		Let $\dt,\eps $ and $T_\sharp$ be as in Proposition \ref{LmT}. Then there exists a time $T_\dagger\ (>T_\sharp)$ of order $|\log (\dt/\eps)|$ such that after the evolution of the RMCF for the time $T_\dagger$, the function $ v^*(T_\dagger,\cdot)$ is defined over $\Sigma^{\mathbf  r({T_\dagger})}$ satisfying  
		\begin{enumerate}
			\item $\| v^*(T_\dagger,\cdot)\chi_{\mathbf  r({T_\dagger})}\|_{C^{2,\al/2}}=\dt;$ 
			\item $\| v^*(T_\dagger,\cdot )\chi_{\mathbf  r({T_\dagger})}\|_{Q}\geq\dt^{d}$, for some $d>0$ independent of $\dt$;
			\item $v^*(T_\dagger,\cdot)\chi_{\mathbf  r({T_\dagger})}\in \mathcal K(1/(C\dt)). $
		\end{enumerate}
	\end{proposition}
	\begin{proof}
		The conclusions (1) and (2) will be proved in Section \ref{SSCQ}. Item (3) follows from Lemma \ref{LmSchlag}(2). Indeed, the inequality in Lemma \ref{LmSchlag}(2) gives that for large time $n$, the ratio estimate $\frac{\|v_2(n+1)\|_Q}{\|v_1(n+1)\|_Q}<C\dt$ stabilizes. This translates to the cone condition in item (3). 
	\end{proof}

	\subsection{Proof of the main theorem}\label{SSProofMain}
	
	Now we prove the following theorem, which gives the cone condition for the perturbed RMCF.  
	\begin{theorem}\label{ThmMainSec4}
		Let $\mathbf{M}_t$ and $M_t$ be as in $(\star)$. There exist $\delta_0>0$, $\delta_1>0$ and $r>0$ with the following significance: after a small initial positive perturbation on the initial data $M_0$, there is a time $T'$ such that  the perturbed RMCF $\widetilde{M}_{T'}$ can be written as a graph of the function $u$ on $\Sigma^{\mathbf{r}({T'})}$, and decomposing $u=u_1+u_2$ in $E=E_1\oplus E_2$ we have
		\begin{enumerate}
			\item $\|u\|_{Q}^{1/d}\geq C \|u\|_{C^{2,\alpha}(\Sigma\cap B_r)}\approx\delta_0$, where $d$ is in Proposition \ref{ThmDynamics};
			\item $\|u_1\|_Q\geq C\dt_0^{-1}\|u_2\|_Q$ for some constant $C$ independent of $\delta_0,\delta_1$ or $r$.
		\end{enumerate}
	\end{theorem}
	
	\begin{proof}
		
		\textbf{Step 1:} Let us fix an integer $\ell\geq 4$. At first we will choose a large radius $r$ such that $\Sigma\backslash B_r$ is very close to regular cones, say,  $\Sigma$ can be written as a graph over these cones with the $C^\ell$-norm of the graph is less than $10^{-10}\delta_1$ (See \cite[Section 2]{CS} for details). We first choose $T$ sufficiently large such that $M_T$ is a graph of a function $f$ over $\Sigma\cap B_r$ with $\|f\|_{C^{2,\alpha}(\Sigma^r)}\leq \delta_1$, and we also assume $T$ is sufficiently large such that item (3) in Theorem \ref{mainthm-FK} holds for $T$ (we choose $R(t)\equiv r$). We will also choose $T$ sufficiently large such that after the transplantation, $\overline{\phi_1^r}(T)$ is very close to $\phi_1$ on $\Sigma$, in the sense that $\|\overline{\phi_1^r}(T)-\phi_1\|_{C^{2,\alpha}(\Sigma^r)}\leq \delta_1$. In the definition of $\mathbf{r}$ (see Proposition \ref{PropPL}), we can choose $C$ so that $\mathbf{r}(T)=r$. We will fix $\delta_0$ and $\delta_1$ later, but they are both very small constants. 
		
		Let $u_0>0$ be a positive pertubation solution on $M_0$. We pick $\epsilon$ small, to be determined. Let $u$ be the solution to the perturbed RMCF on $M_t$ and $u^\star$ be the solution to the linearized RMCF, both with the initial data $\epsilon u_0$. By the approximation of the RMCF equation and the linearized RMCF, $$\|u(\cdot,T)-u^\star(\cdot,T)\|_{C^{2,\alpha}(M_T)}\leq C\delta_1^\eps \|u(\cdot,T)\|_{C^{2,\alpha}(M_T)}.$$ 
		
		\textbf{Step 2:}  By Theorem  \ref{mainthm-FK}, we have
		$ \frac{|\langle u^\star(T),\phi^{r}_1(T)\rangle_{L^2(M^{r}_{T})}|}{\|u^\star(t_i)\|_{Q(M^{r}_{T})}}>c'.$
		Then when $\delta_1$ is sufficiently small, we can use the triangle inequality to get
		$  \frac{|\langle u(T),\phi^{r}_1(T)\rangle_{L^2(M^{r}_{T})}|}{\|u(T)\|_{Q(M^{r}_{T})}}>c',$
		where $C'$ and $c'$ are some constants that may vary from line to line. Let $\overline{u}(T)$ be the transplantation of $u(T)$ to $\Sigma^R$. Then by Lemma C.1 in \cite{SX}, we also have
		$\frac{|\langle \overline u(T),\overline{\phi^{r}_1}(T)\rangle_{L^2(M^{r}_{T})}|}{\|\overline u(T)\|_{Q(M^{r}_{T})}}>c'.$
		Finally, because we have assumed that $T$ is sufficiently large so that $\|\overline{\phi_1^r}(T)-\phi_1\|_{L^2}\leq \delta_1$,  we have	$\frac{|\langle \overline u(T),\phi_1\rangle_{L^2(M^{r}_{T})}|}{\|\overline u(T)\|_{Q(M^{r}_{T})}}>c',$ which implies that $\overline{u}(T)$ lies in a cone $\cK(\kappa)$ for a $\kappa$ depending on $\delta_1$ (see Section \ref{SS:The cone preservation property}).
		
		{\bf Step 3:} 
		Now we can use Proposition \ref{ThmDynamics} to show that after a definite amount of time $T_\dagger\approx \log|\delta_0/\|\overline{u}\chi_r\|_{Q}|$, we have 
		\begin{itemize}
			\item $\|\overline{u}\chi_{\mathbf{r}({T_\dagger})}\|_{C^{2,\alpha}(B_r\cap\Sigma)}\approx\delta_0$, $\|\overline{u}\chi_{\mathbf{r}({T_\dagger})}\|_{Q}\approx\delta_0^{\alpha_l}$, 
			\item $\overline{u}\chi_{\mathbf{r}({T_\dagger})}\in\cK(1/(C\delta_0))$.
		\end{itemize}
		This implies the desired bound on $\overline{u}(T_\dagger)$.
		
		{\bf Step 4:}
		It remains to show that if we write the perturbed flow as a graph of the function $v$ over $\Sigma\cap B_{\mathbf{r}({T_\dagger})}$, $v$ satisfies the same bound as $\overline{u}$. This is a consequence of the transplantation bound.

	\end{proof}
	
	Theorem \ref{ThmMainSec4} can be used to describe the local feature of the MCF after a positive perturbation. In particular, we have the following local description of the MCF near a singularity modeled by an asymptotic conical shrinker: under the assumption $(\star)$, after a sufficiently small positive perturbation on the initial data, there is a spacetime neighbourhood $\cN$ of $(0,0)$ in which there is no singularity modeled by $\Sigma$.
	
	Recall that the positive perturbation on a non-generic shrinker is not the only possible direction of the perturbation that can decrease the Gaussian area of a non-generic shrinker. Colding-Minicozzi \cite{CM1} have proved that infinitesimal translations (corresponding to eigenfunction $\langle x,\vec e\rangle$ for any unit vector $\vec{e}$) and infinitesimal dilations (corresponding to the eigenfunction $H$) also decrease the Gaussian area of a non-generic shrinker. 
	
	For $\{\mathbf{M}_\tau\}$, the translations and dilation on the initial data will move the singularity $(0,0)$ in the spacetime to somewhere else. We can show that the positive perturbation can perturb the asymptotic conical singularity better than translations and dilation. We are ready to give the proof of the conclusion of the main Theorem \ref{thm:main1} for positive initial perturbations.

	\begin{proof}[Proof of Theorem \ref{thm:main1} for positive initial perturbations] Using the volume growth estimate $\frac{\mathrm{Vol}(\widetilde M_t\cap B_R)}{R^n}\leq C$ (for all $t$ large and all $R$) for RMCFs (c.f. Lemma 2.9 of \cite{CM1}), we get the estimate $\cF(M_t\backslash B_R)\leq CR^n \mathrm{e}^{-R^2/4}$. 
		
		Fix $\dt$ ahead of time and let   $T_\dagger$ be as in Proposition \ref{ThmDynamics}. By choosing the $C^{2,\al}$ norm $\eps$ of the initial perturbation sufficiently small, we may allow $T_\dagger$ sufficiently large and we pick $R=\mathbf r(M_{T_\dagger})\simeq \mathrm{e}^{T_\dagger/2}$, so that $$\cF(M_{T_\dagger}\backslash B_R)\leq CR^n \mathrm{e}^{-R^2/4}<\dt^{3},$$  and fix it in the following. Next we write $\widetilde M_{T_\dagger}^R$ as a graph of a function $u^*$ over $\Sigma^R$, and suppose $\|u^*(T_\dagger,\cdot )\|_{Q}=\dt^{d'}$ with $d'\leq d$ by Proposition \ref{ThmDynamics}(2). Here we can guarantee that $d'>1$ is sufficiently close to 1 (see the interpolation formula in Section \ref{SSCQ}). 
		
		By the second variation formula, we get 
		\begin{equation}\label{Eq2ndVariation}\cF(\widetilde M_{T_\dagger}^R)-\cF(\Sigma)=-\int_{\Sigma^R} u^* L_\Sigma  u^* \mathrm{e}^{-\frac{|x|^2}{4}}+O(\dt^{3}).\end{equation}
		By Theorem \ref{ThmMainSec4}(2), we get that the projection of $u^*$ to the $\phi_1$ component dominates, so we get 
		$$\cF(\widetilde M_{T_\dagger}^R)-\cF(\Sigma)\leq -(\lambda_1-O(\dt))\dt^{2d'}+O(\dt^3).$$
		Lemma 7.10 of \cite{CM1} shows that the entropy of $\Sigma$ is attained, so we have $\cF(\Sigma)=\lambda(\Sigma)$. 
		It was proved in \cite{CM1} (c.f. Theorem 4.30 and 4.31 of \cite{CM1}) that perturbation in the direction of $\phi_1$ strictly decreases the $F$-functional, despite of translation or dilations. So if we have a translation or dilation $\mathcal R$ of size $\dt$, the same calculation as Theorem 4.30 of \cite{CM1} gives the estimate $\cF(\mathcal R\widetilde M_{T_\dagger}^R)-\cF(\Sigma)\leq -\dt^{2.5}$. 
		
		This completes the proof. 
		
	\end{proof}
	
	\begin{proof}[Proof of Theorem \ref{thm:MainThmLocal}]
		The proof is the same as the proof of Theorem \ref{thm:main1}, but just uses the definition of the local entropy. We omit the proof here.
	\end{proof}
	
	Similarly, we have Corollary \ref{CorMain} holds for positive initial perturbations. 
	
	\begin{proof}[Proof of Corollary \ref{CorMain}  for positive initial perturbations]
		
		The proof is a straightforward consequence of Theorem \ref{thm:MainThmLocal} and the monotonicity formula of the local entropy (see \cite{Su} for a detailed discussion of the monotonicity formula of local entropy).
	\end{proof}
	\subsection{Generic perturbations}\label{SSGeneric}
	In this subsection, we complete the proofs of Theorem \ref{thm:main1} and Corollary \ref{CorMain} by allowing generic initial perturbations that are not necessarily positive.

	\begin{proof}[Completing the proofs of Theorem \ref{thm:main1} and Corollary \ref{CorMain}]
		We only need to prove Theorem \ref{ThmMainSec4} holds. The proof is similar to \cite[Theorem 3.11]{SX} so we only sketch the proof here. The key is to prove Step 2 in Theorem \ref{ThmMainSec4} for an open dense subset $\cS$ of $\{u\in C^{2,\alpha}(M_0)\ |\ \|u\|_{C^{2,\alpha}}=1\}$.
		
		Let $\cS$ be the subset of $\{u\in C^{2,\alpha}(M_0)\ |\ \|u\|_{C^{2,\alpha}}=1\}$ such that for any $u_0\in \cS$ Step 2 in Theorem \ref{ThmMainSec4} holds. By the well-posedness of the RMCF, the openness of $\cS$ is straightforward. So we only need to prove the denseness of $\cS$. In fact, for any $u_0$ with $\|u_0\|_{C^{2,\alpha}(M_0)}=1$, let $u$ be the solution to the linearized RMCF equation with the initial data $u_0$, there are two possible growth rates of $u$:
		
		Case 1: $\|u\|_{L^2(M_t)}$ grows faster than $\mathrm{e}^{(\lambda_1(\Sigma)-\epsilon)t}$ for any $\epsilon>0$. Then the proof of Theorem \ref{mainthm-FK} implies that Theorem \ref{mainthm-FK} also holds for such $u$. This implies that such $u_0\in\cS$.
		
		Case 2: $\|u\|_{L^2(M_t)}$ grows slower than $\mathrm{e}^{(\lambda_1(\Sigma)-\epsilon)t}$ for some $\epsilon>0$. Then we can add a small positive function to $u_0$, and normalize it to get a nearby initial condition $u_0'$. Because the positive function grows faster than $\mathrm{e}^{(\lambda_1(\Sigma)-\epsilon)t}$ for any $\epsilon>0$, it will dominate the whole function. Thus $u_0'\in\cS$. 
		
		Combining both cases above we know that $\cS$ is dense. 
	\end{proof}
	\subsection{Ancient solution}\label{SSAncient}
	In the proof of the main theorem, we use only the finite time dynamics without establishing a stable/unstable manifold theorem. The existence of the stable manifold for noncompact shrinkers is an interesting open problem with the main difficulty being again the failure of writing $M_t$ as a graph over $\Sigma$. A similar argument as Theorem 4.1 of \cite{SX} gives an ancient solution dictated by Theorem \ref{thm:ancient}. We refer the reader to \cite{SX} for more details.

	\section{Graphical estimates near a shrinker}\label{SEstimate}
	In this section, we give necessary graphical estimates for an RMCF close to a shrinker on a large compact set. The proofs of Proposition \ref{PropPL} and Theorem \ref{ThmMainSec4}(1) will be completed in this section. We shall mainly solve two problems in this section. The first one is the estimate of the graphical radius, i.e., Proposition \ref{PropPL}, which will be used to control the boundary term due to the cutoff. The second one is to give a bound on the $C^{2,\al}$ norm of the difference function $v^*$ between $M_t$ and $\widetilde M_t$, which is needed in the assumption of Proposition \ref{LmCM3} and the sequel. 
	
	This section is organized as follows. In Section \ref{SSPL}, we use pseudolocality to complete the proof of Proposition \ref{PropPL}. In Section \ref{SSHolder}, we estimate the $C^{2,\al}$ norm of the perturbed flow $\widetilde{M}_t$. In Section \ref{SSCQ}, we bound the  $C^{2,\al}$ norm by the $Q$-norm using an interpolation argument as well as the higher derivative estimate by Ecker-Huisken, and hence completing the proof of Theorem \ref{ThmMainSec4}(1). 
	
	\subsection{Pseudolocality}\label{SSPL}
	In order to handle the noncompactness of the limit shrinker $\Sigma$, we use a pseudolocality lemma to extend the graphical scale of $M_t$. We use the formulation of \cite[Proposition 5.1]{CS}, which is reformulated from Theorem 1.5 in \cite{INS}. 
	
	\begin{lemma}\label{LmPseudoMCF}
		Given $\delta>0$, there exist $\gamma>0$ and a constant $\rho=\rho(n,\delta)>0$ such that for $x\in\R^{n+1}$, if an MCF $\{\mathbf M_t\}_{t\in[0,1]}$ satisfies that $\mathbf M_{-1}\cap B_\rho(x)$ is a Lipschitz graph over the plane $\{x_{n+1}=0\}$ with the Lipschitz constant less than $\gamma$ and $0\in \mathbf M_{-1}$, then $\mathbf M_t\cap B_\rho(x)$ intersects $B_\delta(x)$ and remains a Lipschitz graph over $\{x_{n+1}=0\}\cap B_\delta(x)$ with Lipschitz constant less than $\delta$ for all $t\in[-1,0]$.
	\end{lemma}
	
	We next reformulate this lemma in the context of RMCFs. Note that the Lipschitz constant of a graph is invariant under dilation and that an RMCF $(M_t),\ t\in[0,\infty)$ is related to an MCF $(\mathbf M_{\tau}),\ \tau\in [-1,0]$ through ${M}_t=\mathrm{e}^{\frac{t}{2}}\mathbf M_{-\mathrm{e}^{-t}}$. 
	
	\begin{lemma}\label{LmPseudoRMCF}
		Given $\delta>0$, there exist $\gamma>0$ and a constant $\rho=\rho(n,\delta)>0$ such that if an RMCF $\{M_t\}_{t\in[T,\infty)}$ satisfies that $M_{T}\cap B_{\sqrt{\mathrm{e}^T}\rho}(x)$ is a Lipschitz graph over the plane $L$ passing $x$ with Lipschitz constant less than $\gamma$ and $x\in M_{T}$, then $M_t \cap B_{\sqrt{\mathrm{e}^t}\rho}(\sqrt{\mathrm{e}^{t-T}}x)$ intersects $B_{\sqrt{\mathrm{e}^t}\delta}(\sqrt{\mathrm{e}^{t-T}}x)$ and remains a Lipschitz graph over $L\cap B_{\sqrt{\mathrm{e}^t}\delta}(\sqrt{\mathrm{e}^{t-T}}x)$ with Lipschitz constant less than $\delta$ for all $t\in[T,\infty)$.
	\end{lemma}
	
	We will use Lemma \ref{LmPseudoRMCF} in the following way to control the gradient of graphs. 
	
	\begin{corollary}\label{CorPseudo}
		Suppose $\Sigma$ is a conical shrinker, and $M_t$ is an RMCF. Let $R$ be sufficiently large. Then when $T$ is sufficiently large, suppose $M_T$ is a graph of $u(\cdot,T)$ on $\mathbb A_{R-1,R}\cap\Sigma$ with $|\nabla u(\cdot,T)|\leq \gamma$, then $M_{T+t}$ is a graph of $u(\cdot,T+t)$ on $\mathbb A_{\mathrm{e}^{t/2} (R-1),\mathrm{e}^{t/2}R}\cap\Sigma$ with $|\nabla u(\cdot,t)|\leq \delta$ for $t>0$.
	\end{corollary}
	\begin{proof}
		Suppose at the time $T$ very large, $M_T$ is a graph of the function $m(\cdot, T)$ over $\Sigma$ in the ball of radius $R$, such that $|\nabla m(\cdot, T)|\leq \gamma/3$. Then we know that $M_T$ is a graph on the tangent space $T_x\Sigma$ with Lipschitz constant less than $\gamma/2$. Then we can decompose $B_R\backslash B_{R-1}$ into several small balls $B^j_{\sqrt{\mathrm{e}^{-T}}\rho}(x)$ with radius $\sqrt{\mathrm{e}^{-T}}\rho$. On each $B^j(x)$ we use pseudolocality Lemma \ref{LmPseudoRMCF}, we see that $M_{T+t}\cap \sqrt{\mathrm{e}^{T+t}}B_{\mathrm{e}^{t/2}\delta}^j(\mathrm{e}^{t/2}x)$ is a graph on the same plane with Lipschitz constant $\delta$ for $t>0$.
		
		If we assume $\Sigma$ is conical, then the tangent space $T_x\Sigma$ should be very close to the tangent space of $\Sigma$ near $\mathrm{e}^{t/2}x$. Moreover, this Lipschitz constant can be translated to the gradient of $M_{T+t}$ as a graph $m_{T+t}$ over the ball of radius $\mathrm{e}^{t/2}R$ roughly. So we obtain the corollary.
	\end{proof}
	
	The above results concern the Lipschitz bound of the graph. Next we use Ecker-Huisken's interior estimate to get a higher-order estimate of the graph. The following lemma is a consequence of Ecker-Huisken's interior estimate of MCFs (c.f. \cite[Section 4]{EH}).
	
	\begin{lemma}\label{LmPseudoMCF-high}
		Under the assumptions of Lemma \ref{LmPseudoMCF}, for any integer $\ell>0$, there exists a constant $C_\ell$ depending on $\delta,\gamma$ and $\rho$, such that for all $t\in[-1/2,0]$, we have that $\mathbf{M}_t\cap B_\dt$ is a $C^\ell$ graph with the $C^\ell$-norm less than $C_\ell$.
	\end{lemma} 
	
	Similarly, this lemma has a reformulation for RMCFs and a consequence for RMCF.
	
	\begin{lemma}\label{LmPseudoRMCF-high}
		Given $\delta>0$, there exist $\gamma>0$ and a constant $\rho=\rho(n,\delta)>0$ such that if an RMCF $\{M_t\}_{t\in[T,\infty)}$ satisfies that $M_{T}\cap B_{\sqrt{\mathrm{e}^T}\rho}(x)$ is a Lipschitz graph over the plane $L$ passing $x$ with Lipschitz constant less than $\gamma$ and $x\in M_{-1}$, then for any integer $\ell>0$, $M_t \cap B_{\sqrt{\mathrm{e}^t}\rho}(\sqrt{\mathrm{e}^{t-T}}x)$ intersects $B_{\sqrt{\mathrm{e}^t}\delta}(\sqrt{\mathrm{e}^{t-T}}x)$ and is a $C^\ell$ graph with the $\ell$-th derivative is bounded by $C_\ell(\sqrt{\mathrm{e}^{t-T}})^{-\ell+1}$ for all $t\in[T+1/2,\infty)$, where $C_\ell$ is a cosntant depending on $\delta,\gamma$ and $\rho$.
	\end{lemma}
	\begin{corollary}\label{CorPseudo-high}
		Suppose $\Sigma$ is a conical shrinker, and $M_t$ is an RMCF converging to $\Sigma$. Let $R$ be sufficiently large, and suppose $M_T$ is a graph of $u(\cdot,T)$ on $\mathbb A_{R-1,R}\cap\Sigma$ with $|\nabla u(\cdot,T)|\leq \gamma$ for some   $T$ sufficiently large. Then for any integer $\ell>0$, $M_{T+t}$ is a graph of $u(\cdot,T+t)$ on $\mathbb A_{\mathrm{e}^{t/2} (R-1),\mathrm{e}^{t/2}R}\cap\Sigma$ with $\|\nabla^\ell u(\cdot,T+t)\|_{C^0}\leq (\mathrm{e}^{t/2})^{-\ell+1}C_\ell$ for $t>1/2$.
	\end{corollary}
	
	\begin{remark}
		We give a remark about the Corollary \ref{CorPseudo} and Corollary \ref{CorPseudo-high}. Note that in Lemma \ref{LmPseudoRMCF} and Lemma \ref{LmPseudoRMCF-high}, the statement is about $M_t$ being a graph over a fixed tangent space. Meanwhile in Corollary \ref{CorPseudo-high}, the statement is about $M_t$ being a graph over $\Sigma$. In fact, here we use the structure of an asymptotically conical shrinker. In particular, we use the fact that the curvature decays on $\Sigma$, and there exists $R_0>0$ such that for any $R\geq R_0$ and $C>1$, any $x$ on $\partial B_R\cap \Sigma$, $T_x\Sigma$ and $T_{C^\#x}\Sigma$ is only different by a very small rotation $\theta$ (we use $C^\# x$ to denote the point on $\Sigma$ whose projection to the asymptotic cone is $C$ times the point on $\Sigma$ which is the projection of $x$). The existence of such $R_0$ is a consequence of the structure of an asymptotically conical shrinker (see \cite{BW} and \cite{CS}). 
	\end{remark}
	We next work on the proof of Proposition \ref{PropPL}. 
	
	\begin{proof}[Proof of Proposition \ref{PropPL}]
		Suppose $\{M_t\}$ is an RMCF.  Fix $R>0$ from Corollary \ref{CorPseudo-high}. We assume $T$ is sufficiently large such that $M_t$ is a graph of a function $m_t$ over $\Sigma\cap B_{4R}$ with $\|m_t\|_{C^\ell(\Sigma\cap B_{4R})}\leq \gamma$, where $\gamma$ comes from Corollary \ref{CorPseudo-high}. Such $T$ exists because $M_t$ converges to $\Sigma$ in $C_{\loc}^\infty$ sense.
		
		Corollary \ref{CorPseudo-high} implies that $M_{T+t}$ is a $C^\ell$ graph over $(\mathbb A_{R \mathrm{e}^{t/2},2R \mathrm{e}^{t/2}})\cap\Sigma$ for $t>1/2$. Meanwhile, for $t<2$, we know that $\mathbb A_{R \mathrm{e}^{t/2},2R \mathrm{e}^{t/2}}\cap B_{4R}$ is always nonempty. So for $t\in[1/2,2]$, $M_{T+t}$ is a $C^\ell$ graph over $B_{2R \mathrm{e}^{t/2}}\cap \Sigma$, with the $C^\ell$-norm satisfying the bound in Corollary \ref{CorPseudo-high}. 
		
		Now we repeat the above process, starting from $t=1,2,3,\cdots$. Then the graphical region expands at the rate $2R\mathrm{e}^{t/2}$. This concludes the proof.	
	\end{proof}

	\subsection{The regularity estimates for the perturbed flow}\label{SSHolder}
	In the last section, we see that once the RMCF is very close to $\Sigma^R$, the graphical region expands exponentially. For $M_t$ converging to $\Sigma$ in the $C^\infty_{\loc}$ sense, it is always close to $\Sigma^R$, and hence the expansion will last for all time. In our setting, the perturbed RMCF $\widetilde{M}_t$ will not be close to $\Sigma^R$ for all time. We need a quantitative characterization of the closeness of $\widetilde{M}_t$ and $\Sigma^R$.
	
	In this subsection we study some estimates of the RMCF as a graph over a part of $\Sigma$. In particular, we generalize the H\"older estimates in \cite{CM2} to the noncompact asymptotically conical shrinker setting. We will always fix a conical shrinker $\Sigma$ and $R_*$ large such that on $B_{R_*}\cap \Sigma$, there exists $C_\ell>0$ such that $|\nabla^k A|\leq C(1+|x|)^{-k-1}$ for $k=0,1,\cdots,\ell+1$, and $\Sigma\cap \partial B_{R_*}$ is the union of several submanifolds in $R_*\mathbb S^n$, which are sufficiently close to the cross sections of the asymptotic cone of the conical ends, and the tangent space on $\Sigma\cap \partial B_{R_*}$ is very close to the tangent space of the asymptotic cones. 
	
	The main goal of this subsection is to prove the following proposition:
	
	\begin{proposition}\label{propHolder}
		Given $\epsilon_1>0$, there exist $\delta_1$, $C$, $\eps$ and $\alpha>0,$ so that the following holds. Suppose $\{\widetilde M_t\}$ is an RMCF with $\mathbf{r}(\widetilde M_T)\geq R_*$ for some large $T$, and $\widetilde M_{T+t},\ t\in [0,T^*],$ can be written as the graph of a function $u$ on $\Sigma^{R_*}$ satisfying
		\begin{enumerate}
			\item	$|u(T)|+|\nabla u(T)|\leq \delta\leq \delta_1, \text{ and } |\Hess_{u(T)}|\leq \delta_1,$	
			\item $\|u(\cdot,T+t)\|_{C^{2}}\leq \delta_1$ on $\Sigma^{R_*}$ for $t\leq T^*$, 
		\end{enumerate}
		then $\widetilde M_{T+t}$ can be written as the graph of a function $u$ on $\Sigma^{\mathbf{r}({T+t})}$ for $t\in[0,T^*]$ and
		\begin{enumerate}
			\item  on $\Sigma^{\mathbf{r}( {T+t})}$ we have the estimates $$|u(x,T+t)|\leq C^t (\delta+\epsilon_1)(1+|x|),\ |\nabla u(T+t)|\leq C^t\sqrt{(\delta+\epsilon_1)},\ |\Hess_{u(T+t)}|\leq C^t\dt_1,$$
			\item on $\Sigma^{R_*}$ we have $\|u(\cdot,T+t)\|_{C^{2,\alpha}(\Sigma^{R_*})}\leq C^t (\delta+\epsilon_1)^\eps$,
			\item for any integer $\ell$, $\|u(\cdot,T+t)\|_{C^\ell(\Sigma^{R_*})}\leq C_\ell$ for some constant $C_\ell$ depending on $\delta_1$.
		\end{enumerate}
	\end{proposition}
	\begin{lemma}\label{Lem:C0bound}
		Suppose $\{\widetilde M_t\}$ is an RMCF. Given $\epsilon_1>0$, there exist $\delta_1>0$,  $\tau>0$, and $C>0$ such that the following holds. Suppose that $\widetilde M_T$ can be written as a graph of the function $u(\cdot,T)$ over $\Sigma^{r},\ r>R_*$ with $\|u(\cdot,T)\|_{C^1(\Sigma^{r})}\leq \delta\leq \delta_1$, then for $t\in[0,\tau]$, we have
		$$\sup_{\Sigma^{r}}|u(\cdot,T+t)|\leq C\sup_{\Sigma^{r}}|u(\cdot,T)|+\epsilon_1.$$
	\end{lemma}
	
	\begin{proof}
		First we prove this bound on $\mathbb A_{r-1,r}\cap\Sigma$. We use the pseudolocality argument. We choose sufficiently small $\delta_1>0$, such that we can use Corollary \ref{CorPseudo}. More precisely, we pick $\delta$ in Lemma \ref{LmPseudoRMCF} to be $\epsilon_1$, and we pick $T$ in Lemma \ref{LmPseudoRMCF} to be $0$, and we pick $\rho<\mathrm{e}^{-1}$ in Lemma \ref{LmPseudoRMCF}. Corollary \ref{CorPseudo} applies to showing that on $\mathbb A_{r-1,r}\cap\Sigma$, the RMCF $\widetilde M_t$ can be written as a graph, whose supremum is bounded by $\mathrm{e}^{t/2} \epsilon_1 \rho  \leq \epsilon_1$.
		
		Next we consider the bound on $\Sigma^{r-1}$. Let $\eta$ be the cutoff function which is $1$ on $\Sigma^{r-1}$ and $0$ outside $\Sigma^{r}$. Let us consider the supremum of $\eta u$. Suppose the supremum of $\eta u(\cdot,t)$ is attained at $p$. If $|p|\geq r-1$, then we can use the previous pseudolocality argument to prove that $\sup_{\Sigma^{r-1}}|u|\leq \epsilon_1$. Otherwise, we know that $p$ is a local maximum of $u(\cdot,t)$. Then we use the maximum principle (c.f. Lemma 4.4 of \cite{CM3}) to show that $\sup_{\Sigma^{r-1}}|u(\cdot,T+t)|\leq C\sup_{\Sigma^{r-1}}|u(\cdot,T)|+\epsilon_1$. Combining the two cases together we get the desired bound. 
	\end{proof}
	
	We next have the following local curvature estimate near a conical singularity. 
	
	\begin{lemma}\label{Lem:CurEst}
		For any $\epsilon_2>0$, there exist $\delta$, $\eta,\tau$ and $C_0$ with the following significance. Suppose $\{\mathbf {\widetilde M}_t\}_{t\in[0,T]}$ is an MCF with $\max_{x\in \mathbf {\widetilde M}_0\cap B_{r}}|A|^2(x,0)\leq C<C_0$,\ $r>R_*$, and  satisfies  
		\begin{itemize}
			\item  $\mathbf {\widetilde M}_0\cap B_\delta(x)$ is a graph over some hyperplane $L$, with Lipschitz constant less than $\eta$, for any $x\in \mathbf {\widetilde M}_0\cap \mathbb A_{r,r+1}$.
		\end{itemize}
		Then for $t\in[\tau/2,\tau]$, we have
		$\max_{x\in \mathbf {\widetilde M}_t\cap B_{r+1}}|A|^2(x)\leq 2C+2\epsilon_2.$
	\end{lemma}
	
	\begin{proof} We extend the argument of \cite[Corollary 4.6]{CM3} to the noncompact setting. 
		By the pseudolocality argument, if $\eta$ is sufficiently small, there exists $\tau>0$ such that on $\mathbf {\widetilde M}_t\cap \mathbb A_{r+1/3,r+1/2}$, when $t\in[\tau/2,\tau]$, we have $|A|^2<\epsilon_2$. Next we define 
		$m(t)=\max_{x\in \mathbf {\widetilde M}_t\cap B_{r+1/2}}|A|^2(x,t).$
		We have two cases: if $m(t)\leq \epsilon_2$ for all $t\in[\tau/2,\tau]$, then the inequality is proved. Otherwise, $m(t)$ is attained at somewhere on $\mathbf {\widetilde M}_t\cap B_{r+1/3}$. Using the inequality
		$(\partial_t-\Delta_{\mathbf {\widetilde M}_t})|A|^2\leq |A|^4,$
		we can use the interior maximum principle to show that $ m'(t)\leq m(t)^2$. This implies that 
		$m(t)\leq \min\{2C,2\epsilon_2\},$
		if $t\in[0,\tau]$, where $\tau$ is a constant depending on $C$ and $\epsilon_2$. 
	\end{proof}
	
	\begin{proposition}\label{propHolder-small}
		Given $\epsilon_1>0$, there exist $\tau>0$, $\delta_1$, $C$, $\eps$ and $\alpha>0$ so that the following holds. Suppose $\widetilde M_t$ is an RMCF with $\mathbf{r}( T)\geq R_*$, and $\widetilde M_{T+t}$ can be written as the graph of the function $w$ on $\Sigma\cap B_{\mathbf{r}(T)}$ satisfying
		$
		|w|+|\nabla w|\leq \delta\leq \delta_1, \text{ and } |\Hess_w|\leq \delta_1,
		$
		then $\widetilde M_{T+t}$ can be written as the graph of a function $u$ on $\Sigma^{\mathbf{r}({T+t})}$ for $t\in[\tau/2,\tau]$ with  $u(\cdot,0)=w$ and  
		\begin{enumerate}
			\item On $\Sigma^{\mathbf{r}({T+t})}$ and for  $t\in[\tau/2,\tau]$, we have $$|u(x,t)|\leq C(\delta+\epsilon_1)(1+|x|),\ |\nabla u(\cdot,t)|\leq C^{1/2}\sqrt{(\delta+\epsilon_1)}, \ |\Hess_{u(T+m-t)}|\leq C\dt_1 ,$$
			\item  on $\Sigma^{\mathbf{r}(T)}$ we have $\|u(\cdot,\tau)\|_{C^{2,\alpha}(\Sigma^{\mathbf{r}(T)})}\leq C(\delta+\epsilon_1)^\eps$.
		\end{enumerate}
	\end{proposition}
	
	\begin{proof}
		Lemma \ref{Lem:C0bound} shows the first bound in (1) and Lemma \ref{Lem:CurEst} shows the third bound (with the relation between the MCF and the RMCF) in (1). The second bound in (1) comes from the interpolation on $\Sigma^{\mathbf{r}({T+\tau/2})}$, and from pseudolocality on $\Sigma\cap (\mathbb A_{\mathbf{r}({T}),\mathbf{r}({T+\tau/2})})$.
		
		The above statement in (1) implies that $u$ is a strictly parabolic equation on $\Sigma^{\mathbf{r}({T})}$ when $t\in[\tau/2,\tau]$. Then (2) follows from a similar argument to \cite[Proposition 3.28]{CM2}, with the only difference being that we need to use the interior Schauder estimate (c.f. Theorem 3.6 of \cite{CS}).
		
	\end{proof}
	We next give the proof of Proposition \ref{propHolder}. 
	\begin{proof}[Proof of Proposition \ref{propHolder}]
		Whenever $\|u(\cdot,T+t)\|_{C^2(\Sigma^r)}\leq \delta_1$, we can repeatly use Proposition \ref{propHolder-small}, each time in $\Sigma^{\mathbf{r}(T)}$. Outside $\Sigma^{\mathbf{r}(T)}$, we repeatly use pseudolocality arguments. Notice that although at each step we only prove the desired bound for $t\in[\tau/2,\tau]$, we can start from $\tau/2$ to iteratively use the argument, and therefore the desired bound can be obtained for all $t\in[1,T^*]$. For item (3), we use the higher-order estimate of curvature by Ecker-Huisken (c.f. \cite[Theorem 3.4]{EH}). More precisely, Ecker-Huisken proved that the higher-order continuous norm of $u$ is uniformly bounded by a constant depending on $\delta_1$.
	\end{proof}

	\subsection{Bounding the $C^{2,\al}$ norm by the $Q$-norm}\label{SSCQ}
	
	In application, we want to iterate Proposition \ref{LmCM3} on a long period of time, which requires that  $\|u(\cdot,T+t)\|_{C^{2,\al}}<\dt$ on $\Sigma^r$ for $t$ in that period of time. However, $\|u(\cdot,T+t)\|_{C^{2,\al}}$ may become very large if $t$ is large, so Proposition \ref{LmCM3} can only be used on a short time period $[T,T+\tau]$. It may happen that
	the $Q$-norm does not acquire sufficient growth over such a time interval. To overcome this difficulty, in this subsection, we prove items (1) and (2) of Proposition \ref{ThmDynamics}. The key point is to interpolate the higher derivative estimate of Ecker-Huisken with the $Q$-norm to bound the $C^{2,\al}$ norm. A similar argument was used in \cite{Sc}, \cite{CM3} and \cite{CS}. We cite the following interpolation lemma from 
	\cite[Appendix B]{CM3}. It is mentioned in \cite[Appendix B]{CM3} that the result also extends to a hypersurface with scale-invariant curvature bounds. 
	\begin{lemma}\label{LmCMinter}
		There exists $C$ only depending on $\Sigma$ and $r$ such that if $u$ is a $C^\ell$ function on $\Sigma^r$, then  we have for $a_{\ell,n}=\frac{\ell}{\ell+n}$, $b_{\ell,n}=\frac{\ell-1}{\ell+n}$, $c_{\ell,n}=\frac{\ell-2}{\ell+n}$
		\begin{equation*}
			\|u\|_{L^\infty(\Sigma^r)}\leq C\{\|u\|_{L^1(\Sigma^{r+1})}+\|u\|_{L^1(\Sigma^{r+1})}^{a_{\ell,n}} \|\nabla^\ell u \|^{1-a_{\ell,n}}_{L^\infty(\Sigma^{r+1})}\},
		\end{equation*}
		\begin{equation*}
			\|\nabla u\|_{L^\infty(\Sigma^r)}\leq C\{\|u\|_{L^1(\Sigma^{r+1})}+\|u\|_{L^1(\Sigma^{r+1})}^{b_{\ell,n}} \|\nabla^\ell u \|^{1-a_{\ell,n}}_{L^\infty(\Sigma^{r+1})}\},
		\end{equation*}
		\begin{equation*}
			\|\nabla^2u\|_{L^\infty(\Sigma^r)}\leq C\{\|u\|_{L^1(\Sigma^{r+1})}+\|u\|_{L^1(\Sigma^{r+1})}^{c_{\ell,n}} \|\nabla^\ell u \|^{1-a_{\ell,n}}_{L^\infty(\Sigma^{r+1})}\}.
		\end{equation*}
	\end{lemma}
	
	\begin{proof}
		The proof is the same as \cite[Appendix B]{CM3} so we omit it here. We only would like to remind the reader that in \cite[Appendix B]{CM3}, the constant $C$ does not depend on $r$ because of a scale-invariant argument. Here $C$ depends on $r$ but this is sufficient for our purpose.
	\end{proof}
	
	We use $\chi_r$ to denote the characterizing function on $\Sigma^r$, i.e., $\chi_r=1$ on $\Sigma^r$ and $0$ elsewhere.

	\begin{lemma}\label{LmInterpolation}
		If $\widetilde M_t$ is a graph of the function $u$ on $\Sigma\cap B_{\mathbf{r}(t)}$ and $\|u(\cdot,t)\|_{C^\ell(\Sigma^r)}\leq C_\ell$, then we have 
		$		\|u(\cdot,t)\|_{C^{2,\alpha}(\Sigma^{r})}\leq C \|u(\cdot,t)\chi_{\mathbf{r}(t)}\|^{\alpha(\ell)}_{Q},$
		where $\alpha(\ell)=\frac{\ell-3}{\ell+n}$ and $C$ is a constant depending only on $\Sigma$, $\delta_1$ and $\ell$.
	\end{lemma}
	\begin{proof}
		The interpolation lemma \ref{LmCMinter} together with the $C^\ell$ uniform upper bound proved in Proposition \ref{propHolder} implies that
		$\|u\|_{C^{2,\alpha}(\Sigma^{r})}\leq C\|u\|^{\alpha(\ell)}_{L^2(\Sigma^{r+1})}\leq C\|u\chi_{\mathbf{r}(t)}\|^{\alpha(\ell)}_Q.$
	\end{proof}
	Items (1) and (2) of Proposition \ref{ThmDynamics} follow from this lemma. 
	
	\section{The convergence of geometric quantities}\label{SConvergence}
	In this section, we prove the convergence of eigenfunctions and eigenvalues under assumption $(\star)$, and hence verifying the assumptions of Proposition \ref{mainthm-FK} in the case of conical singularities. In Section \ref{SSConvergeEigen}, we prove the convergence of the leading eigenvalue $\lambda_1(M_t)$, and hence verifying assumption (1) of Proposition \ref{mainthm-FK}. In Section \ref{SSConvergeFunction}, we prove the convergence of the leading eigenfunction. In Section \ref{SSSpectralGap}, we prove the spectral gap, and hence verifying assumption (2) of Proposition \ref{mainthm-FK}.

	Given a hypersurface $M$ (possibly with a boundary), we define the first eigenvalue $\lambda_1(M)$ of the linearized operator $L_M:=\Delta_M-\frac{1}{2}\langle x,\nabla_M\cdot\rangle +(|A|^2+1/2)$ to be the number
	\begin{equation*}
		\lambda_1(M):=-\inf_{f\in W_w^{1,2}(M)}\frac{\int_M [|\nabla f|^2 -(|A|^2+1/2)f^2 ]\mathrm{e}^{-\frac{|x|^2}{4}}d\mu}{\int_M f^2 \mathrm{e}^{-\frac{|x|^2}{4}}d\mu},
	\end{equation*}
	where the $W_w^{1,2}(M)$ is the weighted $W^{1,2}$ space, i.e., the completion of the following space
	\begin{equation*}
		\left\{f\in C_c^\infty(M):\int_M (f^2+|\nabla f|^2)\mathrm{e}^{-\frac{|x|^2}{4}}d\mu<\infty\right\}.
	\end{equation*}
	If $M$ has a boundary, $\lambda_1(M)$ is called the first Dirichlet eigenvalue of $L_M$. Notice that if $M$ is noncompact, then $\lambda_1(M)$ could be $\infty$. We remind the reader that compared with the many other contexts in MCFs like \cite{CM1} and \cite{BW}, our definition has a minus sign.

	We first list some important properties of the first eigenvalue. 
	\begin{proposition}
		Suppose $\Sigma$ is a shrinker with finite entropy.
		\begin{itemize}
			\item $($\cite{CM1}$)$ $\lambda_1(\Sigma)\geq 1$. If $\Sigma$ is neither a sphere $S^n(\sqrt{2n})$ nor a generalized cylinder $S^k(\sqrt{2k})\times\R^{n-k}$, then $\lambda_1(\Sigma)>1$.
			\item $($\cite{BW}$)$ If $\Sigma$ is an asymptotically conical shrinker, then $\lambda_1(\Sigma)<+\infty$. 
		\end{itemize}
	\end{proposition}
	
	\subsection{The convergence of the leading eigenvalues}\label{SSConvergeEigen}
	In this subsection, we will always fix an asymptotically conical shrinker $\Sigma$, and $M_t$ is an RMCF converging to $\Sigma$ in $C_{\loc}^\infty$ sense as $t\to\infty$. We define $\lambda_1(t):=\lambda_1(M_t)$ to be the first eigenvalue of the linearized operator $L_{M_t}$ on $M_t$, and we define $\lambda_1=\lambda_1(\Sigma)$ to be the first eigenvalue of the linearized operator $L_\Sigma$ on $\Sigma$. We will also localize the eigenvalues on the hypersurfaces. For any hypersurface $M$, and radius $R>0$, we define $\lambda_1^R(M)$ to be the Dirichlet eigenvalue of $L_M$ on $M\cap B_R$. By simple comparison argument we have $\lambda_1^R(M)\geq \lambda_1(M)$.
	
	Our goal is to prove the following theorem:

	\begin{theorem}\label{thm:eigenvalue convergence}
		Suppose  $(\star)$, then $\lambda_1(M_t)\to\lambda_1(\Sigma)$ as $t\to\infty$. 
	\end{theorem}
	
	We need several lemmas to prove Theorem \ref{thm:eigenvalue convergence}. The first lemma shows that the RMCF $M_t$ has a similar curvature bound as the shrinker $\Sigma$.
	
	\begin{lemma}\label{lem:eigenconverge-Abound}
		Suppose $(\star)$, then there exist constant $0<T_0<\infty$ and $0<C<\infty$ only depending on the MCF $\mathbf{M}_\tau$ such that for $t>T_0$, $p\in M_t$,
		we have $|A|(p,t)\leq C(1+|p|)^{-1}.$
	\end{lemma}
	
	\begin{proof}
		We divide $\R^{n+1}$ into three parts. The first part is $B_{R_0}$ for a sufficiently large $R_0$. Then when $T_0$ is sufficiently large, on $B_{R_0}$, for $t>T_0$, $M_t$ is a graph over $\Sigma$ and the $C^{k}$-norm of the graph is small (depending on $T_0$). Then we can see that when $t>T_0$, inside $B_{R_0}$, $M_t$ has uniformly bounded curvature $|A|$.
		
		The second part is $B_{\mathrm{e}^{t/2}\delta}\backslash B_{R_0}$, where $\delta>0$ is a constant. By the pseudolocality of the RMCF and by the similar argument to Proposition \ref{propHolder} (also see \cite[Section 9]{CS}), inside this part $|A|$ is bounded by $C(1+|x|)^{-1}$. 
		
		The third part is the domain outside $B_{\mathrm{e}^{t/2}\delta}$. We fix this constant $\delta>0$, and we consider the MCF $\mathbf{M}_\tau$. Outside $B_\delta(0)$, $\mathbf{M}_\tau$ has no singularity when $\tau\leq 0$; moreover $\mathbf{M}_\tau$ is compact, and thus $\mathbf{M}_\tau\backslash B_\delta(0)$ has uniform curvature upper bounds (say $C'$) when $\tau\leq 0$. Now we change the view back to the RMCF, which says that outside $B_{\mathrm{e}^{t/2}\delta}$, the curvature on $M_t$ is bounded by $\mathrm{e}^{-t/2}C'$. On the other hand, because $\mathbf{M}_\tau$ is an MCF of closed hypersurfaces, $\diam(\mathbf{M}_\tau)$ is uniformly bounded (say $C''$). Therefore $M_t$ has the diameter at most $C''\mathrm{e}^{t/2}$. As a consequence, outside $B_{\mathrm{e}^{t/2}\delta}(0)$, the curvature of $M_t$ is also bounded by $C(1+|x|)^{-1}$.
		
		Combining three parts together gives us the desired curvature bound.
	\end{proof}
	
	Next lemma shows a uniform upper bound for the $\lambda_1(M_t)$.
	
	\begin{lemma}
		$\lambda_1(M_t)<C<+\infty$.
	\end{lemma}
	
	\begin{proof}
		Suppose $u$ is any smooth function on $M_t$, with $\int_{M_t}u^2 \mathrm{e}^{-\frac{|x|^2}{4}}d\mu =1$. By the uniform curvature upper bound, when $t>T_0$,
		\begin{equation*}
			-\frac{\int_{M_t}[|\nabla u|^2 -(|A|^2+1/2)u^2]\mathrm{e}^{-\frac{|x|^2}{4}}d\mu}{\int_{M_t}u^2\mathrm{e}^{-\frac{|x|^2}{4}}d\mu}\leq
			\frac{\int_{M_t}(C+1/2)u^2 \mathrm{e}^{-\frac{|x|^2}{4}}d\mu}{{\int_{M_t}u^2\mathrm{e}^{-\frac{|x|^2}{4}}d\mu}}\leq C+1/2.
		\end{equation*}
	\end{proof}
	
	In the following, we prove the convergence of $\lambda_1(M_t)\to\lambda_1(\Sigma)$. First we need some convergence properties for compact domains.
	
	\begin{proposition}\label{LmEigenfunctions}
		We have the following convergence properties:
		\begin{enumerate} 
			\item  Suppose $M$ is a hypersurface with bounded entropy, then $\lambda_1^R(M)\to \lambda_1(M)$ as $R\to\infty$. 
			\item For fixed $R>0$, $\lambda^R_1(M_t)\to \lambda^R_1(\Sigma)$ as $t\to\infty$.
		\end{enumerate}
	\end{proposition}
	\begin{proof}
		Item (1) was proved in Section 9 of \cite{CM1}. Item (2) is a corollary of the following more general result. 
	\end{proof}
	\begin{proposition}\label{thm:CptEigenConv}
		Suppose $(\Sigma,\partial\Sigma)$ is a manifold with boundary (in our case it is $\Sigma\cap B_R$), $\{g_t\}$ is a family of metrics on $\Sigma$, converging to a limit metric $g_\infty$ in $C^{2}$, and $\{V_t\}$ is a family of positive smooth functions converging to a limit function $V_\infty$ in $C^0$. Suppose $d\mu_t$ is the volume measure with respect to $g_t$. Then the leading eigenvalue of the adjoint operator $L_t$ with respect to the functional
		$	\int_\Sigma |\nabla u|_{g_t}^2-V_t u^2 d\mu_t$
		converges to the leading eigenvalue of $L=L_\infty$, as $t\to\infty$.
	\end{proposition}
	
	\begin{proof}
		Let $\lambda_1(t)$ be the first eigenvalue with respect to the time $t$ moment, and $\lambda_1$ be the first eigenvalue with respect to the limit.
		
		On one hand, suppose $\phi$ is an eigenfunction of $\lambda_1$ on the limit, then
		$$
		\frac{ \int_\Sigma (|\nabla \phi|_{g_\infty}^2-V_\infty \phi^2 )d\mu_\infty }{\int_\Sigma \phi^2 d\mu_\infty}=-\lambda_1.
		$$
		By our assumption and the minimizing property of the leading eigenvalue, we have
		\begin{equation}
			-\lambda_1(t)\leq \frac{\int_\Sigma (|\nabla \phi|_{g_t}^2-V_t\phi^2 )d\mu_t }{\int_\Sigma \phi^2 d\mu_t} 
			\to 
			\frac{\int_\Sigma (|\nabla \phi|_{g_\infty}^2-V_\infty\phi^2) d\mu_\infty }{\int_\Sigma \phi^2 d\mu_\infty}=-\lambda_1
		\end{equation}
		as $t\to\infty$. Thus we get $\limsup_{t\to\infty} -\lambda_1(t)\leq -\lambda_1.$
		
		On the other hand, let $\phi_t$ be the leading eigenfunction with respect to $g_t$. We will assume that $\phi_t>0$ and $\int_\Sigma \phi_t^2 d\mu_\infty=1$. Then
		\begin{equation*}
			-\lambda_1(t)=\frac{\int_\Sigma (|\nabla \phi_t|_{g_\infty}^2-V_\infty \phi^2 )d\mu_\infty }{\int_\Sigma \phi_t^2 d\mu_\infty}.
		\end{equation*} 
		Let us estimate 
		\begin{equation*}
			\left|\frac{\int_\Sigma (|\nabla \phi_t|_{g_\infty}^2-V_\infty \phi_t^2) d\mu_\infty }{\int_\Sigma \phi_t^2 d\mu_\infty}-\frac{\int_\Sigma (|\nabla \phi_t|_{g_t}^2-V_t \phi_i^2 )d\mu_t }{\int_\Sigma \phi_t^2 d\mu_t}\right|.
		\end{equation*}
		Since $g_t\to g_\infty$ in $C^2$, we have the following
		\begin{itemize}
			\item $d\mu_t\to d\mu_\infty$ in $C^0$. Moreover, $\frac{d\mu_t}{d\mu_\infty}=1+o(t)$. When we write $o(t)$ we mean a quantity $\to 0$ as $t\to\infty$.
			\item $\left||\nabla \phi_t|_{g_t}^2-|\nabla \phi_t|_{g_\infty}^2\right|\leq o(t)|\nabla \phi_t|^2_{g_\infty}$ where $o(t)\to 0$ in $C^0$. Similarly, we can write $|\nabla \phi_t|_{g_t}^2= |\nabla \phi_t|_{g_\infty}^2(1+o(t))$.
			\item We can also write $V_t=V_\infty(1+o(t))$ (note $V_\infty$ is positive).
		\end{itemize}
		From the above items, we conclude that
		\begin{equation*}
			\frac{\int_\Sigma (|\nabla \phi_t|_{g_t}^2-V_t \phi_t^2 )d\mu_t }{\int_\Sigma \phi_t^2 d\mu_t}
			=
			\frac{\int_\Sigma (|\nabla \phi_t|_{g_\infty}^2-V_\infty \phi_t^2 )d\mu_\infty }{\int_\Sigma \phi_t^2 d\mu_\infty}(1+o(t)).
		\end{equation*}
		We remark that $o(t)$ is a quantity depending on how much $g_t$ and $V_t$ are close to $g_\infty$ and $V_\infty$ respectively, but does not depend on $\phi_t$. In conclusion, we have
		$
		-\lambda_1(t) \geq -\lambda_1(1+o(t)).
		$
		As a result, $	\liminf_{t\to\infty} -\lambda_1(t)\geq -\lambda_1.$
		Thus we conclude that $\lim\lambda_1(t)=\lambda_1$.
	\end{proof}
	
	\begin{lemma}\label{lem:eigenconverge-1}
		For all $\epsilon>0$, there exists $0<T<\infty$, such that $\lambda_1(M_t)\geq\lambda_1(\Sigma)-\epsilon$ if $t>T$. As a consequence, we have
		$\liminf_{t\to\infty}\lambda_1(t)\geq \lambda_1(\Sigma)$.
	\end{lemma}
	
	\begin{proof}
		The proof is similar to the proof of Proposition \ref{thm:CptEigenConv}. First, by (2) in Proposition \ref{LmEigenfunctions}, we can show $\lambda_1^R(M_t)\to\lambda_1^R(\Sigma)$ for any fixed $R$, as $t\to\infty$. Second, we have $\lambda_1^R(M_t)\leq \lambda_1(M_t)$, and (1) in Proposition \ref{LmEigenfunctions} shows that $\lambda_1^R(\Sigma)\to\lambda_1(\Sigma)$ as $R\to\infty$. Thus we conclude this lemma.
	\end{proof}
	
	\begin{lemma}\label{lem:eigenconverge-2}
		For all $\epsilon>0$, there exists $0<T<\infty$, such that $\lambda_1(M_t)\leq\lambda_1(\Sigma)+\epsilon$ if $t>T$.
		As a consequence, we have 
		$\limsup_{t\to\infty}\lambda_1(t)\leq \lambda_1(\Sigma).$\end{lemma}
	\begin{proof}
		We choose two constants $R$ and $\kappa$ to be determined later, and we assume $t$ is sufficiently large such that $\lambda_1^{2R}(M_t)\geq (1-\kappa)\lambda_1^{2R}(\Sigma)\geq (1-\kappa)\lambda_1(\Sigma)$.
		
		Suppose $\phi_1^t$ is a first eigenfunction on $M_t$. Then
		\[\lambda_1(M_t)=-\frac{\int_{M_t}[|\nabla \phi_1^t|^2 -(|A|^2+1/2)(\phi_1^t)^2]\mathrm{e}^{-\frac{|x|^2}{4}}d\mu}{\int_{M_t}(\phi_1^t)^2\mathrm{e}^{-\frac{|x|^2}{4}}d\mu}.
		\]
		We also assume $\eta$ is a smooth cutoff function, which is constant $1$ on $B_R$, and $0$ outside $B_{2R}$, with $|\nabla \eta|\leq C/R$. Define $v=\phi_1^t \cdot\eta$ and $w=\phi_1^t\cdot(1-\eta)$. Then $v$ is supported on $M_t\cap B_R $, and $w$ is supported on $M_t\backslash B_R$, and $v+w=\phi_1^t$. Then
		\[
		\frac{\int_{M_t}[|\nabla v|^2 -(|A|^2+1/2)v^2]\mathrm{e}^{-\frac{|x|^2}{4}}d\mu}{\int_{M_t}v^2\mathrm{e}^{-\frac{|x|^2}{4}}d\mu} \geq -\lambda_1^{2R}(M_t),
		\]
		\[
		\frac{\int_{M_t}[|\nabla w|^2 -(|A|^2+1/2)w^2]\mathrm{e}^{-\frac{|x|^2}{4}}d\mu}{\int_{M_t}w^2\mathrm{e}^{-\frac{|x|^2}{4}}d\mu}
		\geq 	\frac{\int_{M_t}[ -(C(1+|x|)^{-2}+1/2)w^2]\mathrm{e}^{-\frac{|x|^2}{4}}d\mu}{\int_{M_t}w^2\mathrm{e}^{-\frac{|x|^2}{4}}d\mu}
		\geq -C(1+R)^{-2}-1/2.
		\]
		Also, \begin{equation}\label{eq:eigenconverge1}
			\begin{split}
				&\int_{M_t}[\nabla v\cdot\nabla w] \mathrm{e}^{-\frac{|x|^2}{4}}d\mu
				=
				\int_{M_t}[(\phi_1^t\nabla \eta+\eta\nabla \phi_1^t)\cdot(-\phi_1^t\nabla\eta+(1-\eta)\nabla \phi_1^t)] \mathrm{e}^{-\frac{|x|^2}{4}}d\mu
				\\
				=&
				\int_{M_t}[-(\phi_1^t)^2|\nabla\eta|^2 +\eta(1-\eta)|\nabla \phi_1^t|^2
				+\phi_1^t(1-2\eta)\nabla\eta\cdot\nabla\phi_1^t] \mathrm{e}^{-\frac{|x|^2}{4}}d\mu
				\\
				\geq &
				\int_{M_t}[-C(\kappa)(\phi_1^t)^2|\nabla\eta|^2-\kappa|\nabla \phi_1^t|^2] \mathrm{e}^{-\frac{|x|^2}{4}}d\mu
			\end{split}
		\end{equation}
		\begin{equation}
			\begin{split}
				\geq &
				\int_{M_t}[-\big(C(\kappa)|\nabla\eta|^2-\kappa(|A|^{2}+1/2)\big)(\phi_1^t)^2
				\\
				&-\kappa\big(|\nabla \phi_1^t|^2-(|A|^2+1/2)(\phi_1^t)^2\big)] \mathrm{e}^{-\frac{|x|^2}{4}}d\mu\\
				&(\text{note $|\nabla \eta|\leq CR^{-1}$})\geq 
				-(C(\kappa)R^{-2}+\kappa|A|^2+\kappa/2-\kappa\lambda_1(t))
				\int_{M_t}(\phi_1^t)^2\mathrm{e}^{-\frac{|x|^2}{4}}d\mu 
				\\
				&\geq 
				-C\kappa\int_{M_t}(\phi_1^t)^2\mathrm{e}^{-\frac{|x|^2}{4}}d\mu.
			\end{split}
		\end{equation}
		In the last inequality we assume that $R$ is sufficiently large.
		Let 
		$
		m:=\min\{-\lambda_1^{2R}(M_t),-C(1+R)^{-2}-1/2\}.
		$
		Then $m<0$ and when $R$ is sufficiently large, $-\lambda_1^{2R}(\Sigma)\leq\frac{3}{2}( -C(1+R)^{-2}-1/2)$, and therefore when $t$ is very large, $-\lambda_1^{2R}(M_t)\leq( -C(1+R)^{-2}-1/2)$. Therefore $m=-\lambda_1^{2R}(M_t)$. We have
		\begin{multline}\label{eq:eigenconverge2}
			\int_{M_t}[|\nabla v|^2 -(|A|^2+1/2)v^2]\mathrm{e}^{-\frac{|x|^2}{4}}d\mu
			+
			\int_{M_t}[|\nabla w|^2 -(|A|^2+1/2)w^2]\mathrm{e}^{-\frac{|x|^2}{4}}d\mu
			\geq
			\\ m\left(\int_{M_t}v^2\mathrm{e}^{-\frac{|x|^2}{4}}d\mu
			+
			\int_{M_t}w^2\mathrm{e}^{-\frac{|x|^2}{4}}d\mu
			\right)
			=m\int_{M_t}(\phi_1^t)^2\mathrm{e}^{-\frac{|x|^2}{4}}d\mu.
		\end{multline}
		
		Together with \eqref{eq:eigenconverge1} and \eqref{eq:eigenconverge2}, we get
		\begin{multline*}
			-\lambda_1(M_t)=
			\frac{\int_{M_t}[|\nabla \phi_1^t|^2 -(|A|^2+1/2)(\phi_1^t)^2]\mathrm{e}^{-\frac{|x|^2}{4}}d\mu}
			{\int_{M_t}(\phi_1^t)^2\mathrm{e}^{-\frac{|x|^2}{4}}d\mu}
			\\
			\geq m-C\kappa=-\lambda_1^{2R}(M_t)-C\kappa
			\geq (1-\kappa)\lambda_1(\Sigma)-C\kappa.
		\end{multline*}
		So for any $\epsilon>0$, we can choose $\kappa$ small such that $-\lambda_1(M_t)\geq -\lambda_1(\Sigma)-\epsilon$. This shows the desired result.
	\end{proof}

	\begin{proof}[Proof of Theorem \ref{thm:eigenvalue convergence}]
		Combining Lemma \ref{lem:eigenconverge-1} and Lemma \ref{lem:eigenconverge-2} gives the proof.
	\end{proof}
	
	\subsection{The convergence of the leading eigenfunctions} \label{SSConvergeFunction}In this subsection, we prove the following theorem:
	
	\begin{theorem}\label{thm:conv1steigenfun}
		After normalization, the first eigenfunction $\phi_1^t$ on $M_t$ converges to $\phi_1$ as $t\to\infty$ in $C_{\loc}^\infty$ sense, where $\phi_1$ is the first eigenfunction on $\Sigma$.
	\end{theorem}
	
	\begin{proof}
		By the standard elliptic theory, $\phi_1^t$ is smooth and positive on $M_t$. For any fixed ball of radius $R$, when $t$ is very large, a part of $M_t$ can be written as a graph over $\Sigma\cap B_R$. So if we fix a point $P$ on $\Sigma$ with $|P|\leq \sqrt{2n}$, and we denote by $P_t$ the corresponding point on $M_t$ ($P_t$ is well-defined when $t$ is sufficiently large), we can divide $\phi_1^t$ by a constant so that $\phi_1^t(P_t)=1$. Then the standard Harnack inequality and the Schauder estimate show that $\phi_1^t$ converges locally smoothly to a limit function $\phi_1$. Finally, because $\lambda_1(t)\to\lambda_1$, $\phi_1$ must satisfy the equation 
		$L\phi_1=\lambda_1(\Sigma)\phi_1.$
		Thus $\phi_1$ is exactly the first eigenfunction if it belongs to the space
		$\left\{u:\int_\Sigma u^2\mathrm{e}^{-\frac{|x|^2}{4}}d\mu<\infty\right\}$  (see \cite[Proposition 4.1]{BW}).

		On $M_t$, $\phi_1^t$ satisfies the equation
		$\cL_{M_t}\phi_1^t+(|A|^2+1/2-\lambda_1(t))\phi_1^t=0.$
		So by the maximum principle, $\max\phi_1^t$ is attained at somewhere $|A|^2+1/2\geq \lambda_1(t)$. By the decay of $|A|$ (see Lemma \ref{lem:eigenconverge-Abound}) and the convergence of the first eigenvalues, when $t$ is sufficiently large, $\max \phi_1^t$ is attained only in $M_t\cap B_{r_0}$ with some fixed $r_0>0$. Then by the Harnack inequality and the entropy bound of the flow, we know that $\int_{M_t} (\phi_1^t)^2\mathrm{e}^{-\frac{|x|^2}{4}}d\mu$ is uniformly bounded. Therefore $\int_\Sigma (\phi_1)^2\mathrm{e}^{-\frac{|x|^2}{4}}d\mu$ is finite.
	\end{proof}
	
	We may also assume that the normalization is $\|\phi_1^t\|_{L^2(M_t)}=1$, and similarly $\phi_1^t$ converges to a limit $\phi_1$, which is a first eigenfunction on $\Sigma$. At this moment, we do not know whether $\|\phi_1\|_{L^2(\Sigma)}=1$. Our goal is to show that $\|\phi_1\|_{L^2(\Sigma)}=1$.

	\begin{corollary}\label{cor:1stcutsmall}
		there exist $R_0>0$ and a function $\eta:[R_0,\infty)\to \R_+$ such that $\lim_{R\to\infty}\eta(R)=0$, and $\|\phi_1^t(1-\chi_R)\|_{L^2(M_t)}<C R^{-1}$.   
	\end{corollary}
	
	\begin{proof}
		From the decay rate of $\phi_1$ in \cite[Proposition 4.1]{BW}, there exists $R_0>0$ such that for $p\in\Sigma$ and $|p|>R\geq R_0$, we have $\phi_1(p)<C|p|^{-1}$. Then by the locally smooth convergence Theorem \ref{thm:conv1steigenfun}, $\phi_1^t\to\phi_1$ on $\Sigma\cap B_{R+1}$ for some constant $\eta>0$. Thus, $\phi_1^t\leq C|p|^{-1}$ on $\Sigma\cap \partial B_{R_0}$ as well, when $t$ is sufficiently large.  
		
		In the proof of Theorem \ref{thm:conv1steigenfun}, we have shown that by the maximum principle, $\max\phi_1^t$ is attained at somewhere $|A|^2+1/2\geq \lambda_1(t)$. This is also true if we replace the maximum with the local maximum. Thus, when $R$ is sufficiently large, we know that $\phi_1^t\leq CR^{-1}$ on $M_t\backslash B_R$. This implies that
		\begin{equation*}
			\|\phi_1^t(1-\chi_R)\|_{L^2(M_t)}\leq\left( \int_{M_t\backslash B_R} CR^{-2}\mathrm{e}^{-\frac{|x|^2}{4}} d\mu\right)^{1/2} \leq CR^{-1},
		\end{equation*}
		where in the last inequality we use the fact that the entropy of $M_t$ is uniformly bounded. This yields the desired inequality.
	\end{proof}
	
	As a consequence, we can show that $\phi_1^t$ actually converges to $\phi_1$ with $\|\phi_1\|_{L^2(\Sigma)}=1$. It is a consequence of the following theorem.

	\begin{theorem}\label{thm:conv1steigenfunL2}
		There exists $R_0>0$ with the following significance. Let us normalize $\phi_1^t$ such that $\|\phi_1^t\|_{L^2(M_t)}=1$. Then for any $R>R_0$, $\|\overline{\phi_1^t}-\phi_1\|_{L^2(M_t\cap B_R)}\to 0$ as $t\to\infty$ and $\|\phi_1\|_{L^2(\Sigma)}=1$. Here $\overline{\phi_1^t}$ is the function $\phi_1^t$ pulled back to $\Sigma$.
	\end{theorem}
	
	\begin{proof}
		The proof is the same as that of Theorem \ref{thm:conv1steigenfun}. Recall that $\max \phi_1^t$ is attained only in $M_t\cap B_{r_0}$ with some fixed $r_0>0$. Also, $\|\phi_1^t\|_{L^2(M_t)}=1$ implies that $\max\phi_1^t\geq \lambda(t)^{-1}\geq C$ for some constant $C$, inside $M_t\cap B_{r_0}$. Again, by the Harnack inequality, we know that $\phi_1^t$ has a uniformly lower bound $c_0$ on $M_t\cap B_{r_0}$. As a consequence of the Arzela-Ascoli theorem, $\phi_1^t$ smoothly converges to a limit on $M_t\cap B_{r_0}$. From Theorem \ref{thm:conv1steigenfun}, this limit $\phi_1$ must be a first eigenfunction on $\Sigma$. 
		
		It remains to show $\|\phi_1\|_{L^2(\Sigma)}=1$. Corollary \ref{cor:1stcutsmall} implies that $\|\phi_1^t(1-\chi_R)\|_{L^2(M_t)}<CR^{-1}$ when $t$ is sufficiently large. Thus $\|\phi_1^t\chi_R\|_{L^2(M_t)}>1-CR^{-1}$. This implies that $\|\phi_1\|_{L^2(\Sigma\cap B_R)}\geq 1-CR^{-1}$, and as a consequence $\|\phi_1\|_{L^2(\Sigma)}\geq1$. On the other hand, $\|\phi_1\|_{L^2(\Sigma)}\leq \limsup \|\phi_1^t\|_{L^2(M_t)}=1$. So $\|\phi_1\|_{L^2(\Sigma)}=1$.
	\end{proof}
	
	\subsection{The spectral gap between the first two eigenvalues}\label{SSSpectralGap}
	In this subsection, we study the gap between the first eigenvalue and the second eigenvalue. We denote by $\lambda_2(t)=\lambda_2(M_t)$ and $\lambda_2=\lambda_2(\Sigma)$.
	\begin{theorem}
		Suppose $(\star)$. Then there exists $C>0$ such that $$\lambda_1(t)-\lambda_2(t)\geq C.$$
	\end{theorem} 
	
	\begin{proof}
		We prove by contradiction. Assume $\lambda_1(t)-\lambda_2(t)\to 0$ as $t\to\infty$. Then since $\lambda_1(t)\to\lambda_1(\Sigma)$, we have $\lambda_2(t)\to\lambda_1(\Sigma)$ as $t\to\infty$. Let $\phi_2^t$ be a second eigenfunction of $L_{M_t}$, namely 
		$\phi_2^t$ satisfies the equation
		\[\cL_{M_t}\phi_2^t+(|A|^2+1/2-\lambda_2(t))\phi_2^t=0.\]
		Suppose $p_t$ is the maximum point of $\phi_2^t$ and $q_t$ is the minimum point of $\phi_2^t$. By the elliptic theory, $\phi_2^t$ must change sign, so $\phi_2^t(p_t)>0$ and $\phi_2^t(q_t)<0$. Moreover, by the maximum principle, $|A|^2+1/2-\lambda_2(t)\geq 0$ at $p_t$ and $q_t$. When $t$ is sufficiently large, $\lambda_2(t)$ is close to $\lambda_1(\Sigma)>1$. Therefore Lemma \ref{lem:eigenconverge-Abound} implies that $|p_t|$ and $|q_t|$ are bounded by some uniform constant $R_0>0$.
		
		By multiplying a constant, we may assume $\phi^t_2(p_t)=1$, and by the gradient estimate of $\phi^2_t$ on a bounded domain $M_t\cap B_{R_0}$, we have $\phi^t_2(q_t)\geq -C$ where $C$ is a constant. This also implies a global $L^\infty$ bound of $\phi^t_2$, and hence a weighted $W^{1,2}$ bound. Then elliptic theory shows that after passing to a subsequence, $\phi^t_2$ converges to a limit $\phi_2$ on $\Sigma$ in $C_{\loc}^{\infty}$ sense, and if $p_t\to p$, we have $\phi_2(p)=1$. Moreover, $\phi_2$ satisfies the equation
		$\cL_{\Sigma}\phi_2+(|A|^2+1/2-\lambda_1)\phi_2=0.$
		As a consequence, $\phi_2$ must be a first eigenvalue of $L_{\Sigma}$. This means that $\phi_2>0$ everywhere, which contradicts the fact that $\phi^t_2(q_t)<0$ for all $t$ (note that $q_t$ is uniformly bounded).
	\end{proof}

	\appendix
	
	\section{The transplantation of functions}\label{SS:transplantation}
	In this section we fix a radius $R$ sufficiently large, and we only consider $t$ sufficiently large, such that a part of $M_t$ can be written as a smooth graph of a function $f$ on $\Sigma\cap B_R$. We denote this part by $M_t^R$. Recall that in Section \ref{SSConeLinear} we introduce the notion called the transplantation, namely for any function $g$ defined on $M_t^R$, we defined a function $\varphi_t^* g$ defined on $\Sigma\cap B_R$, such that 
	$ \varphi_t^*g(x)= g(x+f(x)\bn).$
	For simplicity, sometimes we just write $\bar{g}$ to denote $\varphi^*_t g$.
	
	Suppose $N_t^R$ is a graph of the function $g$ over $M_t^R$. If $f$ and $g$ are sufficiently small (in the $C^\ell$-norm for $\ell\geq 4$), then $N_t^R$ can be viewed as a graph of the function $v$ where $v$ is defined on $\Sigma$. The following lemma is similar to Theorem C.2 in \cite{SX} and we omit the proof here.
	
	\begin{lemma}[Theorem C.2 in \cite{SX}]\label{Lem;transplantation}
		For any $\eps'>0$, there exists $\eps>0$ with the following significance: if for $\|f\|_{C^{\ell}(\Sigma\cap B_{R+1})}\leq \eps$, $\|g\|_{C^{\ell}(\Sigma\cap B_{R+1})}\leq \eps$, then
		\[\|v-(f+\bar g)\|_{C^{\ell}(\Sigma\cap B_R)}\leq \eps'\|\bar g\|_{C^{\ell}(\Sigma\cap B_{R+1})}.\]
	\end{lemma}
	
	The above bound requires $\|f\|_{C^\ell}$ and $\|g\|_{C^\ell}$ are sufficiently small. On the asymptotically conical part, $f$ and $g$ may not have a small $C^\ell$-norm. So we also need the following estimate, in the spirit of the pseudolocality theorem. 
	
	\begin{lemma}\label{Lem:transplantation-cone}
		For any $\eps'>0$ and integer $\ell>0$, there exists $\eps>0$ with the following significance: suppose $M_t$ and $\widetilde{M}_t$ are two RMCFs as in Section \ref{SSDifferenceEq}, and $m(\cdot,t)$ is the graph function of $M_t$ over $\Sigma$ and $v(\cdot,t)$ is the graph function of $\widetilde{M}_t$ over $M_t$. Suppose at $t=0$, $\|v(\cdot,0)\|_{C^{2,\alpha}(M_0^r)}\leq \eps$ and $\|m(\cdot,0)\|_{C^{2,\alpha}(\Sigma\cap B_r)}\leq \eps$. Let $u(\cdot,t)$ be the graph function of $\widetilde{M}_t$ on $\Sigma\cap B_{\mathbf{r}(M_t)}$. Then on $(B_{\mathbf{r}(M_t)}\backslash B_{r})\cap\Sigma$, at the time $t$, for $r+1<R<\mathbf{r}(M_t)-2$, we have
		\begin{equation*}
			\|\nabla^k(u(\cdot,t)-(m(\cdot,t)+\bar u))\|_{C^0((B_{R+1}\backslash B_{R})\cap\Sigma)}\leq C\eps'|\nabla^k \bar u|_{C^0((B_{R+2}\backslash B_{R-1})\cap\Sigma)}
		\end{equation*}
		for $k=0,1,2,\cdots,\ell$.
	\end{lemma}
	
	The proof uses Lemma \ref{Lem;transplantation} on the asymptotically conical region of $\Sigma$. In fact, the asymptotically conical region is a part of the MCF after rescaling. Before the rescaling, the MCF is a graph very close to a cone. Thus the desired estimate is just the usual transplantation after a rescaling.

	\section{The polar-spherical transplantation and the $L$-operator}\label{AppP}
	As we have explained in Section \ref{SSDifferenceEq}, we introduce a polar-spherical coordinates approach to transplant functions on $M_t^{\mathbf r(t)}$ to $\Sigma^{\mathbf r(t)}$. In this appendix, we shall study this transplantation in more detail and in particular give the estimate of the error term $\mathcal Q$ in \eqref{EqDifference1} (see Lemma \ref{LmAppP} below). 
	
	Suppose $\Sigma$ is a fixed asymptotically conical shrinker. Then $\Sigma$ is very close to a cone $\cC:=\{r \theta,\ \theta\in \mathcal S\subset\mathbb S^n(1),\ r\geq 0\}$ far away from the origin, where $\mathcal S$ is a codimension-$1$ submanifold of $\mathbb S^n(1)$.  Then there is $R$ sufficiently large such that $\Sigma\backslash B_R$ can be identified with $\cC\backslash B_R$. In the following we will use $(r,\theta)$ to describe the points on $\Sigma\backslash B_R$, such that for $x=(r,\theta)$, $|x|=r$ and $x/|x|$ is perpendicular to $\theta$ on the unit sphere.
	
	In the following, $\bar r=\bar r(r)$ is a smooth function that is identified with $r$ when $r\geq R+1$, and equals $0$ when $r\leq R$. Then any function $m$ defined on $\Sigma\backslash B_R$ can be written as $c(r,\theta)\bar r+f(x)$.
	Now suppose $M$ is a hypersurface that is very close to $\Sigma$ so that $M$ can also be written as a graph over $\cC$ outside $B_R$ (but possibly inside a much larger ball). Then we can write $M$ as a graph of $m(x)=c(r,\theta)\bar r+f(x)$, where $f$ is supported on $B_{R+1}$. Here $m$ is defined as follows: inside $B_{R}$, $M$ can be written as a graph of the function $f$ over $\Sigma$, and outside $B_{R+1}$, $m$ is the difference between them when they are both viewed as spherical graphs over $\cC$.
	
	We consider an RMCF $M_t$ converging to a conical shrinker in the $C^\infty_{\loc}$ sense as $t\to\infty$. We next estimate the function $c$ in the representation $c(r,\theta)r+f$ above. Let $\mathbb A_{r_1,r_2}$ be the annulus with the inner radius $r_1>R$ and the outer radius $r_2$. Suppose $M$ is a graph of $m(x)=c(r,\theta)\bar r$ over $\cC\cap \mathbb A_{r_1,r_2}$, with $\|m\|_{C^{2,\alpha}}\leq \eps$. Then $aM$, the rescaling of $M$ by $a$, is a graph of the function $c(r/a,\theta)\bar r$ over $\cC\cap \mathbb A_{ar_1,ar_2}$. As a consequence, the $C^{2,\al}$ norm of $c$ is unchanged after rescaling (actually, the derivative terms become better).
	If $M$ is a graph of $m(x)=c(r,\theta)\bar r$ over $\Sigma\cap \mathbb A_{r_1,r_2}$ rather than $\cC\cap \mathbb A_{r_1,r_2}$, similar analysis holds true when $R$ is sufficiently large. In fact, from \cite{BW} and \cite{CS} we know that the geometry of $\Sigma\cap \mathbb A_{ar_1,ar_2}$ converges to $\cC\cap \mathbb A_{ar_1,ar_2}$ as $a\to\infty$.
	
	As a consequence, the pseudolocality estimates for the graph of the RMCF do hold for this setting. We have the following lemma:
	
	\begin{lemma}
		Suppose $\Sigma$ is an asymptotically conical shrinker, and $M_t$ is an RMCF converging to $\Sigma$. Let $R>R_0$ sufficiently large. Then when $T$ is sufficiently large, suppose $M_T$ is a graph of the function $u(\cdot,T)=c(\cdot,T)\bar r$ on $\mathbb A_{R-1,R}\cap\Sigma$ with $|\nabla c(\cdot,T)|\leq \gamma$, then for any integer $\ell\geq 0$, $M_{T+t}$ is a graph of $u(\cdot,T+t)$ on $\mathbb A_{\mathrm{e}^{t/2}(R-1),\mathrm{e}^{t/2}R}\cap\Sigma$ with 
		\[
		\|\nabla^\ell_r c(\cdot,T+t)\|_{C^0}\leq \mathrm{e}^{-lt/2}C_\ell,\ \ \ 
		\|\nabla^\ell_\theta c(\cdot,T+t)\|_{C^0}\leq C_\ell,
		\]
		for $t>1/2$. Moreover, $C_\ell\to 0$ as $\gamma\to 0$.
	\end{lemma}
	
	We next define $\varphi_t$ as in Definition \ref{DefTransplant}. It is clear that $\varphi$ preserves the Gaussian weight. We next show that the weighted Sobolev norm behaves well under this pullback. Suppose $M$ can be writen as a graph of the function $m(x)=c(r,\theta)\bar r+f(x)$ over $\Sigma^{\mathbf{r}(t)}$, with $\|f\|_{C^{2,\alpha}}\leq \eps$, $\|c\|_{C^{2,\alpha}}\leq \eps$. Suppose $v$ is a function over $M$ and define $u^*=\varphi^* v$, then we have $|D\varphi^{-1}|\leq (1+C\eps)$, $|\langle \partial_t\varphi_t,\nabla v\rangle|\leq |\langle \vec H_{M_t}+\frac{x^\perp_{M_t}}{2},\bn_{\Sigma}\rangle||\nabla v|\leq C\eps |\nabla v|$, which shows the smallness of the last two terms in \eqref{EqQ}. We also have
	\begin{equation}\label{EqNormCompare}
		\begin{aligned}
			\int_{M^{\mathbf r(t)}} |v(x)|^2\mathrm{e}^{-\frac{|x|^2}{4}}d\mu(x)&=\int_{\Sigma^{\mathbf r(t)}} |u^*(y)|^2\mathrm{e}^{-\frac{|\varphi^{-1}_t(y)|^2}{4}}\det(D\varphi^{-1})d\mu(y)\\
			&\leq (1+ C\eps ) \int_{\Sigma^{\mathbf r(t)}} |u^*|^2\mathrm{e}^{-\frac{|y|^2}{4}} d\mu(y). 
		\end{aligned}
	\end{equation}
	This implies that the weighted $L^2$-norm is comparable to the functions on $M$ to $\Sigma$. We have similar conclusions for weighted higher-order Sobolev norms.

	We next discuss the difference between the two operators $\varphi^*(L_{M}v)$ and $L_{\Sigma}(\varphi^*v)$ under the transplantation. We have the following lemma.
	\begin{lemma}\label{LmAppP} Suppose $M$ satisfies
		\begin{enumerate}
			\item the function $m:\ \Sigma\to \R$ can be written as $ m(x)=c(r,\theta)\bar r+f(x)$, where $\|c\|_{C^{2,\al}}<\eps$ and $\|f\|_{C^{2,\al}}\leq \eps $;
			\item the second fundamental form decays as $|\nabla^j A|(x)\leq C\frac{1}{|x|^{j+1}},\ j=0,1,2.$
		\end{enumerate}
		Then for all smooth $v:\ M\to \R$, we have the following pointwise bound $$|\varphi^*(L_{M}v)-L_{\Sigma}(\varphi^*v)|(x)\leq C\eps(1+|\nabla \varphi^*v(x)|+|\nabla^2 \varphi^*v(x)|+|x|\cdot |\nabla \varphi^*v(x)|).$$
	\end{lemma}
	
	\begin{proof}
		For simplicity, we denote by $u=\varphi^*v$. Inside $B_{R+1}$ the calculation is the standard computation on the normal bundle, similar to \eqref{EqP}. 
		
		Let us do the computations outside $B_{R+1}$, namely, we only care about the situation where $m=c(r,\theta)\bar{r}$, which is the distance between $M$ and $\Sigma$ as two sections of the conical neighbourhood of $\cC$.
		
		Firstly, we consider a graph $N$ over the cone $\cC$, locally given by $c(r,\theta)r$. Temporarily we use $\varphi$ to denote the retraction to $\cC$. Let $r\theta\in \cC$. We define $\gamma_\theta(t)$ to be the exponential curve of $\theta\in \cC\cap S^n(1)$ inside the unit sphere, along the unit normal direction at $\theta$. We also view it as a curve in $\R^{n+1}$. Then $N$ is the submanifold given by $\{r\gamma_\theta(c(r,\theta))\}$. By rescaling, we may assume $(r,\theta)=(1,\theta)$.
		
		Now we are going to use Fermi coordinates near $(r,\theta)\in\cC$. For the definition of the Fermi coordinate (see \cite[Lemma A.2]{LZ}). It is canonical to define a Fermi coordinate near $(r,\theta)\in \cC$ by conically extending the Fermi coordinate of $\theta\in\cC\cap S^n(1)$ in the unit sphere. With this coordinate, the term $|\varphi^*((\Delta_{M}+|A|^2+1/2)v)-(\Delta_{M}+|A|^2+1/2)(\varphi^*v)|(x)$ can be computed by using a slightly twisted metric of the Euclidean metric at $(r,\theta)$, while the error is estimated in \cite[Appendix A]{LZ}, implying that we have the desired bound. Similarly, $|\varphi^*(\langle v,x\rangle)-\langle \varphi^*v,x\rangle|(x)$ is straightforwardly bounded by $C\eps |x|\cdot |\nabla \varphi^*v(x)|$. 
		
		To extend the above estimate to the situation that $M$ is a graph over $\Sigma$, we can first write both of them as graphs over $\cC$. Then the calculations are mostly verbatim. Thus we get the desired bound.
	\end{proof}
	
	\section{Proof of Proposition \ref{LmCM3}}\label{AppPropCM}
	In this appendix, we give the proof of Proposition \ref{LmCM3}. 
	\begin{proof}[Proof of Proposition \ref{LmCM3}]
		We shall compare the solutions to the equations \eqref{EqDifference1} and \eqref{EqLinearAuto}. 
		The proof is similar to that of Proposition 4.3 of \cite{CM2}. The main difficulty is that we do not have the Dirichlet boundary value condition for $v^*$. Thus, there is a boundary term when doing integration by parts that need to be absorbed. 
		
		For simplicity and without loss of generality, we consider only $n=0$ and suppress the subscript of $v_n$. We consider the $L^2$ bound first. Define $w=v^*\chi -v$
		where $\chi$ is a smooth function that is 1 on the set $\{|x|\leq \mathbf r(t)-1\}$ and $0$ on $\{|x|\geq \mathbf r(t)\}$ with $\mathbf r(t):=\mathbf r(M_{T_\sharp+t})$. Therefore $\frac{\partial}{\partial t}\chi$ is of order $O(\mathrm{e}^{t/2})$ and supported on the annulus $\mathbb A_{\mathbf r(t)}:=\{\mathbf r(t)-1\leq |x|\leq \mathbf r(t)\}$ and $\nabla_x \chi$  and $\nabla_x^2 \chi$ are bounded by a constant and supported on the boundary of $B_{\mathbf r(t)}$. 
		
		We have $\partial_t v^*=L_\Sigma v^*+\mathcal Q(v^*)$ over $\Sigma^{\mathbf r(t)}$ (equation \eqref{EqDifference1}) and $\partial_t v=L_\Sigma v$ over $\Sigma$ (equation \eqref{EqLinearAuto}), so we get
		\begin{equation}\label{Eqdw}
			\begin{aligned}
				\partial_t w&=\chi\partial_t v^*+v^*\partial_t\chi-\partial_tv=\chi(L_\Sigma v^*+\mathcal Q(v^*))+v^*\partial_t\chi-\partial_tv\\
				&=L_\Sigma w+\chi \mathcal Q(v^*)+(-2\nabla\chi\cdot\nabla v^*-v^*\mathcal L_\Sigma \chi+v^*\partial_t\chi)\\
				&=:L_\Sigma w+\chi \mathcal Q(v^*)+\mathcal B.
			\end{aligned}
		\end{equation}
		Here the term $\mathcal B$ is supported on the annulus $\mathbb A_{\mathbf r(t)}$ and is bounded by $O(\|v^*\|_{C^1(\mathbf r(t))}\mathrm{e}^{t})$. 
		
		We next estimate $\mathcal Q( v^*)=P(v^*)+\varphi_t^*Q(v^*)$ over the ball $B_{\mathbf r(t)}$ (c.f. \eqref{EqQ}).  The estimate  $|P(v^*)|\leq \dt(|\mathrm{Hess}_{v^*}|+|\nabla v^*|+|v^*|+|x||\nabla v^*|)$ is given in Lemma \ref{LmAppP}. We also have the same estimate for $ \varphi_t^*Q( v^*)$, following from Lemma 5.3 of \cite{CM3}. In the following, we shall use the abbreviation  $V=|A|^2+\frac12. $ We will need the following estimate of Ecker (see also \cite{BW} Lemma B.1). 
		\begin{equation}\label{EqEcker}\int_\Sigma|f|^2|x|^2\mathrm{e}^{-\frac{|x|^2}{4}}d\mu\leq 4\int_\Sigma (n f^2+4|\nabla f|^2)\mathrm{e}^{-\frac{|x|^2}{4}}d\mu \end{equation}
		to suppress the slow linear growth term $\dt|x||\nabla v^*|$ in $\mathcal Q$. 
		
		Then we compute
		\begin{equation}\label{Eq1stDifference}
			\begin{aligned}
				&\partial_t \frac12\int_\Sigma |w(t)|^2 \mathrm{e}^{-\frac{|x|^2}{4}}d\mu = \int_\Sigma w\partial_t w\mathrm{e}^{-\frac{|x|^2}{4}}d\mu
				=\int_\Sigma w(  L_\Sigma w+\chi \mathcal Q+\mathcal B) \mathrm{e}^{-\frac{|x|^2}{4}}d\mu\\
				&=\int_\Sigma \left(-|\nabla w|^2 +Vw^2+ w\chi\mathcal Q(v^*)\right) \mathrm{e}^{-\frac{|x|^2}{4}}d\mu+O(\|v^*,v\|^2_{C^1(\mathbb A_{\mathbf r(t)})}\mathrm{e}^{-\mathbf r(t)^2/4})\\
				&\leq C\int_\Sigma |w(t)|^2 \mathrm{e}^{-\frac{|x|^2}{4}}d\mu+\int_\Sigma |\mathcal Q(v^*)|^2\chi \mathrm{e}^{-\frac{|x|^2}{4}}d\mu+O(\|v^*,v\|^2_{C^1(\mathbb A_{\mathbf r(t)})}\mathrm{e}^{-\mathbf r(t)^2/4})\\
				&\leq C\int_\Sigma |w(t)|^2 \mathrm{e}^{-\frac{|x|^2}{4}}d\mu+\dt\int_\Sigma (|\mathrm{Hess}_{v^*}|+|\nabla v^*|+|v^*|)^2\chi \mathrm{e}^{-\frac{|x|^2}{4}}d\mu+O(\|v^*,v\|^2_{C^1(\mathbb A_{\mathbf r(t)})}\mathrm{e}^{-\frac{\mathbf r(t)^2}{4}}),
			\end{aligned}
		\end{equation}
		where we use a constant $C$ to bound $V$ and in the last $\leq$, we use Lemma \ref{LmAppP} and \eqref{EqEcker}. 
		
		We next have 
		\begin{equation}\label{Eq2ndDifference}
			\begin{aligned}
				&\partial_t\frac12\int_\Sigma (|\nabla w|^2-V w^2) \mathrm{e}^{-\frac{|x|^2}{4}}d\mu=\int_\Sigma (\nabla w\cdot\nabla  \partial_tw-Vw \partial_t w) \mathrm{e}^{-\frac{|x|^2}{4}}d\mu\\
				&=\int_\Sigma - (L_\Sigma w)  (  L_\Sigma w+\chi \mathcal Q(v^*)+\mathcal B) \mathrm{e}^{-\frac{|x|^2}{4}}d\mu\\
				&\leq \int_\Sigma (-|L_\Sigma w|^2+ L_\Sigma w\cdot \mathcal Q(v^*))\chi \mathrm{e}^{-\frac{|x|^2}{4}}d\mu+O(\|v^*,v\|^2_{C^2(\mathbb A_{\mathbf r(t)})}\mathrm{e}^{-\mathbf r(t)^2/4})\\
				&\leq \int_\Sigma 2 | \mathcal Q(v^*)|^2\chi  \mathrm{e}^{-\frac{|x|^2}{4}}d\mu+O(\|v^*,v\|^2_{C^2(\mathbb A_{\mathbf r(t)})}\mathrm{e}^{-\mathbf r(t)^2/4})\\
				&\leq \int_\Sigma \dt(|\mathrm{Hess}_{v^*}|+|\nabla v^*|+|v^*|)^2\chi  \mathrm{e}^{-\frac{|x|^2}{4}}d\mu+O(\|v^*,v\|^2_{C^2(\mathbb A_{\mathbf r(t)})}\mathrm{e}^{-\mathbf r(t)^2/4}).
			\end{aligned}
		\end{equation}
		Combining \eqref{Eq1stDifference} and \eqref{Eq2ndDifference}, we get
		$$\partial_t \|w\|^2_{Q}\leq C(A)\|w\|^2_{L^2}+\int_\Sigma \dt(|\mathrm{Hess}_{v^*}|+|\nabla v^*|+|v^*|)^2\chi  \mathrm{e}^{-\frac{|x|^2}{4}}d\mu+O(\|v^*,v\|^2_{C^2(\mathbb A_{\mathbf r(t)})}\mathrm{e}^{-\mathbf r(t)^2/4}).$$
		Integrating over the time interval $[0,1]$, we get (noting $\|w(0)\|_Q=0$)
		$$\|w(1)\|^2_{Q}\leq \int_0^1 \mathrm{e}^{C(A)t}\int_\Sigma \dt(|\mathrm{Hess}_{v^*}|+|\nabla v^*|+|v^*|)^2\chi \mathrm{e}^{-\frac{|x|^2}{4}}d\mu\,dt+O(\|v^*,v\|^2_{C^2(\mathbb A_{\mathbf r(t)})}\mathrm{e}^{-\mathbf r(t)^2/4}).$$
		Let us next estimate $\int_\Sigma (|\mathrm{Hess}_{v^*}|+|\nabla v^*|+|v^*|)^2\chi \mathrm{e}^{-\frac{|x|^2}{4}}d\mu$. First we recall the drifted Bochner formula
		$|\Hess_u|^2=\frac{1}{2}\cL|\nabla u|^2-\Ric_{-|x|^2/4}(\nabla u,\nabla u)-\langle \nabla\cL u,\nabla u\rangle.$
		Notice that the Bakry-Emery Ricci on the conical shrinker $\Sigma$ is always bounded, so we have
		\[\begin{split}
			\int_\Sigma |\Hess_{v^*}|^2\chi  \mathrm{e}^{-\frac{|x|^2}{4}}d\mu
			\leq &
			\int_\Sigma\left(\frac{1}{2}\cL|\nabla v^*|^2+C|\nabla v^*|^2-\langle \nabla\cL v^*,\nabla v^*\rangle\right)\chi  \mathrm{e}^{-\frac{|x|^2}{4}}d\mu
			\\
			=&
			\int_\Sigma (C|\nabla v^*|^2 + (\cL v^*)^2)\chi  d\mu + \int_\Sigma (\frac{1}{2}\nabla|\nabla v^*|^2+\cL v^*\nabla v^*)\cdot\nabla\chi  \mathrm{e}^{-\frac{|x|^2}{4}}d\mu
			\\
			=&
			\int_\Sigma (C|\nabla v^*|^2+ (\cL v^*)^2)\chi \mathrm{e}^{-\frac{|x|^2}{4}} d\mu
			+ O(\|v^*\|^2_{C^2(\mathbb A_{{\mathbf r}(t)})}\mathrm{e}^{-\mathbf r(t)^2/4})
			.
		\end{split} \]
		Next we calculate
		\[\begin{split}
			&\partial_t\int_\Sigma|\nabla v^*|^2\chi \mathrm{e}^{-\frac{|x|^2}{4}}\mathrm{e}^{-\frac{|x|^2}{4}} d\mu
			\leq \int_\Sigma(\langle \nabla \partial_t v^*,\nabla v^*\rangle \chi+|\nabla v^*|^2 \frac{d}{dt}\chi )\mathrm{e}^{-\frac{|x|^2}{4}} d\mu
			\\
			=&
			\int_\Sigma \langle \nabla L v^*,\nabla v^*\rangle \chi \mathrm{e}^{-\frac{|x|^2}{4}}d\mu +O(\|v^*\|^2_{C^2(\mathbb A_{{\mathbf r}(t)})}\mathrm{e}^{-\mathbf r(t)^2/4})
			\\
			=&
			\int_\Sigma	(-|\cL v^*|^2 +C|\nabla v^*|^2+C|v^*|^2)\chi \mathrm{e}^{-\frac{|x|^2}{4}}d\mu+O(\|v^*\|^2_{C^2(\mathbb A_{{\mathbf r}(t)})}\mathrm{e}^{-\mathbf r(t)^2/4})
			.
		\end{split} \]
		
		So by adding these two inequalities we get
		\[\int_\Sigma |\Hess_{v^*}|^2 \chi  \mathrm{e}^{-\frac{|x|^2}{4}}d\mu+\partial_t\int_\Sigma |\nabla v^*|^2\chi  \mathrm{e}^{-\frac{|x|^2}{4}}d\mu \leq C\int_\Sigma |\nabla v^*|\chi  \mathrm{e}^{-\frac{|x|^2}{4}}d\mu
		+O(\|v^*\|^2_{C^2(\mathbb A_{{\mathbf r}(t)})}\mathrm{e}^{-\mathbf r(t)^2/4}).\]
		Integrating this, we get 
		\[\begin{split}
			\int_{T}^{T+1}\int_\Sigma |\Hess_{v^*}|^2 \chi \mathrm{e}^{-\frac{|x|^2}{4}} d\mu dt+ \mathrm{e}^{-C(T+1)}\int_\Sigma |\nabla v^*(\cdot,T+1)|^2\chi  \mathrm{e}^{-\frac{|x|^2}{4}}d\mu 
			\\
			\leq  \mathrm{e}^{-CT}\int_\Sigma |\nabla v^*(\cdot,T)|^2\chi  \mathrm{e}^{-\frac{|x|^2}{4}}d\mu 
			+C\int_{T}^{T+1}\int_\Sigma |v^*|^2 \chi \mathrm{e}^{-\frac{|x|^2}{4}} d\mu dt
			+O(\|v^*\|^2_{C^2(\mathbb A_{{\mathbf r}(t)})}\mathrm{e}^{-\mathbf r(t)^2/4}).
		\end{split}
		\]
		
		We then absorb $O(\|v^*,v\|^2_{C^2(\mathbb A_{\mathbf r(t)})}\mathrm{e}^{-\mathbf r(t)^2/4})$ by $\|v^*\|^2_{Q}$ using the definition of $T_\sharp$ (Proposition \ref{LmT}) as well as Proposition \ref{prop:v bound} on $v$.  This completes the proof. 
		
	\end{proof}

	\section{The heat kernel of $L$ on conical shrinkers}\label{AppKernel}
	In this section, we sketch the existence of the heat kernel of the linearized operator $L_\Sigma$ on a conical shrinker $\Sigma$ and give the necessary gradient estimate for the solution to the heat equation $\partial_t v=L_\Sigma v$. 
	
	\subsection{Existence of the heat kernel}
	We prove the following result on the existence of the heat kernel in this subsection. 
	\begin{proposition}\label{prop:HKnoncpt}
		There exists heat kernel $\cH(x,y,t):=\sum_{i}\mathrm{e}^{-\mu_i t}u_i(x)u_j(x)$ defined on $\Sigma\times\Sigma\times(0,\infty)$, satisfying 
		\begin{enumerate}
			\item $\partial_t \cH(x,y,t)=L_x\cH(x,y,t)$,
			\item $\cH(x,y,t)=\cH(y,x,t)$,
			\item the reproducing property \eqref{EqReproducing},
			\item the semi-group property \eqref{EqCocycle}.
		\end{enumerate}
	\end{proposition}
	
	In Section \ref{SSHeat}, we have discussed the properties of the heat kernel on a closed RMCF, where the existence comes from the classical theory of the heat kernel of adjoint elliptic operators. For noncompact hypersurfaces, there is no classical theory, and we need to show the existence of the heat kernel.
	
	The proof is very similar to Colding-Minicozzi's proof of the existence of the drifted heat kernel on a noncompact shrinker (see \cite[Section 5]{CM5}). We will only sketch the essential modifications here, and the rest of the steps are verbatim from \cite{CM5}.
	
	The key step is to estimate the eigenfunctions of the linearized operator $L$. Suppose $\phi_i$ is the $i$-th eigenfunction on the conical $\Sigma$ with $\int \phi_i^2 \mathrm{e}^{-\frac{|x|^2}{4}}d\mu=1$. In \cite{BW}, Bernstein-Wang proved that $\phi_1$ decays at infinity. Here we show similar results for higher eigenfunctions.
	
	Suppose $\phi$ satisfies $L\phi=\mu\phi$ on $\Sigma$. Note that by the elliptic spectral theory, there are only finitely many $\mu\geq 0$. Then we have
	$\cL|\phi|+(|A|^2+1/2)|\phi|\geq \mu|\phi|.$
	Since the conical shrinker has uniformly bounded curvature, we conclude that
	$\cL |\phi| \geq (\mu-C)|\phi|=\mu'|\phi|.$
	Then we write $\Sigma_t=\sqrt{-t}\Sigma$ to be the MCF associated with the shrinker $\Sigma$, and let 
	\begin{equation}\label{eq:Eigentosubsol}
		v(y,t)=(-t)^{-\mu'}|\phi|\left(\frac{y}{\sqrt{-t}}\right).
	\end{equation}
	Then similar to the computations in \cite[Lemma 2.4]{CM5}, we have $(\partial_t-\Delta_{\Sigma_t})v\leq 0$. Moreover, since in our case $v\geq 0$, \cite[Lemma 2.11]{CM5} still holds, and hence the proof of \cite[Theorem 2.1]{CM5} is still true. In particular, we can show that
	\begin{equation*}
		|\phi|^2(x)\leq C_n\lambda(\Sigma)\||\phi|\|_{L^2(\Sigma)}^2(4+|x|^2)^{-2\mu'}.
	\end{equation*}
	Here $C_n$ is a dimensional constant and $\lambda(\Sigma)$ is the entropy of $\Sigma$. Thus, we get the following pointwise estimate:
	\begin{lemma}\label{lem:eigenfunctionbound}
		Suppose $\phi$ satisfies $L\phi=\mu\phi$ on $\Sigma$, then $		|\phi|(x)\leq C\|\phi\|_{L^2(\Sigma)}(4+|x|^2)^{-(\mu-C)}.$
	\end{lemma}
	From now on, we will assume $\{\phi_i\}_{i=1}^\infty$ are eigenfunctions of $L$ on $\Sigma$, with $\|\phi_i\|_{L^2(\Sigma)}=1$. With our notation, $\mu_i$ is non-increasing in $i$. Let us define the spectrum counting function $\cN(\mu)$ to be the number of eigenvalues $\mu_i\geq \mu$ counted multiplicity. In order to study $\cN(\mu)$, we introduce the space of functions defined on an ancient MCF $M_t$:
	\begin{equation*}
		\tilde{\cP}_d:=\{u\ |\ (\partial_t-\Delta_{M_t})u\leq 0,\ |u(x,t)|\leq C(1+|x|^d+|t|^{d/2})\text{ for all $x\in M_t, t<0$}\}.
	\end{equation*}
	Compared with the space $\cP_d$ defined in \cite{CM5}, $\tilde{\cP}_d$ consists of subsolutions to the heat equation on the MCF. However, subsolutions still fit in \cite{CM5}. Repeat Colding-Minicozzi's proof, we have the following theorem, which is \cite[Theorem 0.5]{CM5}:
	\begin{theorem}\label{thm:Thm0.5inCM}
		There exists a dimensional constant $C_n$ so that for ancient MCF $M_t$ with $\lambda(M_t)\leq \lambda_0$ and $d\geq 1$, then $\dim\tilde\cP_d\leq C_n\lambda_0d^n$. Here $\lambda(M_t)$ is the entropy of $M_t$.
	\end{theorem}
	
	The proof is exactly the same as the proof in \cite{CM5}. Although now the space consists of subsolutions, one can see that the discussions in \cite{CM5} (like \cite[Lemma 3.4]{CM5}, \cite[Lemma 4.1]{CM5} - actually in the papers cited there, these lemmas were stated for subsolutions) are also valid for subsolutions. 
	
	Once we have Theorem \ref{thm:Thm0.5inCM}, let $M_t=\sqrt{-t}\Sigma$ and we notice that by \eqref{eq:Eigentosubsol} we can transfer an eigenfunction $\phi_i$ to a subsolution $v_i$ of the heat equation on $\sqrt{-t}\Sigma$ which belongs to $\tilde{\mathcal P}_{-2\mu+2C}$. Thus, a consequence of Theorem \ref{thm:Thm0.5inCM} is the following Theorem, which is the analogue of Theorem \cite[Theorem 0.7]{CM5}.
	
	\begin{theorem}
		There exists $C_n$ so that $	\cN(\mu)\leq C_n\lambda(\Sigma)(-(\mu-C))^n 
		$ for $-(\mu-C)\geq 1/2$.
	\end{theorem}
	This theorem immediately shows the following properties of the spectrum of $L$ on $\Sigma$:
	\begin{itemize}
		\item the spectrum of $L$ on $\Sigma$ is discrete,
		\item the eigenvalues $\mu_i\to -\infty$ as $i\to\infty$.
	\end{itemize}
	
	Now let us sketch the proof of Proposition \ref{prop:HKnoncpt}. The proof is the same as that proof of \cite[Theorem 5.3]{CM5}. The only change is to replace $\cL$ with $L$. All the previous Theorems suggest that the proof is almost verbatim.
	
	\subsection{The estimate of $v$}
	Let $v$ be the solution to the equation
	$\begin{cases}
		\partial_t v=L_\Sigma v\\
		v|_{t=T}=v_0
	\end{cases}$.
	Here $v_0\geq 0$ is supported on $B_{\mathbf{r}(M_T)}$. We can solve $v$ by convoluting the initial data with the heat kernel of $L_\Sigma$. Equivalently, we can expand $v$ into Fourier series in the weighted $L^2$ space, namely
	$v=\sum_{i=1}^\infty \mathrm{e}^{\lambda_i t}a_i \phi_i,$
	if
	$v_0=\sum_{i=1}^\infty a_i \phi_i.$ 
	We want to estimate $v_0$ near the boundary of $B_{\mathbf{r}(M_{T+t})}$ where $t\leq 1$. 
	
	\begin{proposition}\label{prop:v bound}
		There exists a constant $C$ only depending on the shrinker $\Sigma$ such that 
		\[\|v\|_{C^2(\Sigma)}\leq C\mathrm{e}^{Ct}\|v_0\|_{C^1(B_{\mathbf{r}(M_T)})}.\]
	\end{proposition}
	
	\begin{proof}
		First we prove 
		$\|v\|_{C^0(\Sigma)}\leq C\mathrm{e}^{t}\|v_0\|_{C^0(B_{\mathbf{r}(M_T)})}.$ We start by proving that $|v|(x)\to 0$ as $x\to\infty$. In fact, we can always choose $\pm A\phi_1$ as the upper barrier and the lower barrier of the initial data $v_0$, where $A$ is a large constant. Then $v(x)\leq \mathrm{e}^{\lambda_1 t}A\phi_1(x)$ (this can be seen from expressing the solutions to the linearized equation by convoluting the initial value with the heat kernel), and since $\phi_1(x)\to 0$ as $x\to\infty$ (see \cite{BW}), $|v|(x)\to 0$ as $x\to \infty$.
		
		Then the maximum of $v$ is attained at some bounded region, so we can use the the maximum principle to show that 
		$\partial_t(\max v)\leq (|A|^2+1/2)\max v.$
		We also have a similar estimate for the minimum. Therefore, $\|v\|_{C^0(\Sigma)}\leq C\mathrm{e}^t\|v_0\|_{C^0(B_{\mathbf{r}(M_T)})}$.
		
		Next we estimate the gradient of $v$. $f:=\frac{1}{2}|\nabla v|^2$ satisfies the equation:
		\begin{equation*}
			\partial_t f=\cL f+2(|A|^2+1/2)f-|\nabla ^2 v|^2 -(\Ric +\nabla x)(\nabla v,\nabla v)+v\langle \nabla v,\nabla |A|^2\rangle.
		\end{equation*}
		Observe that on $\Sigma$, $\Ric+\nabla x$, $|A|^2$ and $\nabla|A|^2$ are all bounded from above. Moreover from previous discussions $\|v\|_{C^0}$ is bounded by $C\mathrm{e}^t\|v_0\|_{C^0}$. Therefore,	$	\partial_t f\leq \cL f+C f+C\mathrm{e}^{t}\|v_0\|_{C^0},$	where $C$ only depends on $\Sigma$. As a consequence, $g:=\mathrm{e}^{-Ct}f-C'\mathrm{e}^{Ct}\|v_0\|^2_{C^0}$ satisfies the equation $	\partial_t g\leq \cL g\leq Lg.$ Then use the same argument as above for the maximum shows that
		$g\leq  C\mathrm{e}^t\|g\|_{C^0(B_{\mathbf{r}(M_T)})}.$
		Therefore $	f \leq C\mathrm{e}^{Ct}(\|f\|_{C^0(B_{\mathbf{r}(M_T)})}+\|v_0\|^2_{C^0(B_{\mathbf{r}(M_T)})}).$ 
		Thus,
		\begin{equation*}
			\|v\|_{C^1(\partial B_{\mathbf{r}(M_{T+t})})}\leq C\mathrm{e}^{Ct}\|v_0\|_{C^1(B_{\mathbf{r}(M_T)})}.
		\end{equation*}
		
		Finally we prove the $C^2$ estimate of $v$. The idea is similar to the $C^1$ estimate. Let $f=\frac{1}{2}|D^2 v|^2$, then
		\[\partial_t f=\langle D^2 v, D^2 \partial_t v\rangle=\langle D^2 v,D^2\Delta v\rangle -\frac{1}{2}\langle D^2 v,D^2 \langle x,\nabla v\rangle \rangle +\langle D^2v, D^2(|A|^2+1/2)v\rangle.\]
		Then similar to the calculation of the previous $C^1$ case, we have
		\begin{equation*}
			\partial_t f\leq \cL f+C f+C |v|^2+C|\nabla v|^2\leq \cL f+Cf +C\|v\|_{C^1}^2.
		\end{equation*}
		Here $C$ depends on the curvature and the derivative of the curvature on the shrinker, and we know they are uniformly bounded. Then just like above, now we choose $g:= \mathrm{e}^{-Ct}f-C'\mathrm{e}^{Ct}\|v_0\|_{C^1}^2$, then $g$ is an upper barrier and the maximum principle shows that 
		$g\leq C\mathrm{e}^{t}\|g\|_{C^0(B_{\mathbf{r}(M_T)})}.$
		So again we obtain $		f \leq C\mathrm{e}^{Ct}(\|f\|_{C^0(B_{\mathbf{r}(M_T)})}+\|v_0\|^2_{C^1(B_{\mathbf{r}(M_T)})}),$ and as a consequence we get 
		$	\|v\|_{C^2(\partial B_{\mathbf{r}(M_{T+t})})}\leq C\mathrm{e}^{Ct}\|v_0\|_{C^2(B_{\mathbf{r}(M_T)})}.$
	\end{proof}
	
	We use the integration by parts formula in the proof of Proposition \ref{LmCM3}. To verify that we can do integration by parts of higher-order derivatives of $v$, we need the following lemma:
	
	\begin{lemma}
		We have $	\int_\Sigma (|v|^2+|\nabla v|^2 +|\cL v|^2)\mathrm{e}^{-\frac{|x|^2}{4}}d\mu <\infty.$
	\end{lemma}
	
	\begin{proof}
		Proposition \ref{prop:v bound} implies that 
		$		\int_\Sigma (|v|^2+|\nabla v|^2)\mathrm{e}^{-\frac{|x|^2}{4}}d\mu <\infty.$
		So it only remains to show
		$	\int_\Sigma |\cL v|^2 \mathrm{e}^{-\frac{|x|^2}{4}}d\mu <\infty,$
		and we only need to prove that 
		\[	\int_\Sigma \|v\|_{C^2(\Sigma)} \mathrm{e}^{-\frac{|x|^2}{4}}d\mu <\infty,\ \text{and}\ \int_\Sigma |\langle x,\nabla v\rangle|^2 \mathrm{e}^{-\frac{|x|^2}{4}}d\mu <\infty.\]
		For the first integral, we use Proposition \ref{prop:v bound} 
		$ \|v\|_{C^2(\Sigma)}\leq C \mathrm{e}^{Ct}\|v_0\|_{C^2(B_{\mathbf{r}(M_T)})},$
		so the first integral is finite; for the second integral, we use the Cauchy-Schwarz inequality to see that
		$|\langle x,\nabla v\rangle|^2\leq |x|^2 |\nabla v|^2,$
		so the second integral is finite.
	\end{proof}
	
	As a consequence, we have the following corollary.
	
	\begin{corollary}\label{cor:integration by parts}
		For any $u$ satisfying
		$		\int_\Sigma (|u|^2+|\nabla u|^2)\mathrm{e}^{-\frac{|x|^2}{4}}d\mu <\infty,$
		we have 
		\begin{equation*}
			\begin{aligned}
				&\int_{\Sigma} \nabla u\cdot \nabla v \mathrm{e}^{-\frac{|x|^2}{4}}d\mu
				=
				-\int_{\Sigma} u \cL v \mathrm{e}^{-\frac{|x|^2}{4}}d\mu
				=
				-\int_{\Sigma} v \cL u   \mathrm{e}^{-\frac{|x|^2}{4}}d\mu,\quad \mathrm{and}\\ 
				&		\int_{\Sigma} \nabla u\cdot \nabla \cL v \mathrm{e}^{-\frac{|x|^2}{4}}d\mu
				=
				-\int_{\Sigma} \cL u \cdot \cL v \mathrm{e}^{-\frac{|x|^2}{4}}d\mu.
			\end{aligned}
		\end{equation*}
	\end{corollary}
	
	\begin{proof}
		This is a consequence of the previous lemma together with Corollary 3.10 of \cite{CM1}.
	\end{proof}


\begin{thebibliography}{99}
		\bibitem{A} Jeffrey H. Albert, {Generic properties of eigenfunctions of elliptic partial differential operators.} Trans. Amer. Math. Soc. 238 (1978), 341-354.
		
		\bibitem{AIC} Angenent, S., Ilmanen, T., Chopp, D. L. {A computed example of nonuniqueness of mean curvature flow in $R^3$.} 
		Comm. Partial Differential Equations 20 (1995), no. 11-12, 1937-1958.
		
		\bibitem{BK} Bamler, Richard H., and Bruce Kleiner. {On the multiplicity one conjecture for mean curvature flows of surfaces.} arXiv preprint arXiv:2312.02106 (2023).
		\bibitem{BW} Bernstein, Jacob, and Wang, Lu. {A topological property of asymptotically conical self-shrinkers of small entropy.} Duke Math. J. 166 (2017), no. 3, 403-435. 
		\bibitem{BW2} Bernstein, Jacob, and Wang, Lu. {Closed hypersurfaces of low entropy in $\mathbb{R}^4$ are isotopically trivial.} Duke Math. J. 171 (2022), no. 7, 1531-1558.
		\bibitem{CC} Chow, Chu, et al, {The Ricci flow, techniques and applications, III, IV.} American Mathematical Society, 2007.
		\bibitem{CCMS1} Chodosh, Otis, and Choi, Kyeongsu, and Mantoulidis, Christos, and Schulze, Felix. {Mean curvature flow with generic initial data.} 
		Invent. Math. 237 (2024), no. 1, 121-220.
		\bibitem{CCMS2} Chodosh, Otis and Choi, Kyeongsu, and Mantoulidis, Christos, and Schulze, Felix. {Mean curvature flow with generic low-entropy initial data.} Duke Math. J. 173 (2024), no. 7, 1269-1290.
		
		\bibitem{CS} Chodosh, Otis, and Schulze, Felix. {Uniqueness of asymptotically conical tangent flows.} Duke Math. J. 170 (2021), no. 16, 3601-3657.
		
		\bibitem{CM1} Colding, Tobias H., and William P. Minicozzi. {Generic mean curvature flow I; generic singularities.} Annals of Mathematics (2012): 755-833.
		\bibitem{CM2} Colding, Tobias Holck, and William P. Minicozzi. {Uniqueness of blowups and Lojasiewicz inequalities.} Annals of Mathematics (2015): 221-285.
		\bibitem{CM3} Colding, Tobias Holck, William P. Minicozzi. {Dynamics of closed singularities.} Annales de l'Institut Fourier, Tome 69 (2019) no. 7, pp. 2973-3016.
		\bibitem{CM4} Holck Colding, Tobias, and William P. Minicozzi. {Wandering Singularities.} J. Differential Geom. 119 (2021), no. 3, 403-420.
		\bibitem{CM5} Colding, Tobias Holck, William P. Minicozzi. {Complexity of parabolic systems.} Publ. Math. Inst. Hautes \'Etudes Sci. 132 (2020), 83-135.
		\bibitem{CMP} Colding, Tobias Holck; Minicozzi, William P., II; Pedersen, Erik Kj\ae r. {Mean curvature flow.} Bull. Amer. Math. Soc. (N.S.) 52 (2015), no. 2, 297-333.
		
		\bibitem{EH} K. Ecker and G. Huisken, {Interior estimates for hypersurfaces moving by mean curvature}, Invent. Math. 105 (1991), no. 3, 547-569.
		\bibitem{G} Getoor, R. K. {Markov operators and their associated semi-groups.} Pacific J. Math. 9 (1959), 449-472.
		
		\bibitem{INS} Ilmanen, Tom, Andr\'e Neves, and Felix Schulze. {On short time existence for the planar network flow.} J. Differential Geom. 111 (2019), no. 1, 39-89.
		
		\bibitem{H} Huisken, Gerhard. {Asymptotic behavior for singularities of the mean curvature flow.} J. Differential Geom. 31 (1990), no. 1, 285-299.
		\bibitem{HM} Julian Haddad, Marcos Montenegro, {The $C^r$ dependence problem of eigenvalues of the Laplace operator on domains in the plane}, 
		J. Differential Equations 264 (2018), no. 5, 3073-3085.
		
		\bibitem{I} Ilmanen, Tom. {Singularities of mean curvature flow of surfaces.} preprint (1995).
		
		
		\bibitem{KKM} Kapouleas, Nikolaos; Kleene, Stephen James; M$\phi$ller, Niels Martin. {Mean curvature self-shrinkers of high genus: noncompact examples.} J. Reine Angew. Math. 739 (2018), 1-39.
		\bibitem{LeZh} Lee, Tang-Kai, and Xinrui Zhao. {Closed mean curvature flows with asymptotically conical singularities.} arXiv preprint arXiv:2405.15577 (2024).
		
		
		\bibitem{LZ} Li, Martin and Zhou, Xin. {Min-max theory for free boundary minimal hypersurfaces I - regularity theory.} 
		J. Differential Geom. 118 (2021), no. 3, 487-553.
		
		\bibitem{Ngu} Nguyen, Xuan Hien. {Construction of complete embedded self-similar surfaces under mean curvature flow, Part III.} Duke Math. J. 163 (2014)
		
		\bibitem{RS} Reed, M., and B. Simon. {Methods of modern mathematical physics. II. Fourier analysis, self-adjointness.} Academic Press [Harcourt Brace Jovanovich, Publishers], New York-London
		\bibitem{Sc} Schulze, Felix {Uniqueness of compact tangent flows in mean curvature flow. }J. Reine Angew. Math. 690 (2014), 163-172.
		\bibitem{Su} Ao Sun, {Local Entropy and Generic Multiplicity One Singularities of Mean Curvature Flow of Surfaces.} 
		J. Differential Geom. 124 (2023), no. 1, 169-198.
		\bibitem{SX} Ao Sun, Jinxin Xue. {Initial Perturbation of the Mean Curvature Flow for closed limit shrinker. arXiv preprint arXiv:2104.03101. (2021)} 
		\bibitem{U} K. Uhlenbeck. {Generic properties of eigenfunctions.} 
		Amer. J. Math. 98 (1976), no. 4, 1059-1078.
		
		\bibitem{V} Vuillermot, Pierre-A. {A generalization of Chernoff's product formula for time-dependent operators.} 
		J. Funct. Anal. 259 (2010), no. 11, 2923-2938.
		
		\bibitem{VWZ} Vuillermot, Pierre-A., Walter F. Wreszinski, and Valentin A. Zagrebnov. {A general Trotter-Kato formula for a class of evolution operators.} 
		J. Funct. Anal. 257 (2009), no. 7, 2246-2290.
		
		\bibitem{Wa1} Wang, Lu. {Uniqueness of self-similar shrinkers with asymptotically conical ends.}
		J. Amer. Math. Soc. 27 (2014), no. 3, 613-638.
		
		\bibitem{Wa2} Wang, Lu. {Asymptotic structure of self-shrinkers.} arXiv preprint arXiv:1610.04904 (2016).
		\bibitem{Wh} White, Brian. {Partial regularity of mean-convex hypersurfaces flowing by mean curvature.} Internat. Math. Res. Notices 1994, no. 4, 186 ff., approx. 8 pp.
		\bibitem{Z} Zhang Liqun. {On the generic eigenvalue flow of a family of metrics and its application} Comm. Anal. Geom. 7 (1999), no. 2, 259-278.
		
	\end{thebibliography}
\end{document}